\documentclass[a4paper]{article}

\usepackage[T1]{fontenc}
\usepackage{amsfonts, amssymb, amsmath, amsthm}
\usepackage{mathtools}
\usepackage{graphicx}
\usepackage{enumerate}
\usepackage{latexsym,mathrsfs}
\usepackage[english]{babel}
\usepackage[usenames,dvipsnames]{color}
\usepackage{pgf,tikz}
\usetikzlibrary{arrows}
\usepackage{hyperref}
\hypersetup{bookmarksopen=true}
\usepackage{subcaption}

\usepackage{tocloft}
\newtheorem{claim}{Claim}
\newtheorem{coro}{Corollary}[section]
\newtheorem{thm}[coro]{Theorem}
\newtheorem{lem}[coro]{Lemma}
\newtheorem{prop}[coro]{Proposition}

\theoremstyle{definition}
\newtheorem{rem}[coro]{Remark}
\newtheorem{exa}[coro]{Example}
\newtheorem{defn}[coro]{Definition}
\newtheorem{obs}[coro]{Observation}

\newcommand{\Aut}[1]{{\operatorname{Aut}(}#1)}
\newcommand{\Rep}[1]{{\operatorname{Rep}(}#1)}

\newcommand{\id}[0]{{\operatorname{id}}}
\newcommand{\Ima}[1]{{\operatorname{Im}(}#1)}
\newcommand{\dist}[2]{{\operatorname{dist}(}#1, #2)}

\newcommand{\Sym}[1]{{\operatorname{Sym}(}#1)}

\newcommand{\Sp}[2]{{\mathsf {S}(}#1, #2)}
\newcommand{\B}[2]{{\mathsf {B}(}#1, #2)}

\newcommand{\N}{\mathbb{N}}

\newcommand{\T}{{\mathcal{T}}}
\newcommand{\Ra}{{\mathcal{R}}}
\newcommand{\Pa}{{\mathcal{P}}}

\newcommand{\inv}[1]{#1^{-1}}
\newcommand{\ch}[1]{{\operatorname{Ch}(}#1)}
\newcommand{\dtw}[0]{{\operatorname{dist}_{TW}}}

\newcommand{\dw}[0]{{\mathsf{d}_{W}}}
\newcommand{\gwr}{\text{--}\mathrm{WR}}
\newcommand{\dash}{\nobreakdash-\hspace{0pt}}
\newcommand{\adj}[1]{\overset{#1}{\sim}} 

\DeclareMathOperator{\Stab}{Stab}
\DeclareMathOperator{\Fix}{Fix}

\DeclareMathOperator{\proj}{proj}

\numberwithin{equation}{section}

\title{Universal groups for right-angled buildings}
\author{Tom De Medts, Ana C. Silva and Koen Struyve%
\footnote{Authors' address: Ghent University, Department of Mathematics, Krijgslaan 281--S22, 9000 Gent, Belgium.
E-mail addresses: {\tt Tom.DeMedts@UGent.be}, {\tt ana.fcd.silva@gmail.com}, {\tt kstruy@gmail.com}. Corresponding author: Tom De Medts.}}

\date{\today}

\begin{document}
\maketitle

\begin{abstract}
    In 2000, M.~Burger and S.~Mozes introduced universal groups acting on trees with a prescribed local action.
    We generalize this concept to groups acting on right-angled buildings.
    When the right-angled building is thick and irreducible of rank at least $2$ and each of the local permutation groups
    is transitive and generated by its point stabilizers, we show that the corresponding universal group is a simple group.

    When the building is locally finite, these universal groups are compactly generated totally disconnected locally compact groups,
    and we describe the structure of the maximal compact open subgroups of the universal groups as a limit of generalized wreath products.
\end{abstract}

\tableofcontents

\section*{Introduction}
\addcontentsline{toc}{section}{Introduction}%

In 2000, Marc Burger and Shahar Mozes introduced {\em universal groups} acting on regular locally finite trees with a prescribed local action \cite{BurgerMozes}.
More precisely, if $F \leq \Sym{d}$ is a finite transitive permutation group of degree $d$, and $T$ is the regular tree of degree $d$,
then the universal group $U(F)$ is the largest possible vertex-transitive subgroup of $\Aut{T}$ such that the induced local action around any vertex
is permutationally isomorphic to $F$.

In this paper, we will generalize this notion to the setting of {\em right-angled buildings}.
A right-angled building is a building for which the underlying Coxeter diagram contains only $2$'s and $\infty$'s, i.e., the corresponding Coxeter group $W$
with generating set $S$ of involutions only contains relations of the form $s_i s_j = s_j s_i$ for some pairs $i \neq j$, but no other relations.
This makes the combinatorics of the building behave very much like trees, while at the same time providing a much larger class of examples.
One similarity is the fact that, given a right-angled Coxeter group $(W,S)$, there is, for each collection of cardinalities $(q_s)_{s \in S}$,
a unique right-angled building of type $(W,S)$ such that each $s$-panel contains exactly $q_s$ chambers,
as was shown by Fr\'ed\'eric Haglund and Fr\'ed\'eric Paulin \cite{HP2003}.
(Such a result is not true at all for buildings in general.)
A more systematic study of combinatorial properties of right-angled buildings was initiated by Anne Thomas and Kevin Wortman in \cite{TW11}
and continued by Pierre-Emmanuel Caprace in \cite{Caprace},
and we rely on their results.

\medskip

Generalizing the concept of universal groups from trees to right-angled buildings involves a careful analysis of {\em colorings} in right-angled buildings.
We found it useful to introduce the concept of a {\em directed right-angled building}, which gives a systematic way to describe the unique right-angled
building with given valencies, unfolding from one particular chamber in the building.

In \cite{TW11}, Thomas and Wortman introduced the notion of a {\em tree-wall} for general right-angled buildings,
following an idea of Marc Bourdon \cite{Bourdon}, who introduced a similar notion for the right-angled Fuchsian case.
An important concept in our study of the right-angled buildings is the {\em tree-wall tree} (associated to a generator $s \in S$);
see Definition~\ref{deftreewall} below.
This will allow us to define a distance between tree-walls of the same type, which will be useful for inductive arguments.

We will also make extensive use of the notion of {\em wings}, introduced in \cite{Caprace}, following an idea from \cite[Proposition 6]{TW11}.
In particular, this will allow us to express a property for groups acting on right-angled buildings
which is similar to Tits' {\em independence property (P)} for groups acting on trees \cite{Tits}; see Proposition~\ref{fixpanel}.

\medskip

The full automorphism group of a thick {\em locally finite} semi-regular right-angled building $\Delta$ is a totally disconnected locally compact (t.d.l.c.) group
with respect to the permutation topology.
The universal group for a thick locally finite semi-regular right-angled building (with respect to a set of finite transitive permutation groups)
is always a chamber-transitive closed subgroup of $\Aut\Delta$,
and so it is itself a t.d.l.c.\@ group which is compactly generated.

Caprace has shown that if $\Delta$ is thick, semi-regular and irreducible (but not necessarily locally finite),
then the subgroup of type-preserving automorphisms is, in fact, an (abstractly) simple group \cite[Theorem 1.1]{Caprace}.
We will show that, when each of the local permutation groups is
generated by its point stabilizers, the universal group is again a simple group.
To prove this result, we will
use the independence property mentioned above.
Our approach is somewhat different from Caprace's approach, and in particular, our methods also give a simpler proof of his result.

Compactly generated (topologically) simple t.d.l.c.\@ groups are of interest since they provide the building blocks for a large class of t.d.l.c.\@ groups;
see \cite[Theorem C]{CM11}.
We investigate the structure of the maximal compact open subgroups of our universal groups.
We will provide explicit descriptions of the stabilizers of spheres of radius $n$ around a fixed chamber,
and also of the stabilizers of $w$-spheres, i.e., sets of elements at a fixed Weyl distance $w$ from a given chamber.
In both cases, we obtain a description either as an iterated semidirect product,
or at once as a generalized wreath product.

\subsubsection*{Organization of the paper}
\addcontentsline{toc}{subsection}{Organization of the paper}%

In Section~\ref{se:prelim}, we introduce the necessary notions from the theory of buildings,
and we recall some facts about subdirect products and wreath products in group theory.

Section~\ref{se:rab} is devoted to right-angled buildings.
We first study the combinatorics of right-angled Coxeter groups;
this gives rise, in particular, to a useful pair of ``Closing Squares Lemmas'' (Lemmas~\ref{Squares} and~\ref{Squares2}) in right-angled buildings.
We then have a closer look at parallel residues (a concept introduced by Tits \cite[p.~277]{Tits-Durham} and studied in the context of right-angled buildings in \cite{Caprace}),
after which we focus on tree-walls and wings, and introduce the tree-wall trees mentioned above.
Section~\ref{ss:colorings} then deals with different types of colorings in right-angled buildings (legal colorings, weak legal colorings and directed legal colorings),
and in Section~\ref{ss:directed} we introduce directed right-angled buildings, which will be useful later in our study of the maximal compact open subgroups.

We are then fully prepared to introduce the universal groups for right-angled buildings in Section~\ref{se:univ}.
After showing a handful of basic properties of these groups, we prove that we can extend the local action on a residue of the building to an element of the universal group
(Proposition~\ref{extuni}).
Then we prove that any non-trivial normal subgroup of $U$ contains subgroups with support on a single wing with respect to a tree-wall (Proposition~\ref{normfix});
this will turn out to be one of the key points to prove the simplicity of the universal group (Theorem~\ref{simplicity}).

In the final Section~\ref{se:cpt-open}, we assume the building to be locally finite,
and we describe the structure of the maximal compact open subgroups of $U$.
Our main results in that section are Theorem~\ref{thelimit} and Proposition~\ref{swstructure}.
It turns out to be very useful to use the notion of {\em generalized wreath products} with respect to a suitably chosen partial order.

\subsubsection*{Acknowledgment}
\addcontentsline{toc}{subsection}{Acknowledgment}%

The idea of defining universal groups for right-angled buildings is due to Pierre-Emmanuel Caprace.
We are grateful to him for sharing his idea with us.

This research was supported by the project ``Automorphism groups of locally finite trees'' (G011012), Research Foundation, Flanders, Belgium (F.W.O.-Vlaanderen).

\section{Preliminaries}\label{se:prelim}
\subsection{Buildings}

We will follow the combinatorial approach to buildings, described, for instance, in \cite{Weiss},
from which we will silently adopt the definitions, notations and conventions.
In particular, if $(W, S)$ is a Coxeter system, then a building $\Delta$ of type $(W,S)$ is an edge-colored graph, the vertices of which are the chambers of the building,
where two vertices are connected by an edge with label $s \in S$ if they are $s$-adjacent, i.e., if the two vertices lie in the same $s$-panel
(so the $s$-panels of $\Delta$ are complete subgraphs).
If $c$ and $d$ are $s$-adjacent chambers, we write $c \adj{s} d$.

The group of type-preserving automorphisms of $\Delta$ will be denoted by $\Aut\Delta$.
We can equip $\Aut \Delta$ with the {\em permutation topology}, defined by choosing as a neighborhood basis of the identity, the family of pointwise stabilizers of finite subsets of $\ch\Delta$, i.e., a neighborhood basis of the identity is given by
\[ \{\Aut\Delta _{(F)} \mid F \text{ is a finite subset of } \ch\Delta \}. \]
It follows that a sequence $(g_i)_{i\in \N}$ of elements of $\Aut \Delta$ has a limit $g\in \Aut \Delta$ if and only if for every chamber $c\in \ch\Delta$ there is a number $N$ (possibly depending on $c$) such that $g_n.c=g.c$ for every $n\geq N$.
One could also use the property above describing convergence of sequences as a definition of the topology and then we think of the permutation topology as the topology of pointwise convergence.
If we think of $\ch\Gamma$ as having the discrete topology and the elements of $\Aut \Delta$ as maps $\ch\Delta \to \ch\Delta$, then the permutation topology is the same as the compact-open topology.

With the definition of the permutation topology, we can characterize the open subgroups in $\Aut \Delta$.
A subgroup $G\leq \Aut \Delta$ is \emph{open} if and only if there is a finite subset $F$ of $\ch\Delta$ such that $\Aut\Delta_{(F)}\subseteq G$.

In the case where $\Delta$ is locally finite, the group  $\Aut \Delta$ is Hausdorff, locally compact and totally disconnected; see, for example, \cite[Lemma~1]{Woess} and {\cite[Corollary~3.1.12]{topologicalgroupsandrelatedstrcutures}}.

\bigskip

We will use the following notation throughout the paper.
Let $\Sigma$ be a Coxeter diagram with vertex set $S$ and let $(W,S)$ be a Coxeter system of type $\Sigma$ with set of generators $S$.
For each pair $s,t \in S$, we let $m_{st}$ be the corresponding entry of the Coxeter matrix, i.e., $m_{st}$ is the order of the element $st \in W$.

\begin{defn}\label{3CBrown}
A \emph{$\Sigma$-elementary operation} is an operation of one of the following two types:
\begin{enumerate}[(1)]
	\item
	 Delete a subword of the form $ss$.
	\item
	 Given $s,t\in S$ with $s\neq t$ and $m_{st}<\infty$, replace an alternating subword $st\cdots$ of length $m_{st}$ by the alternating word $ts\cdots$ of length $m_{st}$.
\end{enumerate}
We call a word \emph{$\Sigma$-reduced}\index{Word ! Reduced} if it cannot be shortened by any finite sequence of the operations (1) and (2) above.
\end{defn}

For Coxeter groups there is a solution for the word problem.

\begin{lem}[{\cite[Section 3C]{Brown}}]\label{reduced words}
Let $(W,S)$ be a Coxeter system with Coxeter diagram $\Sigma$.
\begin{enumerate}
	\item
	A word is reduced if and only if it is $\Sigma$-reduced.
	\item
	 If $w_1$ and $w_2$ are reduced words, then they represent the same element if and only if $w_1$ can be transformed into $w_2$ by applying a sequence of elementary operations of type (2).
\end{enumerate}
\end{lem}

Let $\Delta$ be a building of type $\Sigma$ and let $\delta \colon \ch\Delta\times\ch\Delta \to W$ denote the Weyl distance.
We consider a gallery distance between chambers in a building as in \cite{TW11}.

\begin{defn}\label{galdist}
The \emph{gallery distance}  is defined in $\ch\Delta$ as
\[ \dw(c, c')=l (\delta(c,c')), \]
i.e., the length of a minimal gallery between the chambers $c$ and $c'$.

For a fixed chamber $c_0\in\ch\Delta$ we define  the spheres at a fixed gallery distance from  $c_0$,
\[ \Sp{c_0}{n} =\{ c\in \ch\Delta \mid \dw(c_0, c) = n\}, \]
and the balls
\[ \B{c_0}{n} =\{ c\in \ch\Delta \mid \dw(c_0, c) \leq n\}. \]
\end{defn}

\begin{defn}\label{induced action}
\begin{enumerate}
    \item
        Fix a chamber $c_0$ of $\Delta$ and let $\Ra$ be a residue in $\Delta$ containing  $c_0$.
        Let $G\leq\Aut\Delta$ be a group of automorphisms of $\Delta$. We denote by $G_{c_0}$ the stabilizer of the chamber $c_0$ in the group $G$.
        Moreover, $\Fix_{G}(\Ra)$ will denote the fixator of $\Ra$ in $G$, i.e., the set of elements $g\in G$ such that $g.c=c$ for all $c\in \ch\Ra$.
        In particular $\Fix_{G}(\Ra)$ is subgroup of~$G_{c_0}$.
    \item
        Let $G\leq\Aut\Delta$ be a group of automorphisms of $\Delta$, and let $B$ be any subset of $\ch\Delta$.
        Then we write $G|_B$ for the permutation group induced on $B$, i.e.\@ $G|_B := \Stab_{G}(B)/\Fix_{G}(B)$.
\end{enumerate}
\end{defn}

\begin{defn}\label{convex}
A set of chambers $\mathcal C$ of $\Delta$ is called \emph{combinatorially convex} (or just {\em convex}) if for every pair $c,c' \in \mathcal C$, every minimal gallery from $c$ to $c'$ is entirely contained in $\mathcal C$.
\end{defn}

\subsection{Subdirect products}

In Section~\ref{se:cpt-open}, we will encounter groups that are subdirect products of generalized wreath products.
We briefly recall the notion of a subdirect product in this section, and explain generalized wreath products in the next section.

\begin{defn}\label{subdirectproduct}
Let $G_1, \dots, G_n$ be groups. A {\em subdirect product} of the groups $G_1, \dots , G_n$ is a subgroup $P\leq G_1\times \cdots \times G_n$ of the direct product such that for each $i \in \{ 1,\dots,n \}$, the canonical projection $P\to G_i$ is surjective.
\end{defn}

\begin{exa} \label{Intransitive groups}
Let $G$ be a group acting faithfully on a set $S$, and let
\[ S= S_1 \sqcup S_2 \sqcup \dotsm \sqcup S_n \]
be a decomposition of $S$ into $G$-invariant subsets, i.e., each set $S_i$ is a union of orbits for the action of $G$ on $S$.
Then we get corresponding homomorphisms
\[ \alpha_i \colon G \to \Sym{S_i} . \]
Let $G_i = \Ima{\alpha_i}$, and define a new homomorphism
\[ \varepsilon \colon G \to G_1 \times \dots \times G_n \colon g \mapsto \bigl( \alpha_1(g), \dots, \alpha_n(g) \bigr). \]
Then $\varepsilon$ is injective, hence $G \cong \varepsilon(G) \leq G_1 \times \dots \times G_n$,
and by definition, $\varepsilon(G)$ surjects onto each $G_i$.
Hence we have realized $G$ as a subdirect product of $G_1\times \dots \times G_n$.
\end{exa}

\subsection{Generalized wreath products}\label{sec:genwreath}

The compact open subgroups of the universal groups for locally finite trees have the structure of an iterated wreath product
\cite[p.\@ 143]{BurgerMozes}.
In the situation of locally finite right-angled buildings, we will describe the structure of the compact open subgroups in Section~\ref{se:cpt-open} below.
In order to describe these subgroups, we will require a more general description of groups arising as intersections of distinct wreath products.

The approach we take to these groups, called {\em generalized wreath products}, is the one described by Gerhard Behrendt in \cite{Behrendt}.
His most general definition of a generalized wreath product requires the notion of a so-called {\em systematic subset} of a set $X$ (see \cite[Section 2]{Behrendt}),
but we will not need this concept because we will only be dealing with finite posets (see the comment on top of p.\@~260 of {\em loc.\@~cit.}).

\begin{defn}\label{defn:gwp}
\begin{enumerate}
    \item
        An \emph{equivalence system} $(X, E)$ is a pair consisting of a set $X$ and a set $E$ of equivalence relations on $X$.
        The automorphism group of $(X, E)$ is the set of all permutations of the set $X$ which leave the relations in $E$ invariant, i.e.,
        \begin{multline*}
            \Aut{X,E} = \{ g \in \Sym X \mid x\sim y \iff g.x \sim g.y \\
                \text{ for all } x, y \in X \text{ and for all } {\sim} \in E \} .
        \end{multline*}
    \item
        Let $(S, \prec)$ be a poset (i.e., a partially ordered set%
        \footnote{All partial orders occurring in this paper are {\em strict} partial orders, i.e., binary relations that are irreflexive, transitive and antisymmetric.
        When $\prec$ is a (strict) partial order, we write $\preceq$ for the corresponding non-strict partial order.}%
        ).
        For each $s\in S$, let $G^s$ be a permutation group acting on a set $X_s$.
        Let $X=\prod_{s\in S} X_s$ be the direct product of the sets $\{X_s\}_{s\in S}$.

        For each $s\in S$, we define two equivalence relations on $X$ as follows.
        For each pair of tuples $x, y \in X$ we define
        \begin{align*}
            x \adj{s} y & \iff x_t=y_t \text{ for all } t \succ s. \\
            x \simeq_s y & \iff x_t=y_t \text{ for all } t \succeq s.
        \end{align*}
        We let $E=\{ \adj{s}  \mid s\in S  \} \cup \{ \simeq_s \mid s\in S  \}  $ be the set of all these equivalence relations.
        We define the \emph{generalized wreath product} $G = X \gwr_{s\in S} G^s$ as
        \begin{multline*}
            G = \{ g \in \Aut{X,E} \mid \text{ for each }x \in X \text{ and } s \in S \text{ there is } g_{s,x} \in G^s \\
                \text{such that } (g. y)_s = g_{s,x}.y_s \text{ for all } y\in X \text{ with } y\adj{s} x \} .
        \end{multline*}
        We note that  $X\gwr_{s\in S} G^s$ is indeed a group, with $(\inv g)_{s,x}= \inv{(g_{s, \inv{g}.x})}$ and  $(gh)_{s,x}=g_{s, h.x}\,h_{s,x}$.
\end{enumerate}
\end{defn}

\begin{rem}\label{rem:gwp}
    If $\prec$ is the empty partial order, then we get a direct product of the groups~$G^s$.
    If the poset $S$ is a chain, then we get the iterated wreath product of the groups $G_s$ with its imprimitive action,
    as in \cite[Section 2.6]{DixonMortimer} (sometimes also called the {\em complete wreath product}).

    In the most general case where we will apply this construction, it will, in fact, be possible to view the generalized wreath product
    as an {\em intersection} of wreath products acting on the same product set; see Remark~\ref{rem:intwr} below.
\end{rem}

In our study of the compact subgroups of the universal group it will be useful to consider certain quotients of generalized wreath products. That is what we will do now.
We retain the notation of the previous paragraphs.

\begin{defn}\label{def:ideal}
A subset $I$ of the poset $(S, \prec)$ is called an \emph{ideal} of $S$ if for every $s\in I$ and $t\in S$, $t\preceq s$ implies $t\in I$.
We define a new equivalence relation on $X$ as
\[x\sim_I y \iff x_s = y_s \text{ for all } s\not\in I.\]
As before, let $G = X\gwr_{s\in S} G^s$. We consider its subset
\[ D(I) = \{ g\in G \mid x\sim_I g.x \text{ for all } x\in X \} . \]
\end{defn}

\begin{prop}[{\cite[Proposition 7.1]{Behrendt}}]\label{DInormal}
Let $I$ be an ideal of $S$. Then $D(I)$ is a normal subgroup of $G$.
\end{prop}

The group $D(I)$ can also be described as a generalized wreath product; see \cite[Theorem 7.2]{Behrendt}.
We will describe such a generalized wreath product in the case that $|I|=1$.
Let $r\in S$ such that $I=\{r\}$ is an ideal of $S$; then $t \not\prec r$ for all $t \in S$.
We will write $D(r)$ rather than $D(\{r\})$.
For any subset $T\subseteq S$, let $p_T$ denote the projection map $p_T \colon X \to \prod_{s\in T} X_s$ defined by $x\mapsto (x_s)_{s\in T}$.

\begin{prop}\label{DIorder}
    Let $(S, \prec)$ be a poset, and let $I=\{r\}$ be an ideal of $S$.
    Let $d_r = \prod_{t \succ r} |X_t|$ (where $d_r = 1$ if there are no $t \succ r$).
    Then $D(r)$ is isomorphic to the direct product of $d_r$ copies of $G^r$.
    In particular, if the sets $X_s$ are finite, then
    \[ |D(r)| = |G^r|^{d_r} . \]
\end{prop}
\begin{proof}
    This follows from \cite[Theorem 7.2]{Behrendt}.
    Notice that the general statement from {\em loc.\@~cit.} requires the definition of an additional partial order, which is empty in our case.
\end{proof}

\begin{lem}\label{le:r-ideal}
    Let $(S, \prec)$ be a poset, let $I=\{r\}$ be an ideal of $S$ and let $S'=S\setminus \{r\}$.
    Let $H^s = G^s$ for all $s \in S'$, let $H^r = 1$ and consider the group
    \[ H = X\gwr_{s\in S} H^s \leq G . \]
    Then $G=D(r)\rtimes H$.
\end{lem}
\begin{proof}
Let $X' = \prod_{s\in S'} X_s$ and $X = X_r \times X'$, and consider the generalized wreath product
\[ G'= X'\gwr_{s\in S'} G^s . \]
Since $\{r\}$ is an ideal of $S$, the group $D(r)$ is a normal subgroup of $G$ by Proposition~\ref{DInormal}.

Consider two elements $x_1=(x_r, x')$ and $x_2=(y_r, x')$ of the set $X = X_r \times X'$.
Since $r \not\succ s$ for all $s\in S'$ (because $\{r\}$ is an ideal of $S$),
we have $x_1 \simeq_s x_2$ for all $s\in S'$.
Therefore, by definition of $G$ (see Definition~\ref{defn:gwp}), we obtain that $(g.x_1)_s=(g.x_2)_s$ for all $s\in S'$ and all $g \in G$.

Consider the projection map $p_{S'} \colon X \to X'$.
Then this projection induces a homomorphism $\rho \colon G\to G'$ defined by $(\rho(g).x')_s=(g.(x_r, x'))_s$, for $s\in S'$ and for any $x_r\in X_r$
(the choice of the element of $X_r$ is irrelevant by the observation in the previous paragraph).
Notice that the kernel of $\rho$ is precisely $D(r)$.

Observe now that for any $h \in G'$, the element $\id \times h$ belongs to $G$.
In particular, $\rho$ is surjective, the map $\sigma \colon G' \to G \colon h \mapsto \id \times h$ is a section for~$\rho$ and $\operatorname{im}(\sigma) = H$.
We conclude that $G=D(r)\rtimes H$.
\end{proof}

\section{Right-angled buildings}\label{se:rab}

We now present the definition of right-angled Coxeter groups and right-angled buildings, and from now on,
we will assume that we are always in the right-angled case, unless otherwise stated.

\begin{defn}\label{2.3}
A Coxeter group $W$ is called \emph{right-angled} if the entries of the Coxeter matrix are $1$, $2$ and $\infty$. In other words, the group $W$ can be described as
\begin{multline*}
    W = \langle\{s_i\}_{i\in I} \mid (s_is_j)^{m_{ij}}\rangle, \\
        \text{ with } m_{ij}\in\{2,\infty\} \text{ for all } i\neq j \text{ and } m_{ii}=1 \text{ for all } i\in I.
\end{multline*}
In this case, we call the Coxeter diagram $\Sigma$ of $W$ a \emph{right-angled} Coxeter diagram.

A \emph{right-angled building} $\Delta$ is a building of type $\Sigma$, where $\Sigma$ is a right-angled Coxeter diagram.
\end{defn}

\begin{thm}[{\cite[Proposition 1.2]{HP2003}}]\label{unicityhp}
Let $(W, S)$ be a right-angled Coxeter group and $(q_s)_{s\in S}$ be a family of natural numbers with $q_s\geq 2$.
There exists a right-angled building of type $(W, S)$ such that for every $s\in S$, each $s$-panel has size $q_s$.
This building is unique, up to isomorphism.
\end{thm}

This was proved for the right-angled Fuchsian case in \cite{Bourdon}. According to Haglund and Paulin \cite{HP2003} this result was proved by M. Globus and was also known by M.~Davis, T. Januszkiewicz and J. \'{S}wiatkowski.

A right-angled building with prescribed thickness in the panels, as described in Theorem~\ref{unicityhp}, is called a \emph{semi-regular} building.

\begin{thm}[Theorem 1.1 in \cite{Caprace}]
Let $\Delta$ be a thick semi-regular building of right-angled type $(W,S)$. Assume that $(W,S)$ is irreducible and non-spherical. Then the group $\Aut \Delta $ of type-preserving automorphisms of $\Delta$ is abstractly simple and acts strongly transitively on $\Delta$.
\end{thm}

\subsection{Reduced words in right-angled Coxeter groups}

The elementary operations from Definition~\ref{3CBrown} become easier to describe in \mbox{right-angled} Coxeter groups since the generators either commute or are not related.
To be concrete, let $(W, S)$ be a right-angled Coxeter system with Coxeter diagram $\Sigma$ and set of generators \mbox{$S=\{s_i \mid i\in I\}$}. We define a $\Sigma$-\emph{elementary operation} as an operation of the following two types:
\begin{enumerate}[(1)]
	\item
	Delete a subword of the form $ss$, with $s\in S$.
	\item
	Replace a subword $st$ by $ts$ if $m_{ts}=2$.
\end{enumerate}
A word in the free monoid $M_S$ is then \emph{reduced}\index{Word ! Reduced} if it cannot be shortened by a sequence of \mbox{$\Sigma$-elementary} operations.
Moreover, by Lemma~\ref{reduced words}, two reduced words represent the same element of $W$ if and only if one can be obtained from the other by a sequence of elementary operations of type (2).

We observe that if $w_1$ is a reduced word with respect to $\Sigma$ and $w_2$ is a word obtained from $w_1$ by applying one $\Sigma$-elementary operation of type (2), then  $w_2$ is also a reduced word and $w_1$ and $w_2$ represent the same element of $W$.
Furthermore, the words $w_1$ and $w_2$ only differ in two consecutive letters that have been switched, let us say,
\[ w_1 = s_1\cdots s_rs_{r+1} \cdots s_\ell \text{ and } w_2=s_1\cdots s_{r+1} s_r \cdots s_\ell, \text{with } |s_rs_{r+1}|=2 \text{ in } \Sigma.  \]
If $\sigma\in \Sym \ell$ then we denote by $\sigma . w_1$ the word obtained by permuting the letters in $w_1$ according to the permutation $\sigma$, i.e.,
\[ \sigma. w_1= s_{\sigma(1)}\cdots s_{\sigma(r)} s_{\sigma(r+1)}\cdots s_{\sigma(\ell)}. \]
Hence, if $w_1$ and $w_2$ are as above and $\sigma=(r \ r+1)\in \Sym \ell$, then $\sigma. w_1 = w_2$.

\begin{defn}
Let $w=s_1\cdots s_\ell$ be a reduced word in $M_S$ with respect to a right-angled Coxeter diagram $\Sigma$. Let $\sigma = ( i \ i+1)$ be a transposition in $\Sym \ell$, with $i \in \{1,\dots, \ell-1\}$. We call $\sigma$ a \emph{$w$-elementary transposition}\index{Word ! Elementary transposition} if $s_i$ and $s_{i+1}$ commute in $W$.
\end{defn}

 In this way, we can associate an elementary transposition to each $\Sigma$-elementary operation of type (2).
By Lemma~\ref{reduced words},  two reduced words $w$ and $w'$ represent the same element of $W$ if and only if
\begin{multline*}
    w'= (\sigma_n \cdots \sigma_1). w, \\ \text{ where each } \sigma_i \text{ is a } (\sigma_{i-1}\cdots \sigma_1).w\text{-elementary transposition},
\end{multline*}
i.e., if $w'$ is obtained from $w$ by a sequence of elementary transpositions.

\begin{defn}\label{def:rep}
If $w$ is a reduced word of length $\ell$ with respect to $\Sigma$, then we define
\begin{multline*}
    \Rep w=\{ \sigma\in \Sym \ell \mid \sigma =\sigma_n\cdots \sigma_1, \text{ where each } \sigma_i \text{ is a } \\
        (\sigma_{i-1}\cdots\sigma_1).w\text{-elementary transposition}\}.
\end{multline*}
\end{defn}
The set $\Rep w$ is formed by the permutations of $\ell$ letters which give rise to reduced representations of $w$, according to the relations in the right-angled Coxeter diagram $\Sigma$.

\begin{defn}\label{def:poset}
Let $w=s_1\cdots s_\ell$ be a reduced word of length $\ell$ in $M_S$ with respect to $\Sigma$. Let $S_w = \{1, \ldots , \ell \}$. We define a new partial order ``$\prec_w$'' on $S_w$ as follows:
\[ i\prec_w j \iff \sigma(i) > \sigma(j) \text{ for all } \sigma\in \Rep w.\]
\end{defn}
Note that $i \prec_w j$ implies that $i > j$.
The fact that the order is reversed will be useful in Section~\ref{sec:compgenwreath} when we will describe certain compact open subgroups as generalized wreath products.

\begin{obs}\label{obsrep}
Let $w= s_1 \dotsm s_i \dotsm s_j \dotsm s_\ell$ be a reduced word in $M_S$ with respect to a right-angled Coxeter diagram $\Sigma$.
\begin{enumerate}
    \item\label{obsrep:1}
        Let $\sigma\in \Rep w$ such that $\sigma$ can be written as a product of elementary transpositions $\sigma_n\cdots \sigma_1$.
        Then for each $k\in \{1,\dots, n\}$, the word $(\sigma_k\cdots \sigma_1). w$ is also a reduced representation of $w$.
    \item
        If $\sigma_1$ is a \mbox{$w$-elementary} transposition switching two generators $s_i$ and $s_j$ and $\sigma_2$ is a \mbox{$\sigma_1.w$-elementary} transposition
        switching two generators $s_{i'}$ and $s_{j'}$ with $\{i,j\}\cap \{i',j'\}=\emptyset$ then $\sigma_2$ is also a $w$-elementary transposition.
    \item\label{obsrep:3}
         If $|s_is_j|=\infty$ in $\Sigma$, then $j\prec_w i$.
    \item\label{obsrep:4}
        If $j\not\prec_w i$  then by~\eqref{obsrep:3}, $|s_is_j|=2$,
        and moreover, for each $k\in \{i+1, \dots, j-1\}$, either $|s_is_k|=2$ or $|s_js_k|=2$ (or both).
    \item\label{obsrep:5}
         Let $s_r$ and $s_{r+1}$ be consecutive letters in $w$. Then $|s_r s_{r+1}|=2$ if and only if $r+1\not\prec_{w} r$.

\end{enumerate}
\end{obs}

The next lemma describes some conditions on the existence and structure of distinct reduced representations of elements in right-angled Coxeter groups.

\begin{lem}\label{reduced}
Let $(W, S)$ be a right-angled Coxeter system with Coxeter diagram~$\Sigma$.
Let \mbox{$w=w_1 s_r \cdots s_t w_2$} be a reduced word in $M_S$ with respect to~$\Sigma$.
If $t\not\prec_w r$ then there exist two reduced representations of the word $w$ of the form
\[ w_1 \cdots s_r s_t \cdots w_2 \quad \text{ and } \quad w_1\cdots s_t s_r \cdots w_2 , \]
i.e.\@
one can exchange the positions of $s_r$ and $s_t$ using only elementary operations on the generators of the set $\{s_r, s_{r+1}, \dots, s_{t-1}, s_t\}$,
without touching the prefix $w_1$ and the suffix $w_2$.
\end{lem}

\begin{proof}
We will prove the result by induction on the number $N$ of letters between $s_r$ and $s_t$.
If $N=0$ then $w= w_1s_r s_t w_2,$ and the result follows from Observation~\ref{obsrep}\eqref{obsrep:4}.

Assume by induction hypothesis that if $N\leq n$ the result holds.
Consider \mbox{$w=w_1s_r \cdots s_t w_2$} with $N=n+1$, i.e., with $n+1$ letters between $s_r$ and $s_t$ in $w$, and let  $\sigma\in\Rep w$ such that $\sigma(t)<\sigma(r)$.

If $|s_rs_{r+1}|=2$ then $(r \ r+1) \in \Rep w$ and
\[ \overline{w}=(r \ r+1).w = w_1s_{r+1}s_r \cdots s_t w_2 \]
is a reduced representation of $w$ with $N=n$ letters between $s_r$ and $s_t$. Moreover the permutation \mbox{$\sigma\circ (r \ r+1)\in \Rep{\overline{w}}$} satisfies the conditions of the lemma. Thus the result follows from the induction hypothesis.

Assume now that  $|s_rs_{r+1}|=\infty$. By Observation~\ref{obsrep}\eqref{obsrep:4} we have  $|s_{r+1}s_t|=2$. Furthermore from Observation~\ref{obsrep}\eqref{obsrep:3} we obtain that \mbox{$\sigma(r)< \sigma (r+1)$}. Hence $\sigma(t)< \sigma (r+1)$ and we can apply the induction hypothesis to the generators $s_{r+1}$ and $s_t$ since the number of letters between them is less than or equal to $n$. Thus we obtain that
\[ w'=w_1 s_r  \cdots s_{r+1} s_t \cdots w_2 \ \ \text{ and }\ \
w^*=w_1 s_r \cdots s_t s_{r+1} \cdots w_2 \]
are reduced representations of $w$.

Let $\alpha \in \Rep w$ such that $\alpha.w=w^*$. The number of letters between $s_r$ and $s_t$ in $w^*$ is less than or equal to $n$ and therefore we can apply the induction hypothesis to $w^*$ with $\sigma\circ\alpha^{-1}\in\Rep{w^*}$ and we obtain that
\[ w_1\cdots s_r s_t  \cdots w_2 \ \ \text{ and }\ \
w_1\cdots s_t s_r  \cdots w_2 \]
are two reduced representations of $w^*$ and hence of $w$.
\end{proof}

We will now present two results that can be used in right-angled buildings to modify minimal galleries using the commutation relations of the Coxeter group.
We will refer to these results as the ``Closing Squares Lemmas'' (see also Figure~\ref{fig:CS} below).

\begin{lem}[Closing Squares 1]\label{Squares}
Let $c_0$ be a fixed chamber in a right-angled building $\Delta$.
Let $c_1, c_2 \in \Sp{c_0}n$ and $c_3 \in \Sp{c_0}{n+1}$ such that
\[ c_1 \adj{t} c_3 \quad \text{and} \quad c_2 \adj{s} c_3 \]
for some $s \neq t$.
Then $|st|=2$ in $\Sigma$ and there exists $c_4\in \Sp{c_0}{n-1}$ such that
\[ c_1 \adj{s} c_4 \quad \text{and} \quad c_2 \adj{t} c_4. \]
\end{lem}

\begin{proof}
Let $w_1$ and $w_2$ be reduced representations of $\delta(c_0, c_1)$ and $\delta(c_0,c_2)$, respectively.
Then $w_1t$ and $w_2s$ are two reduced representations of $\delta(c_0, c_3)$ and thus $w_1t = w_2s$ in $W$;
hence $|st|=2$.
Furthermore, $l(w_1s)<l(w_1)$ and thus $l(w_1s)=n-1$. Let $c_4$ be the chamber in $\Sp{c_0}{n-1}$ that is $s$\dash adjacent to~$c_1$.
Then $w_1st = w_1ts = w_2 s s = w_2$ in $W$. Therefore $c_4 \adj{t} c_2$.
\end{proof}

\begin{lem}[Closing Squares 2]\label{Squares2}
Let $c_0$ be a fixed chamber in a right-angled building $\Delta$.
Let $c_1, c_2 \in \Sp{c_0}n$ and $c_3\in \Sp{c_0}{n-1}$ such that
\[ c_1 \adj{s} c_2 \quad \text{and} \quad c_2 \adj{t} c_3 \]
for some $s \neq t$.
Then $|st|=2$ in $\Sigma$ and there exists $c_4\in \Sp{c_0}{n-1}$ such that
\[ c_1 \adj{t} c_4 \quad \text{and} \quad c_3 \adj{s} c_4. \]
\end{lem}
\begin{proof}
Let $w_1$ and $w_2$ be reduced representations of $\delta(c_0,c_1)$ and $\delta(c_0,c_2)$.
As $c_1$ and $c_2$ are $s$-adjacent and are both in $\Sp{c_0}n$, we know that $w_1 = w_2 = s_1\cdots s_{n-1} s$ in $W$.
Let $v_1\in \Sp{c_0}{n-1}$ be  the chamber $s$-adjacent to $c_1$ (and~$c_2$) at Weyl distance $s_1\cdots s_{n-1}$ from $c_0$.

 Applying Lemma~\ref{Squares} to $v_1$ and $c_3$ we obtain $v_2\in \Sp{c_0}{n-2}$ such that
\[ v_2 \adj{t} v_1 \quad \text{and} \quad v_2 \adj{s} c_3.\]
Furthermore $|st|=2$  in $\Sigma$. Since there is a minimal gallery of type $ts$ between $v_2$ and $c_1$, there must be one of type $st$. Hence there is a chamber $c_4\in \Sp{c_0}{n-1}$ that is $t$-adjacent to $c_1$ and $s$-adjacent to $v_2$.
Since $v_2$ is $s$-adjacent to $c_3$ and to~$c_4$, we conclude that $c_3$ and $c_4$ are $s$-adjacent.
\end{proof}

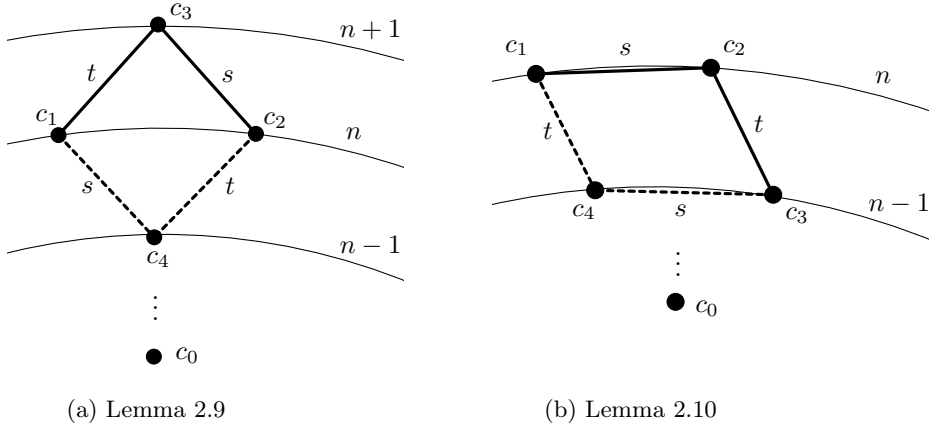
\begin{figure}[ht!]
\begin{subfigure}[b]{0.3\textwidth}
\definecolor{ttqqqq}{rgb}{0.2,0.,0.}
\begin{tikzpicture}[scale=0.8,line cap=round,line join=round,>=triangle 45,x=1.0cm,y=1.0cm]
\clip(11.5,-2) rectangle (18,4.28);
\draw [shift={(14.0370247934,-10.3016528926)}] plot[domain=1.02538367692:2.14734467758,variable=\t]({1.*12.4897881927*cos(\t r)+0.*12.4897881927*sin(\t r)},{0.*12.4897881927*cos(\t r)+1.*12.4897881927*sin(\t r)});
\draw [shift={(14.0370247934,-10.3016528926)}] plot[domain=0.997988156797:2.15611076477,variable=\t]({1.*14.1808741908*cos(\t r)+0.*14.1808741908*sin(\t r)},{0.*14.1808741908*cos(\t r)+1.*14.1808741908*sin(\t r)});
\draw [shift={(14.0370247934,-10.3016528926)}] plot[domain=1.06235112263:2.14028252827,variable=\t]({1.*10.7291062881*cos(\t r)+0.*10.7291062881*sin(\t r)},{0.*10.7291062881*cos(\t r)+1.*10.7291062881*sin(\t r)});
\draw (16.9,2.35394050274) node[anchor=north west] {$n$};
\draw (16.8,0.55) node[anchor=north west] {$n-1$};
\draw (16.8,4.13212405444) node[anchor=north west] {$n+1$};
\draw (15.5471274469,2.6) node[anchor=north west] {$c_2$};
\draw (11.8,2.65) node[anchor=north west] {$c_1$};
\draw (14.1,-1.3) node[anchor=north west] {$c_0$};
\draw (13.71,-0.17) node[anchor=north west] {$\vdots$};
\draw (13.6247668505,0.3) node[anchor=north west] {$c_4$};
\draw (14.0092389698,4.4) node[anchor=north west] {$c_3$};
\draw [line width=1.2pt] (13.9701176712,3.90429032555)-- (12.3332286877,2.07137777611);
\draw [line width=1.2pt] (13.9701176712,3.90429032555)-- (15.5781703153,2.09268773443);
\draw [line width=1.2pt,dash pattern=on 2pt off 2pt] (12.3332286877,2.07137777611)-- (13.9068569939,0.377507567897);
\draw [line width=1.2pt,dash pattern=on 2pt off 2pt] (15.5781703153,2.09268773443)-- (13.9068569939,0.377507567897);
\draw (12.6436849104,3.4) node[anchor=north west] {$t$};
\draw (14.8582815665,3.3) node[anchor=north west] {$s$};
\draw (12.55,1.45) node[anchor=north west] {$s$};
\draw (14.9223602531,1.52) node[anchor=north west] {$t$};
\begin{scriptsize}
\draw [fill=black] (12.3332286877,2.07137777611) circle (3.5pt);
\draw [fill=black] (15.5781703153,2.09268773443) circle (3.5pt);
\draw [fill=black] (13.9,-1.6) circle (3.5pt);
\draw [fill=black] (13.9068569939,0.377507567897) circle (3.5pt);
\draw [fill=black] (13.9701176712,3.90429032555) circle (3.5pt);
\end{scriptsize}
\end{tikzpicture}
\subcaption{\centering Lemma~\ref{Squares}}
\end{subfigure}
\hspace{2.5cm}
\begin{subfigure}[b]{0.3\textwidth}
\definecolor{ttqqqq}{rgb}{0.2,0.,0.}
\begin{tikzpicture}[scale=0.91,line cap=round,line join=round,>=triangle 45,x=1.0cm,y=1.0cm]
\clip(11.7,-2.4) rectangle (18,2.8);
\draw [shift={(14.0370247934,-10.3016528926)}] plot[domain=1.02538367692:2.14734467758,variable=\t]({1.*12.4897881927*cos(\t r)+0.*12.4897881927*sin(\t r)},{0.*12.4897881927*cos(\t r)+1.*12.4897881927*sin(\t r)});
\draw [shift={(14.0370247934,-10.3016528926)}] plot[domain=1.06235112263:2.14028252827,variable=\t]({1.*10.7291062881*cos(\t r)+0.*10.7291062881*sin(\t r)},{0.*10.7291062881*cos(\t r)+1.*10.7291062881*sin(\t r)});
\draw (17.1,2.2426021271) node[anchor=north west] {$n$};
\draw (17,0.453792447377) node[anchor=north west] {$n-1$};
\draw (14.9,2.7) node[anchor=north west] {$c_2$};
\draw (11.75,2.7) node[anchor=north west] {$c_1$};
\draw (14.5,-1.1) node[anchor=north west] {$c_0$};
\draw (14.2,-0.15) node[anchor=north west] {$\vdots$};
\draw (12.7186231146,0.31) node[anchor=north west] {$c_4$};
\draw (15.8,0.26) node[anchor=north west] {$c_3$};
\draw [line width=1.2pt,dash pattern=on 2pt off 2pt] (12.3332286877,2.07137777611)-- (13.1886752932,0.375450417608);
\draw (12.3,1.5) node[anchor=north west] {$t$};
\draw (14.2,0.3) node[anchor=north west] {$s$};
\draw (13.4,2.65) node[anchor=north west] {$s$};
\draw (15.35,1.6) node[anchor=north west] {$t$};
\draw [line width=1.2pt] (12.3332286877,2.07137777611)-- (14.8607567251,2.16094213425);
\draw [line width=1.2pt] (14.8607567251,2.16094213425)-- (15.7609052706,0.310165392801);
\draw [line width=1.2pt,dash pattern=on 2pt off 2pt] (15.7609052706,0.310165392801)-- (13.1886752932,0.375450417608);
\begin{scriptsize}
\draw [fill=black] (12.3332286877,2.07137777611) circle (3.5pt);
\draw [fill=black] (14.8607567251,2.16094213425) circle (3.5pt);
\draw [fill=black] (14.35,-1.25) circle (3.5pt);
\draw [fill=black] (13.1886752932,0.375450417608) circle (3.5pt);
\draw [fill=black] (15.7609052706,0.310165392801) circle (3.5pt);
\end{scriptsize}
\end{tikzpicture}
\subcaption{Lemma~\ref{Squares2}}
\end{subfigure}
\caption{Closing squares}\label{fig:CS}
\end{figure}

As a consequence of the closing squares lemmas, we are able to transform minimal galleries into ``concave'' minimal galleries.
\begin{lem}\label{concave}
Let $c_1$ and $c_2$ be two chambers in $\Delta$. There exists a minimal gallery $\gamma=(v_0, \dots, v_\ell)$ in $\Delta$ between $c_1=v_0$ and $c_2=v_\ell$
such that there are numbers $0 \leq j \leq k \leq \ell $ satisfying the following:
\begin{enumerate}
\item $\dw(c_0, v_i) < \dw(c_0, v_{i-1})$ for all $i\in\{1, \dots, j\}$;
\item $\dw(c_0, v_i) = \dw(c_0, v_{i-1})$ for all $i\in\{j+1, \dots, k \}$;
\item $\dw(c_0, v_i) > \dw(c_0, v_{i-1})$ for all $i\in\{k+1, \dots, \ell \}$.
\end{enumerate}
\end{lem}
\begin{proof}
Let $(v_0, \dots , v_\ell)$ be a minimal gallery from $c_1$ to $c_2$ in $\Delta$.
We will essentially prove the result by closing squares whenever possible.

Let $h(\gamma) := \sum_{i=0}^\ell \dw(c_0,v_i)$ be the ``total height'' of the gallery w.r.t.\@~$c_0$.
Observe that the gallery $\gamma$ is of the required form if and only if it does not contain length $2$ subgalleries of any of the following form:
\begin{enumerate}[(a)]
    \item
        $(a, b, c)$ with $\dw(c_0, a)=r$, $\dw(c_0, b) = r+1$, $\dw(c_0, c) = r$;
    \item
        $(a, b, c)$ with $\dw(c_0, a)=r$, $\dw(c_0, b) = r+1$, $\dw(c_0, c) = r+1$;
    \item
        $(a, b, c)$ with $\dw(c_0, a)=r+1$, $\dw(c_0, b) = r+1$, $\dw(c_0, c) = r$.
\end{enumerate}
Indeed, the exclusion of galleries of type (a) and (b) says that once we start going up, we have to continue going up,
and the exclusion of galleries of type (a) and (c) says that once we stop going down, we can never go down again.

We will now show that if $\gamma$ contains a length $2$ subgallery of any of these forms, then we can replace $\gamma$ by another minimal gallery $\gamma'$ from $c_1$ to $c_2$
for which $h(\gamma') < h(\gamma)$.
Since $h(\gamma)$ is a natural number, this process has to stop eventually, and we will be left with a minimal gallery of the required form.

If we have a subgallery of type (a), then we can apply Lemma~\ref{Squares} to replace $b$ by some chamber $b'$ with $\dw(c_0, b') = r-1$.
If we have a subgallery of type (b) or type (c), then we can apply Lemma~\ref{Squares2} to replace $b$ by some chamber $b'$ with $\dw(c_0, b') = r$.
In all cases, we have replaced one chamber in $\gamma$ by a chamber which is closer to $c_0$, and hence we have indeed decreased the value of $h(\gamma)$, as claimed.
\end{proof}

\begin{coro}\label{lemmaB}
    Let $c_0$ be a fixed chamber in $\Delta$.
    If $c_1, c_2 \in \B{c_0}{n}$, then there exists a minimal gallery from $c_1$ to $c_2$ inside $\B{c_0}n$.
\end{coro}
\begin{proof}
    This follows directly from Lemma~\ref{concave}.
\end{proof}

\subsection{Projections and parallel residues}

Let $\Sigma$ be a right-angled Coxeter diagram with vertex set $S$ and let $(W,S)$ be the Coxeter system of type $\Sigma$ with set of generators $S=(s_i)_{i\in I}$. Let $\Delta$ be a right-angled building of type $\Sigma$.

\begin{defn}
Let $c$ be a chamber in $\Delta$ and $\Ra$ be a residue in $\Delta$.
The \emph{projection} of $c$ on $\Ra$ is the unique chamber in $\Ra$ that is closest to $c$ and it is denoted by $\proj_{\Ra}(c)$.
\end{defn}

The next fact is usually called the gate property and can be found, for instance, in \cite[Proposition 3.105]{AbramenkoBrown}.
\begin{prop}[Gate property]\label{Gate}
Let $c$ be a chamber in $\Delta$ and $\Ra$ be a residue in $\Delta$.
For any chamber $c'$ in $\Ra$, there is a minimal gallery from $c$ to $c'$ passing through $\proj_{\Ra}(c)$, and such that the subgallery from $\proj_{\Ra}(c)$ to $c'$ is contained in $\Ra$.
\end{prop}

Next we present applications of the projection map that will be useful for us later on.
The first shows how the projections to panels are related to the structure of the Coxeter diagram of the building.

\begin{lem}\label{lemmaA}
Let $c_0$ be a fixed chamber of $\Delta$ and let $s\in S$.
Let $c_1\in \Sp{c_0}n$ and $c\in \B{c_0}{n+1}\setminus \ch{\Pa_{s,c_1}}$.
If $\proj_{\Pa_{s,c_1}}(c)=c_2\in \Sp{c_0}{ n+1}$ then $c_2$ is $t$-adjacent to some $c_3\in \Sp{c_0}n$ with $t\neq s$ and $st=ts$ in $W$.
\end{lem}
\begin{proof}
By Lemma~\ref{concave} we can take a concave minimal gallery between $c$ and $c_2$. Let $w=s_1\cdots s_\ell$ be the corresponding reduced representation of $\delta(c,c_2)$.
Let $v_1$ be the chamber $s_\ell$-adjacent to $c_2$ that is at Weyl distance $s_1\cdots s_{\ell-1}$ from $c$. We have \mbox{$\dw(c_0, c_2)\geq \dw(c_0, v_1)$}.

We know that $s_\ell\neq s$ because $\proj_{\Pa_{s,c_1}}(c)=c_2$ and hence $l(w)<l(ws)$. If $v_1\in \Sp{c_0}n$ then the result follows from Lemma~\ref{Squares}.
If $v_1\in \Sp{c_0}{n+1}$, then there exists a (unique) chamber $v_2 \in \Sp{c_0}{n}$ that is $s_\ell$-adjacent to both $v_1$ and $c_2$,
and the result follows again from Lemma~\ref{Squares}.
\end{proof}

The next result will allow us, in the case of semi-regular right-angled buildings, to extend a permutation of an $s$-panel to an automorphism of the whole building in a useful way.

\begin{prop}[{\cite[Proposition 4.1]{Caprace}}]\label{extending}
Let $\Delta$ be a semi-regular right-angled building of type $\Sigma$. Let $s\in S$ and $\Pa$ be an $s$-panel. Given any permutation $\alpha \in \Sym{\ch{\Pa}}$ there is $\widetilde{\alpha}\in \Aut\Delta$ stabilizing $\Pa$ satisfying the following two conditions:

\begin{enumerate}
\item $\widetilde{\alpha}|_{\ch{\Pa}}=\alpha$;
\item $\widetilde{\alpha}$ fixes all chambers of $\Delta$ whose projection to $\Pa$ is fixed by $\alpha$.
\end{enumerate}
\end{prop}

Our last application of the projection map gives a consequence of combinatorial convexity (cf.\@ Definition~\ref{convex}), in terms of the projection map.

\begin{prop}[{\cite[Section 2]{Caprace}}]\label{convexproj}
If a set of chambers $\mathcal C$ is combinatorially convex then for every $c\in \mathcal C$ and every residue $\Ra$ of $\Delta$ with $\ch{\Ra} \cap \mathcal {C} \neq \emptyset$ we have $\proj_{\Ra}(c)\in \mathcal C$.
\end{prop}
\begin{proof}
Let $c_1 \in \ch{\Ra} \cap \mathcal {C}$ and $c\in \mathcal C$. We know that there is a minimal gallery from $c$ to $c_1$ passing through $\proj_{\Ra}(c)$. As $\mathcal C$ is combinatorially convex the result follows.
\end{proof}

If $\Ra_1$ and $\Ra_2$ are two residues, then
\[ \proj_{\Ra_1}(\Ra_2)=\{ \proj_{\Ra_1}(c) \mid c\in \ch{\Ra_2} \} \]
is the set of chambers of a residue contained in $\Ra_1$. This is again a residue (cf.\@~\cite[Section 2]{Caprace}) and the rank of $\proj_{\Ra_1}(\Ra_2)$ is bounded above by the ranks of both $\Ra_1$ and $\Ra_2$.

\medskip

We finally collect some facts from \cite{Caprace} about parallelism of residues.

\begin{defn}
Two residues $\Ra_1$ and $\Ra_2$ are called \emph{parallel} if \mbox{$\proj_{\Ra_1}(\Ra_2)=\Ra_1$} and $\proj_{\Ra_2}(\Ra_1)=\Ra_2$.
\end{defn}

In particular, if $\Pa_1$ and $\Pa_2$ are two parallel panels, then the chamber sets of $\Pa_1$ and $\Pa_2$ are mutually in bijection under the respective projection maps (cf. \cite[Section 2]{Caprace}).

\begin{lem}[{\cite[Lemma 2.2]{Caprace}}]\label{Lemma2.2}
Let $J_1, J_2 \subset S$ be two disjoint subsets with $[J_1, J_2]=1$. Let $c\in \ch\Delta$. Then
\[ \ch{\Ra_{{J_1\cup J_2},c}}=\ch{\Ra_{{J_1},c}}\times \ch{\Ra_{{J_2},c}}, \]
and for $i\in \{1,2\}$, the canonical projection map $\ch{\Ra_{{J_1\cup J_2},c}} \to \ch{\Ra_{{J_i}, c}}$ coincides with the restriction of $\proj_{\Ra_{{J_i}, c}}$ to $\ch{\Ra_{{J_1\cup J_2}, c}}$. In particular, any two $J_i$-residues contained in $\Ra_{{J_1\cup J_2}, c}$ are parallel.
\end{lem}

\begin{lem}[{\cite[Lemma 2.5]{Caprace}}]\label{Lemma2.5}
Let $\Pa_1$ and $\Pa_2$ be panels in $\Delta$. If there are two chambers of $\Pa_2$ having distinct projections on $\Pa_1$, then $\Pa_1$ and $\Pa_2$ are parallel.
\end{lem}

\begin{defn}\label{perp}
Let $J\subseteq S$. We define the set
\[ J^\perp = \{ t\in S\setminus J \mid ts=st \text{ for all } s\in J \}. \]
If $J=\{s\}$ then we denote  the set $J^\perp$ by $s^\perp$.
\end{defn}

\begin{prop}[{\cite[Proposition 2.8]{Caprace}}]\label{Proposition2.8}
Let $\Delta$ be a right-angled building of type $(W, S)$.
\begin{enumerate}
\item Any two parallel residues have the same type.
\item Let $J\subseteq S$. Given a residue $\Ra$ of type $J$, a residue $\Ra'$ is parallel to $\Ra$ if and only if $\Ra'$ is of type $J$ and $\Ra$ and $\Ra'$ are both contained in a common residue of type $J\cup J^\perp$.
\end{enumerate}
\end{prop}

\begin{prop}[{\cite[Corollary 2.9]{Caprace}}]
Let $\Delta$ be a right-angled building. Then parallelism of residues is an equivalence relation.
\end{prop}

\subsection{Tree-walls and wings}

We want to describe the equivalence classes of parallelism of panels in right-angled buildings.
It turns out that these classes are the so called tree-walls, initially defined in \cite{Bourdon} and taken over in \cite{TW11}.
Each tree-wall separates the building into combinatorially convex components, which will be called wings as in \cite{Caprace}.
At the end of this section, we will define a distance between tree-walls of the same type, which will be useful for our inductive arguments later on.

\begin{quote}\em
    We will assume from now on (and until the end of the paper) that $\Delta$ is
    a semi-regular right-angled building with prescribed thickness of the panels given by a set of cardinal numbers $(q_s)_{s\in S}$
    (see Theorem~\ref{unicityhp}).
    \smallskip
\end{quote}

\begin{defn}
Let $s\in S$.
An \emph{$s$-tree-wall} in $\Delta$ is an equivalence class of parallel $s$-panels of $\Delta$.
\end{defn}
By some slight abuse of notation, we will write $\ch\T$ for the set of all chambers contained in some $s$-panel belonging to the $s$-tree-wall $\T$,
and we will refer to these chambers as the {\em chambers of $\T$}.

Using Proposition~\ref{Proposition2.8} we know exactly how to describe the tree-walls in a right-angled building.
\begin{coro}\label{4.39}
Let $\Delta$ be a right-angled building of type $(W, S)$ and let $s\in S$. Then two $s$-panels $\Pa_1$ and $\Pa_2$
belong to the same $s$-tree-wall if and only if they are both contained in a common residue of type $s\cup s^\perp$.
\end{coro}
In other words, the $s$-tree-walls are the sets of $s$-panels contained in a residue of type $s\cup s^\perp$.
\begin{coro}\label{tree residue}
Let $\T$ be an $s$-tree-wall in $\Delta$, let $\Pa$ be an $s$-panel in $\T$, and let $\Ra$ be the residue of type $s\cup s^\perp$ containing $\Pa$.
Then $\ch\T=\ch\Ra$.
\end{coro}

\begin{coro}[{\cite[Corollary 3]{TW11}}]
The following possibilities for tree-walls in $\Delta$ can occur:
\begin{enumerate}
\item An $s$-tree-wall is reduced to a panel if and only if $\langle s^\perp\rangle$ is trivial.
\item An $s$-tree-wall is finite but not reduced to a panel if and only if $\langle s^\perp\rangle$ is finite but non-trivial.
\item An $s$-tree-wall is infinite if and only if $\langle s^\perp\rangle$ is infinite.
\end{enumerate}
\end{coro}

\begin{exa}
\begin{enumerate}
\item  Let $W=\langle s,t \mid s^2=t^2=1\rangle$. Let $\Delta$ be a right-angled building of type $W$. Then $\Delta$ is a tree.
The $s$-tree-walls are the $s$-panels and the $t$-tree-walls are the $t$-panels of $\Delta$.

\item Let $W=\langle s,t,q \mid s^2=t^2=q^2=(st)^2=1\rangle$ and let $\Delta$ be a right-angled building of type $W$.
\begin{itemize}
\item The $q$-tree-walls are the sets of $q$-panels of $\Delta$.
\item The $s$-tree-walls are the sets of $s$-panels  in a common residue of type $\{s, t\}$ in $\Delta$.
\item The $t$-tree-walls are the sets of $t$-panels in a common residue of type $\{s, t\}$ in $\Delta$.
\end{itemize}
\end{enumerate}
\end{exa}

By Corollary~\ref{tree residue}, we can define projections on tree-walls.
\begin{defn}
    Let $s\in S$, let $\T$ be an $s$-tree-wall of $\Delta$, and let $c\in \ch\Delta$.
    We define the {\em projection of $c$ on $\T$} as $\proj_\T(c) := \proj_\Ra(c)$, where $\Ra$ is the residue of type $s \cup s^\perp$ containing the $s$-panels of $\T$.
\end{defn}

\begin{lem}\label{notreduced}
    Let $s\in S$, let $\T$ be an $s$-tree-wall of $\Delta$, let $c\in \ch\Delta$ and let $c_2\in\ch\T$.
    Let $w_1$ and $w_2$ be reduced representations of $\delta(c,\proj_{\T}(c))$ and $\delta(\proj_{\T}(c),c_2)$, respectively.
    Then $w_1w_2$ is a reduced representation of $\delta(c,c_2)$.
\end{lem}
\begin{proof}
    This follows immediately from the gate property (Proposition~\ref{Gate}).
\end{proof}

Each $s$-tree-wall yields a partition of the building into $q_s$ combinatorially convex components, which are called wings. We define wings in a building and present some results that connect wings, projections and tree-walls.

\begin{defn}\label{wings}
Let $c\in \ch\Delta$ and $s\in S$. Then the set of chambers
\[ X_s(c)=\{x\in \ch\Delta \mid \proj_{\Pa_{s,c}}(x)=c\} \]
is called the \emph{$s$-wing} of $c$.
\end{defn}
Notice that if $\Pa$ is any $s$-panel, then the set of $s$-wings of each of the $q_s$ different chambers of $\Pa$ forms a partition of $\ch\Delta$ into $q_s$ subsets.

We note that we defined wings with respect to panels only since it is sufficient for our purposes.
However, this concept can be generalized to residues of any type; see \cite{Caprace}.

\begin{prop}[{\cite[Proposition 3.2]{Caprace}}]
In a right-angled building, wings are combinatorially convex.
\end{prop}

The next lemma is proved in \cite[Lemma 3.1]{Caprace} for general residues.

\begin{lem}\label{le:wings}
Let $s\in S$ and $\T$ be an $s$-tree-wall.
Let $\Pa_1$ and $\Pa_2$ be $s$-panels of $\T$, and let $c_1 \in \Pa_1$ and $c_2\in \Pa_2$.
Then $X_s(c_1)=X_s(c_2)$ if and only if $c_2=\proj_{\Pa_2}(c_1)$, i.e., if and only if $c_1\in X_s(c_2)$.
\end{lem}

In particular, projection between parallel $s$-panels induces an equivalence relation on $\ch\T$.
\begin{prop}\label{3panels}
    Let $\Pa_1, \Pa_2$ and $\Pa_3$ be $s$-panels in a common $s$-tree-wall. Let $c_1\in \Pa_1$, $c_2=\proj_{\Pa_2}(c_1)$ and $c_3=\proj_{\Pa_3}(c_1)$.
    Then $\proj_{\Pa_3}(c_2)=c_3$.
\end{prop}
\begin{proof}
    By Lemma~\ref{le:wings}, $c_2=\proj_{\Pa_2}(c_1)$ implies $X_s(c_1) = X_s(c_2)$,
    and $c_3=\proj_{\Pa_3}(c_1)$ implies $X_s(c_1) = X_s(c_3)$.
    Hence $X_s(c_2) = X_s(c_3)$, and therefore $\proj_{\Pa_3}(c_2) = c_3$.
\end{proof}

By Lemma~\ref{le:wings}, it makes sense to define a partition of $\ch\Delta$ into $s$-wings with respect to an $s$-tree-wall.
\begin{defn}\label{def:T-wings}
    Let $s\in S$ and $\T$ be an $s$-tree-wall.
    Let $\Pa$ be an arbitrary $s$-panel of $\T$.
    Then $\Pa$ induces a partition
    \[ \{ X_s(c) \mid c \in \Pa \} \]
    of $\ch\Delta$ into $q_s$ subsets, which we call the {\em partition of $\ch\Delta$ into $s$-wings with respect to $\T$}.
    By Lemma~\ref{le:wings}, this partition is independent of the choice of $\Pa$ in $\T$.
\end{defn}

We will now study the interaction between different $s$-tree-walls.
\begin{defn}
Let $s\in S$ and let $\T_1$ and $\T_2$ be two  $s$-tree-walls in $\Delta$.
We define the set
\[ \proj_{\T_1}(\T_2)=\{ \proj_{\T_1}(c) \mid c\in \T_2\}. \]
\end{defn}

\begin{prop}\label{4.53}
Let $s\in S$ and let $\T_1$ and $\T_2$ be distinct $s$-tree-walls. Let  $c_2\in \proj_{\T_2}(\T_1)$ and $c_1=\proj_{\T_1}(c_2)$.
Let $w$ be a reduced representation of $\delta(c_1,c_2)$. Then the following hold.
\begin{enumerate}
\item If $t\in S$ with $|ts| \leq 2$ then $l(tw)=l(wt)=l(w)+1$.
\item $c_2= \proj_{\T_2}(c_1)$. \label{4.53.2}
\item $\ch{\T_1}\subseteq X_s(c_2)$.
\item If $ c_2'\in \proj_{\T_2}(\T_1)$ then $\proj_{\Pa_{s, c_2}}(c_2')=c_2$.\label{4.53.4}
\end{enumerate}
\end{prop}
\begin{proof}
\begin{enumerate}
\item Let $t\in S$ such that $|ts| \leq 2$. Let $c$ be a chamber $t$-adjacent to~$c_1$ (and hence $c\in \ch{\T_1}$).
Since $c_1=\proj_{\T_1}(c_2)$, we can apply the gate property (Proposition~\ref{Gate}) to find a minimal gallery from $c$ to $c_2$ passing through $c_1$, of type $tw$.
In particular, $l(w)<l(tw)$.

Since $c_2\in \proj_{\T_2}(\T_1)$, there is a chamber $d\in \T_1$ such that $\proj_{\T_2}(d)=c_2$.
Let $w_1$ be a reduced representation of $\delta (d, c_1)$. By the gate property again, $w_1w$ is a reduced representation of $\delta(d,c_2)$.
Let $e$ be a chamber $t$-adjacent to $c_2$ (and therefore in $\ch{\T_2}$).
The gate property w.r.t.\@ $\proj_{\T_2}(d)=c_2$ now implies that $l(w_1wt) > l(w_1w)$ and hence $l(wt) > l(w)$ as well.

\item By the gate property, $w$ can be written as $w_1w_2$ where $w_2$ is a reduced representation of the subgallery from $\proj_{\T_2}(c_1)$ to $c_2$ inside $\T_2$. Hence $|ts| \leq 2$ for all $t\in w_2$. By Statement 1 we get $l(w_2)=0$ and thus $c_2= \proj_{\T_2}(c_1)$.

\item Let $c\in \ch{\T_1}$ and let $w_1\in W$ be a reduced representation of $\delta(c, c_1)$.
By Lemma~\ref{notreduced}, $w_1w$ is a reduced representation of $\delta(c, c_2)$.
By Statement 1, $l(w_1w)<l(w_1ws)$, and it follows that $\proj_{\Pa_{s,c_2}}(c)=c_2$.

\item  By Statement 3, $\ch{\T_1} \subseteq X_s(c_2)$ and $\ch{\T_1} \subseteq X_s(c_2')$.
Since the $s$\dash wings with respect to $\T_2$ form a partition of $\ch\Delta$ (see Definition~\ref{def:T-wings}),
this implies $X_s(c_2)=X_s(c_2')$, and hence $\proj_{\Pa_{s, c_2}}(c_2')=c_2$ by Lemma~\ref{le:wings}.
\qedhere
\end{enumerate}
\end{proof}

We finish this section by defining a distance between tree-walls of the same type.

\begin{defn}\label{deftreewall}
Let $s\in S$. Let $ V_1$ be the set of all $s$-tree-walls of $\Delta$ and let $V_2$ be the set of all residues of type $S\setminus \{s\}$ of $\Delta$. Consider the bipartite graph $\Gamma_s$ with vertex set $V_1\sqcup V_2$, where an $s$-tree-wall $\T\in V_1$ is adjacent to a residue $\Ra\in V_2$ if and only if $\ch\T\cap\ch{\Ra}\neq\emptyset$.

The graph $\Gamma_s$ will be called the \emph{tree-wall tree of type} $s$.
\end{defn}

Notice that each $s$-tree-wall $\T$ in $\Gamma_s$ has precisely $q_s$ neighbors, corresponding to each of the residues of type $S\setminus \{s\}$
lying in a distinct part of the partition of $\ch\Delta$ induced by the $s$-wings with respect to $\T$ (see Definition~\ref{def:T-wings}).

Moreover, if there is a minimal path $\T_1 - \Ra_1 - \cdots - \Ra_2 - \T_2$ in $\Gamma_s$ and $c_1\in \ch{\T_1}\cap\ch{\Ra_1}$ and $c_2\in \ch{\T_2}\cap\ch{\Ra_2}$, then $c_2\in\proj_{\T_2}(\T_1)$ and $c_1\in \proj_{\T_1}(\T_2)$.
Therefore Proposition~\ref{4.53} then implies that $\ch{\T_2}\subseteq X_s(c_1)$ and $\ch{\T_1}\subseteq X_s(c_2)$.

\begin{prop}\label{tree-wall tree}
Let $s\in S$. The tree-wall tree $\Gamma_s$ is a tree.
\end{prop}
\begin{proof}
The graph $\Gamma_s$ is connected. Therefore it is enough to show that there are no cycles in $\Gamma_s$.
If there were a non-trivial cycle in $\Gamma_s$, say
\[ \T_1 - \Ra_1 -  \T_2 - \cdots - \Ra_n - \T_1 , \]
then the chambers of $\T_2$ would be contained in two distinct $s$-wings with respect to $\T_1$, namely the ones corresponding to $\Ra_1$ and $\Ra_n$, which is a contradiction.
We conclude that $\Gamma_s$ is a tree.
\end{proof}

\begin{defn}
Let $s\in S$ and let $\T_1$ and $\T_2$ be two $s$-tree-walls. We define the \emph{$s$-tree-wall distance}, denoted by $\dtw$, as
\[ \dtw(\T_1, \T_2) =  \frac{1}{2}\, {\sf{dist}}_{\Gamma_s}(\T_1, \T_2), \]
where ${\sf{dist}}_{\Gamma_s}(\T_1, \T_2)$ denotes the discrete distance in the tree-wall tree of type $s$.
\end{defn}

\subsection{Colorings in right-angled buildings}\label{ss:colorings}

In this section we will define different types of colorings of the chambers of semi-regular right-angled buildings. These colorings will be necessary to define the universal group later in Section~\ref{se:univ}.
We will use the following notation throughout the section.

Let $(W,S)$ be a right-angled Coxeter system with Coxeter diagram $\Sigma$ with vertex set $I$ and set of generators $S=\{s_i\}_{i\in I}$.
 Let $(q_s)_{s\in S}$ be a set of cardinal numbers with $q_s\geq 2$ for all $s\in S$.
Let $\Delta$  be the unique right-angled building of type $(W,S)$ with parameters $(q_s)_{s\in S}$ (see Theorem~\ref{unicityhp}).
For each $s\in S$, let $Y_s$ be a set of cardinality $q_s$;
we will refer to $Y_s$ as the set of {\em $s$-colors}.

\begin{defn}\label{coloring}
Let $s\in S$. A map $h_s \colon \ch\Delta\to Y_s$ is called an \emph{$s$-coloring} of $\Delta$ if
\begin{itemize}
    \item[(C)]
        for every $s$-panel $\Pa$ there is a bijection between the colors in $Y_s$ and the chambers in $\Pa$.
\end{itemize}
\end{defn}

\subsubsection{Legal colorings}

\begin{defn}\label{colorings}
Let $s\in S$. An $s$-coloring $h_s \colon \ch\Delta \to Y_s$ is called a \emph{legal $s$-coloring} if it satisfies
\begin{itemize}
    \item[(L)]
        for every $S\setminus \{s\}$-residue $\Ra$ and for all $c_1, c_2\in \ch\Ra$, $h_s(c_1)=h_s(c_2)$.
\end{itemize}
\end{defn}

In particular, if $\Pa_t$ is a $t$-panel then for every $s\in S\setminus \{t\}$ one can consider the $s$-color of the panel $\Pa_t$, denoted by $h_s(\Pa_t)$, since all the chambers in $\Pa_t$ have the same $s$-color.
Similarly, if $\T$ is a $t$-tree-wall with $|st|=\infty$ in $\Sigma$, then by Corollary~\ref{4.39},  $h_s(c_1)=h_s(c_2)$ for all $c_1, c_2 \in \ch\T$, and hence the color $h_s(\T)$ is well-defined.

\begin{prop}\label{elements}
Let $c_0 \in \ch\Delta$ and $(h_s^1)_{s\in S}$ and let $(h_s^2)_{s\in S}$ be two sets of legal colorings of $\Delta$. Then there exists $c \in \ch\Delta$ such that $h_s^1(c)=h_s^2(c_0)$ for all $s\in S$.
\end{prop}
\begin{proof}
We will prove the result recursively.
For each $s\in S$, let $\Pa_{s,c_0}$ be the $s$-panel that contains $c_0$. By the definition of a legal coloring, we know that there exists $c_1 \in \Pa_{s,c_0}$ such that  $h_{s}^1(c_1)= h_{s}^2(c_0)$.
Moreover, $h_{t}^1(c_1)=h_{t}^1(c_0)$ for all $t\neq s$.
Repeating this procedure for each $s\in S$ we find a chamber $c$ such that $h_s^1(c)=h_s^2(c_0)$ for all $s\in S$.
\end{proof}

\begin{prop}\label{existenceconjugate}
Let  $(h_s^1)_{s\in S}$ and $(h_s^2)_{s\in S}$ be two sets of legal colorings of $\Delta$. Let $c_0, c_0'\in \ch\Delta$ such that $h_s^1(c_0')=h_s^2(c_0)$ for all $s\in S$. Then there exists $g\in \Aut\Delta$ such that $g.c_0=c_0'$ and $h_s^2=h_s^1\circ g$, for all $s\in S$.
\end{prop}
\begin{proof}
Consider the set
\[ G_n = \{ g \in \Aut\Delta \mid c_0'=g.c_0 \text{ and } h_s^1\circ g|_{\B {c_0}n}=h_s^2|_{\B {c_0} n} \text{ for all } s\in S\}. \]
We will recursively construct a sequence of elements $g_i$ ($i \in \N$) such that for every $i \in \N$, we have $g_i \in G_i$, and if $j \in \N$ is larger than $i$, then $g_i$ and $g_j$ agree on the ball $\B{c_0}{i}$.

Since $\Aut\Delta$ is chamber-transitive, $G_0$ is non-empty and we can pick a $g_0 \in G_0$ at random.

Let us assume that we already have constructed automorphisms $g_i$ for every $i\leq n$ with the right properties.
In particular $g_n.c_0=c_0'$ and $h_s^1\circ g_n (c)=h_s^2(c)$ for all $c\in \B{c_0}n$.
Therefore without loss of generality we can assume that $c_0'=c_0$ and that
\begin{equation}\label{eq:h1=h2}
    h_s^1(c) = h_s^2(c) \text{ for all } s \in S \text{ and for all } c\in \B{c_0}n . \tag{$*$}
\end{equation}
We will construct an element $g_{n+1}$ of $G_{n+1}$ by modifying $g_n$ (which now acts trivially on $\B{c_0}n$) step by step along $\Sp{c_0}{n+1}$.

Let $v\in \Sp{c_0}n$ and fix some $s\in S$.
Let $\alpha_s$ be the permutation of the chambers of $\Pa_{s,v}$ such that $h_s^2(c)=h_s^1(\alpha_s(c))$ for all $c\in \Pa_{s,v}$.
By Proposition~\ref{extending}, $\alpha_s$ extends to an automorphism $\widetilde{\alpha_s}$ such that $\widetilde{\alpha_s}$ stabilizes $\Pa_{s,v}$ and fixes all the chambers of $\Delta$ whose projection on $\Pa_{s,v}$ is fixed by $\alpha_s$.

We claim that $\widetilde{\alpha}_s$ fixes $\B{c_0}{n+1}\setminus {\ch{\Pa_{s,v}}}$.
 Let $c\in \B{c_0}{n+1}\setminus{\ch{\Pa_{s,v}}}$, and let $d = \proj_{\Pa_{s,v}}(c)$.
If $d \in \B{c_0}{n}$, then $d$ is fixed by $\alpha_s$, and this already implies that $c$ is fixed by $\widetilde{\alpha_s}$.

Suppose now that $d \in \Sp{c_0}{n+1}$.
By Lemma~\ref{lemmaA} there exists $v'\in \Sp{c_0}n$ such that $d$ is $t$-adjacent to $v'$ with $t \neq s$ (and $ts=st$ in $W$).
The definition of a legal coloring together with~\eqref{eq:h1=h2} now implies
\[ h_s^1(d) = h_s^1(v') = h_s^2(v') = h_s^2(d) , \]
so $\alpha_s$ must fix $d$.
Hence the automorphism $\widetilde{\alpha_s}$ fixes $c$ also in this case.

We have thus constructed, for each $s\in S$, an automorphism $\widetilde{\alpha_s} \in \Aut\Delta$, with the property that all elements of $\B{c_0}{n+1}$ that
are moved by $\widetilde{\alpha_s}$ are contained in $\Sp{c_0}{n+1} \cap \Pa_{s,v}$.
We now vary $s$, and we consider the element
\[ \alpha_v=\prod_{s\in S} \widetilde{\alpha_s} \in \Aut\Delta \]
where the product is taken in an arbitrary order.
Even though the element $\alpha_v$ might depend on the chosen order, its action on $\B{c_0}{n+1}$ does not, since the sets of chambers of $\B{c_0}{n+1}$
moved by the elements $\widetilde{\alpha_s}$ for distinct $s$, are disjoint.

Now we have an automorphism $\alpha_v$, for each chamber $v\in \Sp{c_0}n$, fixing \mbox{$\B{c_0}{n+1}\setminus \Sp v1$}.
Next, we want to vary $v$ along $\Sp{c_0}{n}$.
We claim that if $v_1, v_2 \in \Sp{c_0}n$ then $\alpha_{v_1}$ and $\alpha_{v_2}$ restricted to $\B{c_0}{n+1}$ have disjoint support.
The only case that remains to be checked is when $\Sp{v_1}1$ and  $\Sp{v_2}1$ have a chamber $c \in \Sp{c_0}{n+1}$ in common;
we want to show that both $\alpha_{v_1}$ and $\alpha_{v_2}$ fix $c$.

By Lemma~\ref{Squares}, there are $s \neq t$ in $W$ (with $st = ts$) such that $c\adj{s} v_1$ and $c\adj{t} v_2$.
By the definition of a legal coloring together with~\eqref{eq:h1=h2}, we have $h^1_t(c)=h^1_t(v_1)=h^2_t(v_1)=h^2_t(c)$.
Therefore $\alpha_{v_2}(c)=\widetilde{\alpha_{v_2, t}}(c)$ must fix $c$.
Similarly $h^1_s(c)=h^2_s(c)$, and so $\alpha_{v_1}$ fixes $c$.
This proves our claim, and hence $\alpha_{v_1}$ and $\alpha_{v_2}$ restricted to $\B{c_0}{n+1}$ have disjoint support for any two chambers $v_1$ and $v_2$ in $\Sp{c_0}n$.

We can now consider the product
\[ g_{n+1} = \prod_{v\in \Sp{c_0}n}\!\!\alpha_v= \prod_{v\in \Sp{c_0}n} \; \prod_{s \in S} \ \widetilde{\alpha_{v,s}} \in \Aut\Delta , \]
where the product is again taken in an arbitrary order.
By the previous paragraph, the action of $g_{n+1}$ on $\B{c_0}{n+1}$ is independent of the chosen order,
and the sets of chambers of $\B{c_0}{n+1}$ moved by distinct elements $\widetilde{\alpha_{v,s}}$ are disjoint.
Since every $\widetilde{\alpha_{v,s}}$ has the property that $h_s^2(c) = h_s^1(\widetilde{\alpha_{v,s}}(c))$ for all $c\in \Pa_{s,v}$,
we conclude that $h_s^2(c) = h_s^1(g_{n+1}(c))$ for all $c \in \Sp{c_0}{n+1}$, and therefore $g_{n+1} \in G_{n+1}$.

So we have extended the $g \in G_n$ to an element $g_{n+1}$ in $G_{n+1}$ agreeing with $g_n$ on the ball $\B{c_0}{n}$.
The sequence $g_0$, $g_1$, $\dots$ obtained by repeating this procedure hence converges to an element $g \in \Aut\Delta$ (with respect to the permutation topology).
From the construction and the definition of the sets $G_i$, the automorphism $g$ has the desired properties.
\end{proof}

\subsubsection{Weak legal colorings}

We define a weaker version of the legal colorings of Definition~\ref{colorings}. The goal will be to prove later that these two types of colorings play a similar role in the definition of the universal group.

\begin{defn}\label{weakcolor}
Let $s\in S$. An $s$-coloring $h_s$ is called a \emph{weak legal $s$-coloring} if the following holds:
\begin{itemize}
    \item[(W)]
        if $\Pa_1$ and $\Pa_2$ are two $s$-panels in a common $s$-tree-wall then for all $c\in \Pa_1$, we have $h_s(c)=h_s(\proj_{\Pa_2}(c))$.
\end{itemize}
\end{defn}

A legal coloring is in particular a weak legal coloring.
Conversely, the restriction of a weak legal coloring to a tree-wall is a legal coloring.

\begin{lem}\label{delta-finite}
Let $s\in S$. A weak legal $s$-coloring restricted to an $s$-tree-wall $\T$ is a legal coloring of $\ch\T$.
\end{lem}
\begin{proof}
If $c_1,c_2\in \ch\T$ are two $t$-adjacent chambers with $t\in S\setminus \{s\}$ then $c_1$ and $c_2$ lie in distinct parallel $s$-panels of $\T$ and $\proj_{\Pa_{s,c_1}}(c_2)=c_1$. Hence $h_s(c_1)=h_s(c_2)$.
\end{proof}

\begin{defn}\label{Gequivalent}
Let $s\in S$ and let $G\leq\Sym{Y_s}$. Two $s$-colorings $h_s^1$ and $h_s^2$ are said to be \emph{$G$-equivalent} if for every $s$-panel $\Pa$ there is $g\in G$ such that $h_s^1|_{\Pa}=g\circ h_s^2|_{\Pa}$.
\end{defn}

\begin{prop}\label{equivcolor}
Let $s\in S$ and $G\leq\Sym{Y_s}$ be a transitive permutation group.
Then every weak legal $s$-coloring is $G$-equivalent to some legal $s$-coloring.
\end{prop}
\begin{proof}
Let $h_s$ be a weak legal coloring of $\Delta$. We want to show that there is a legal coloring $h_s^\ell$ that is $G$-equivalent to $h_s$.

Let $c_0$ be a fixed chamber of $\Delta$. We will define $h_s^\ell$ recursively using the $s$-tree-wall distance.
Let $\T_0= \T_{s,c_0}$. For each chamber $c\in \ch{\T_0}$, define
\[ h_s^\ell(c)=h_s(c)=\id_{G}\circ h_s(c); \]
by Lemma~\ref{delta-finite}, the restriction of $h_s^\ell$ to $\ch{\T_0}$ is a legal coloring.

Assume that we have defined $h_s^\ell$ for all chambers of every $s$-tree-wall of $\Delta$ at tree-wall distance $\leq n$ from $\T_0$.
Let  $\T_2$ be an $s$-tree-wall at tree-wall distance $n+1$ from $\T_0$. By Proposition~\ref{tree-wall tree} there is a unique $s$-tree-wall $\T_1$ at tree-wall distance $1$ from $\T_2$ that is at tree-wall distance $n$ from $\T_0$.
By our recursion assumption, $h_s^\ell$ is already defined for all chambers of $\T_1$.
Pick some $d\in \proj_{\T_2}(\T_1)$, and let $c_1=\proj_{\T_1}(d)$.
Fix a $g\in G$ such that $g\circ h_s(d)=h_s^\ell(c_1)$ (which exists by transitivity) and define
\[ h_s^\ell(c)=g\circ h_s(c), \text{ for all } c\in \ch{\T_2}. \]
That is, we set $h_s^\ell(d)=h_s^\ell(c_1)$ and carry out the same permutation of the $s$-colors on each $s$-panel of $\T_2$.

We claim that
\begin{equation}\label{eq:hsl}
    h_s^\ell(d') = h_s^\ell\bigl(\proj_{\T_1}(d')\bigr) \quad \text{for all } d' \in \proj_{\T_2}(\T_1) . \tag{$*$}
\end{equation}
So let $d'\in \proj_{\T_2}(\T_1)$ and $c_1'=\proj_{\T_1}(d')$;
then by Proposition~\ref{4.53}\eqref{4.53.2} we have $d=\proj_{\T_2}(c_1)$. Hence from Proposition~\ref{4.53}\eqref{4.53.4} we obtain that  $\proj_{\Pa_{s,d}}(d')=d$ and $\proj_{{ \Pa}_{s,c_1}}(c_1')=c_1$.
Since $h_s$ is a weak legal coloring, we have $h_s(d)=h_s(d')$, and hence by construction $h_s^\ell(d) = h_s^\ell(d')$.
Moreover, since we already know that $h_s^\ell$ is a legal coloring on $\T_1$, we also have $h_s^\ell(c_1)=h_s^\ell(c_1')$.
Since $h_s^\ell(d)=h_s^\ell(c_1)$ and $g$ is fixed in each tree-wall, we conclude that $h_s^\ell(d')=h_s^\ell(c_1')$, proving the claim~\eqref{eq:hsl}.

With this procedure, we recursively define the map $h_s^\ell$ for all the chambers of $\Delta$.
It is obvious that $h_s^\ell$ is a coloring of the chambers of $\Delta$.

We will show that $h_s^\ell$ is a legal coloring.
So let $c_1$ and $c_2$ be chambers in a residue $\Ra$ of type $S\setminus \{s\}$; we have to show that $h_s^\ell(c_1) = h_s^\ell(c_2)$.
Since residues are combinatorially convex, the minimal galleries from $c_1$ to $c_2$ do not contain $s$-adjacent chambers;
in particular, $\proj_{\Pa_{s,c_1}}(c_2)=c_1$.
Let $\T_1$ and $\T_2$ be the $s$-tree-walls of $c_1$ and $c_2$, respectively.
If $\T_1$ coincides with $\T_2$ then we know that $h_s(c_1)=h_s(c_2)$ and thus $h_s^\ell(c_1)=g\circ h_s(c_1)=g\circ h_s(c_2)= h_s^\ell(c_2)$,
where $g\in G^s$ was the permutation used to define $h_s^\ell$ in $\T_1=\T_2$.

If $\T_1\neq \T_2$ then  $\T_1$ and $\T_2$ are both adjacent to the vertex in the $s$-tree-wall tree $\Gamma_s$ that corresponds to the residue $\Ra$.
Thus $\dtw(\T_1, \T_2)=1$.
Assume without loss of generality that
$n = \dtw(\T_{0}, \T_1)=\dtw(\T_{0}, \T_2)-1$,
where $\T_{0}$ is the $s$-tree-wall containing the base chamber $c_0$.
Then $\T_1$ is the unique $s$\dash tree-wall at tree-wall distance $n$ from $\T_{0}$ that is at tree-wall distance $1$ from~$\T_2$.
Therefore $h_s^\ell$ has been defined on $\T_2$ using the coloring in $\T_1$.

Let $d' = \proj_{\T_2}(c_1)$, and let $c_1' = \proj_{T_1}(d')$.
By the gate property (Proposition~\ref{Gate}), there is a minimal gallery from $c_1$ to $c_2$ through $c_1'$ and $d'$,
and the subgalleries from $c_1$ to $c_1'$ and from $d'$ to $c_2$ are completely contained in $\T_1$ and $\T_2$, respectively.
Since these galleries do not contain $s$-adjacent chambers because $c_1$ and $c_2$ are contained in $\Ra$ of type $S\setminus {\{s\}}$, we have $h_s^\ell(c_1) = h_s^\ell(c_1')$ and $h_s^\ell(d') = h_s^\ell(c_2)$.
Finally, by~\eqref{eq:hsl}, we also have $h_s^\ell(c_1') = h_s^\ell(d')$, and hence $h_s^\ell(c_1) = h_s^\ell(c_2)$.
We conclude that $h_s^\ell$ is a legal coloring, and by construction it is $G^s$-equivalent to $h_s$.
\end{proof}

\subsubsection{Directed legal colorings}\label{FC}

In this section we define a particular set of weak legal colorings.
As before, let $Y_s$ (for each $s\in S$) be the set of $s$-colors.
We additionally assume that each such set $Y_s$ contains a distinguished element $1_s$, or shortly 1 if no confusion can arise.

The key point of directed colorings is to get a set of colorings such that in every $s$-panel $\Pa$, the chamber of $\Pa$ closest to a fixed chamber has $s$-color $1$.
This will be particularly useful for studying chamber stabilizers.

\begin{prop}\label{directedc}
        Let $s\in S$ and let $c_0\in\ch\Delta$. Let $h_s$ be a weak legal $s$-coloring of $\Delta$ and $G\leq\Sym{Y_s}$ be a transitive permutation group.

        Then there exists a weak legal $s$-coloring $f_s$ of $\Delta$ which is $G$-equivalent to~$h_s$,
        such that $f_s(\proj_\Pa(c_0))=1_s$ for every $s$-panel $\Pa$.
\end{prop}
\begin{proof}
For each chamber $c \in \ch\Delta$, let $\T(c)$ be the unique $s$-tree-wall containing $c$.
For each  $s$-tree-wall $\T$, we fix an element $g_{\T}\in G$ such that
\[ g_\T\bigl(h_s(\proj_\T(c_0))\bigr)=1; \]
notice that such an element exists because $G$ is transitive.
We now define a coloring $f_s \colon \ch\Delta\to Y_s$ by
\[ f_s(c) := g_{\T(c)}(h_s(c)) \quad \text{for all } c \in \ch\Delta . \]
Let $\Pa$ be an arbitrary $s$-panel, and let $c = \proj_\Pa(c_0)$; we claim that $f_s(c) = 1$.
Indeed, let $\T$ be the $s$-tree-wall containing $\Pa$, and let $c_1=\proj_\T(c_0)$.
Then $\proj_{\Pa_{s,c_1}}(c)=c_1$, so $h_s(c) = h_s(c_1)$, and hence
\[ f_s(c) = g_\T(h_s(c)) = g_\T(h_s(c_1)) = 1 . \]

Next, we claim that $f_s$ is a weak legal $s$-coloring.
Indeed, fix some $s$-tree-wall $\T$; then by definition, $g_\T$ induces the same permutation on each $s$-panel of $\T$.
Since $h_s$ is a weak legal $s$-coloring, it satisfies property (W) from Definition~\ref{weakcolor} for each $s$-tree-wall $\T$,
and hence the same holds for $f_s$; we conclude that also $f_s$ is a weak legal $s$-coloring.
\end{proof}

\begin{defn}\label{deffc}
Let $s\in S$. A weak legal $s$-coloring as in Proposition~\ref{directedc} is called a \emph{directed legal $s$-coloring} of $\Delta$ \emph{with respect to} $c_0$.

In other words, if $f_s$ is a weak legal $s$-coloring and $c_0\in\ch\Delta$ then $f_s$ is called a directed legal $s$-coloring with respect to $c_0$ if for every chamber $c\in\ch\Delta$ at Weyl distance $w$ from $c_0$ with $l(w)<l(ws)$, we have $f_s(c)=1$.
\end{defn}

\begin{obs}\label{obs:form}
By definition, one can construct a weak legal coloring on any right-angled building $\Delta$.
Proposition~\ref{directedc} then implies that, given a chamber $c$ in $\Delta$, there exists a directed legal coloring of $\Delta$ with respect to $c$.
\end{obs}

\subsection{Directed right-angled buildings}\label{ss:directed}

In this section we want to describe each right-angled building in a standard way. This will be useful for the study of the maximal compact open subgroups of the universal group
in Section~\ref{se:cpt-open}.

We will construct a right-angled building with a directed coloring by means of reduced words in the Coxeter group and the set of colors, providing us with a concrete model for the objects we work with.

\begin{defn}\label{def:deltaF}
Let $(W,S)$ be a right-angled Coxeter system with set of generators \mbox{$S=\{s_i\}_{i\in I}$} and Coxeter diagram $\Sigma$.  For each $s\in S$, let $Y_s$ be a set with cardinality $q_s$ (with $q_s \geq 2$) with a distinguished element $1$.
We define an edge-colored graph $\Delta_D$ as follows.

The vertex set of $\Delta_D$  is
\[ \ch{\Delta_D} = \left\{ \begin{pmatrix} s_1&\cdots& s_n \\ \alpha_{1}&\cdots& \alpha_{n} \end{pmatrix} \Biggm|
    \begin{aligned}
        & s_1\cdots s_n \text{ is a reduced word in $M_S$ w.r.t.\@~$\Sigma$}, \\
        & \alpha_{i} \in Y_s \setminus \{1\} \text{ for each } i \in \{ 1,\dots,n \}
    \end{aligned}
    \right\} \]
(where we allow $n$ to be zero, yielding the empty matrix),
modulo the equivalence relation defined by
\[ \begin{pmatrix} s_1&\cdots& s_n \\ \alpha_{1}&\cdots& \alpha_{n} \end{pmatrix}\sim \begin{pmatrix} s_1'&\cdots& s_n' \\ \alpha_{1}'&\cdots& \alpha_{n}' \end{pmatrix} \]
if there exists a $\sigma \in \Rep{s_1\cdots s_n}$ such that
\[ s_{\sigma(1)} \cdots s_{\sigma(n)}=s_1' \cdots s_n' \text{ and } \alpha_{{\sigma(j)}}=\alpha_{j}', \text{ for all } j \in \{1, \dots n\}, \]
i.e., $s_1' \cdots s_n'$ is obtained from $s_1 \cdots s_n$ by performing elementary operations of type (2).
(We are denoting the vertex set of $\Delta_D$ by $\ch{\Delta_D}$ because we will prove later that $\Delta_D$ is a right-angled building.
We will already call the elements of $\ch{\Delta_D}$ chambers.)

We will now define adjacency in this graph.
Let $c=\begin{pmatrix} s_1&\cdots& s_n \\ \alpha_{1}&\cdots& \alpha_{n} \end{pmatrix}$ be an arbitrary chamber; then the neighbors of $c$ are:
\begin{enumerate}
    \item
        all chambers of the form $c'= \begin{pmatrix} s_1&\cdots& s_n & s_{n+1} \\ \alpha_{1}&\cdots& \alpha_{n} & \alpha_{n+1} \end{pmatrix}$.
        In this case, we declare $c$ and $c'$ to be $s_{n+1}$-adjacent.
    \item
        all chambers of the form $c'= \begin{pmatrix} s_1&\cdots& s_{n-1} & s_n \\ \alpha_{1}&\cdots& \alpha_{n-1} & \alpha_{n}' \end{pmatrix}$,
        where $\alpha_n'$ takes any value in $Y_s \setminus \{1, \alpha_n\}$.
        In this case, we declare $c$ and $c'$ to be $s_n$-adjacent.
    \item
        the unique chamber $c'= \begin{pmatrix} s_1&\cdots& s_{n-1} \\ \alpha_{1}&\cdots& \alpha_{n-1} \end{pmatrix}$.
        In this case, we declare $c$ and $c'$ to be $s_n$-adjacent.
\end{enumerate}
\end{defn}

\begin{prop}\label{itsbuild}
    The edge-colored graph $\Delta_D$ is a chamber system with index set $S$ and with prescribed thickness $(q_s)_{s\in S}$.
    Let $s\in S$, and let $\Pa$ be an $s$-panel of $\Delta_D$.
    Then the chambers of $\Pa$ are of the form
    \[ \left\{ \begin{pmatrix} s_1&\cdots& s_n \\ \alpha_{1}&\cdots& \alpha_{n} \end{pmatrix}\right\}
        \cup \left\{\begin{pmatrix} s_1&\cdots& s_n & s \\ \alpha_{1}&\cdots& \alpha_{n} & \alpha_s \end{pmatrix} \bigm| \alpha_s \in Y_s \setminus \{1\} \right\} . \]
\end{prop}
\begin{proof}
It is clear from the definition of $s$-adjacency that the relation ``$s$-adjacent or equal'' is an equivalence relation on the set of chambers of $\Delta_D$,
hence $\Delta_D$ is a chamber system.
The equivalence classes of this relation are precisely the $s$-panels, which are therefore of the required form.
Clearly, each $s$-panel has cardinality $|Y_s| = q_s$.
\end{proof}

Now we want to define colorings on the chamber system $\Delta_D$.

\begin{defn}\label{SFC}
Let $\Delta_D$ be as in Definition~\ref{def:deltaF}, and let $s\in S$.
We  define $F_s \colon \ch{\Delta_D} \to Y_s$ as follows.
Let $c = \begin{pmatrix} s_1&\cdots& s_n \\ \alpha_{1}&\cdots& \alpha_{n} \end{pmatrix}$ be a chamber of $\Delta_D$.
\begin{enumerate}
    \item
        If $l(s_1\cdots s_ns) > l(s_1\cdots s_n)$, then $F_s(c) := 1$.
    \item
        If $l(s_1\cdots s_ns) < l(s_1\cdots s_n)$, there exists a $\sigma \in \Rep{s_1\cdots s_n}$ such that $s_{\sigma(n)}=s$, i.e.,
        \[ s_1\cdots s_n = s_{\sigma(1)}\cdots s_{\sigma(n-1)}s \,\text{ in $W$}. \]
        Then
        \[ F_s(c) = F_s \begin{pmatrix} s_{\sigma(1)}&\cdots& s_{\sigma(n)} \\ \alpha_{\sigma(1)}&\cdots& \alpha_{\sigma(n)} \end{pmatrix}
            := \alpha_{\sigma(n)} \in \{2,\dots, q_s\}. \]
\end{enumerate}
We call $F_s$ the {\em standard $s$-coloring} of $\Delta_D$.
\end{defn}

Our next goal is to prove that given a semi-regular right-angled building $\Delta$ with a set of directed legal colorings with respect to a fixed chamber, one can construct
a color-preserving isomorphism to $\Delta_D$ (equipped with its standard colorings).
In particular, a pair consisting of a right-angled building and a set of directed legal colorings is unique up to isomorphism.

\begin{prop}\label{unirab}
Let $\Delta$ be a right-angled building of type $(W,S)$ with parameters $(q_s)_{s\in S}$,\
and let $(f_s)_{s\in S}$ be a set of directed colorings of $\Delta$ with respect to a fixed chamber $c_0\in\ch\Delta$.
Let $\Delta_D$ be the corresponding chamber system as in Definition~\textup{\ref{def:deltaF}},
and let $(F_s)_{s\in S}$ be its standard $s$-colorings as in Definition~\textup{\ref{SFC}}.

Then there is an isomorphism $\psi \colon \Delta \to \Delta_D$ such that $f_s(c)=F_s(\psi(c))$ for all $s\in S$ and all $c\in \ch\Delta$.
In particular, $\Delta_D$ is a right-angled building of type $(W,S)$, with thickness $(q_s)_{s\in S}$.
\end{prop}
\begin{proof}
We start by setting $\psi(c_0) = (\ ) \in \ch{\Delta_D}$.
Let $c\in\ch\Delta$ be arbitrary, and let $w= s_1\cdots s_n$ be a reduced word in $M_S$ representing $\delta(c_0, c)$ and let $\gamma = (c_0, c_1, \dots, c_n)$ be the minimal gallery of type $s_1\cdots s_n$ connecting the chambers $c_0$ and $c_n=c$.
Then we define
\[ \psi(c) = \left[ \begin{pmatrix} s_1 & \cdots & s_n \\ f_{s_1}(c_1) & \cdots & f_{s_n}(c_n) \end{pmatrix}\right]_\sim . \]

We claim that $\psi$ is well-defined, i.e.\@~that it is independent of the choice of the reduced representation of $w$.
Let $w'$ be another reduced representation of $\delta(c_0, c)$.
We know that there is a $\sigma\in \Rep{w}$ such that $\sigma.w=w'$ and moreover that $\sigma$ can be written as a product of elementary transpositions $\sigma=\sigma_k\dotsm \sigma_1$ such that $\sigma_i$ is a $\sigma_{i-1}\dotsm\sigma_1.w$-elementary transposition, for all $i \in \{1,\dots, n\}$.

Let $i\in \{1,\dots, n\}$.
Consider the reduced words $w_1=\sigma_{i-1}\dotsm \sigma_1.w$ and $w_2=\sigma_i.w_1$.
These two words only differ in two generators $r_j$ and $r_{j+1}$ which are switched by $\sigma_i$.
Let $\gamma_1=(e_0, \dots, e_n)$ and $\gamma_2=(d_0, \dots, d_n)$ be the minimal galleries between $c_0=e_0=d_0$ and $c=e_n=d_n$ corresponding to the word $w_1$ and $w_2$ respectively.

We now prove that the colors $f_{r_t}(e_t)$ and $f_{r_t}(d_{\sigma_i (t)})$ agree for all $t\in \{0, \dots , n\}$.  Note that the galleries $\gamma_1$ and $\gamma_2$ differ only in the chamber $e_j$, and that $\sigma_i$ is a transposition switching $j$ and $j+1$. Therefore we only need to check $t = j$ and $t = j +1$.
We have
\[e_{j-1} \adj{r_j} e_j \adj{r_{j+1}} e_{j+1} \ \text{ and } \ e_{j-1}=d_{j-1}\adj{r_{j+1}} d_j\adj{r_j} d_{j+1}=e_{j+1}.\]
As $(f_s)_{s\in S}$ is a set of direct colorings with respect to $c_0$ we infer that $f_{r_j}(e_j) = f_{r_j}(d_{j+1}) = f_{r_j}(d_{\sigma_i (j)})$ and $f_{r_{j+1}}(e_{j+1}) = f_{r_{j+1}}(d_j) =  f_{r_{j+1}}(d_{\sigma_i(j+1)})$, which implies that the colors indeed agree.

From this we deduce that
\[ \begin{pmatrix} &w_1&  \\ f_{r_1}(e_1)&\cdots& f_{r_n}(e_n) \end{pmatrix}\sim \begin{pmatrix} &w_2& \\ f_{r_1}(d_1)&\cdots& f_{r_n}(d_n) \end{pmatrix}, \]
which allows us to conclude that
\[ \begin{pmatrix} s_1&\cdots& s_n \\ f_{s_1}(c_1)&\cdots& f_s(c_n) \end{pmatrix}\sim \begin{pmatrix} s_{\sigma(1)}&\cdots& s_{\sigma(n)} \\ f_{s_{\sigma(1)}}(c_1')&\cdots& f_{s_{\sigma(n)}}(c_n') \end{pmatrix}, \]
where $(c_0', \dots, c_n')$ is the minimal gallery between $c_0$ and $c$ corresponding to $w'$.
This proves that $\psi$ is independent of the choice of the reduced representation of $w$ and in particular that it is well defined.

Next, we claim that $\psi$ is color-preserving, i.e.\@ that $f_s(c)=F_s(\psi(c))$ for all $s\in S$ and all $c\in \ch\Delta$.
Let $c$ be a chamber of $\Delta$ at Weyl distance $w$ from $c_0$, and $s_1\cdots s_n$ a reduced word in $M_S$ with respect to $\Sigma$ representing $w$.
By construction of $\psi$ and definition of $F_{s_n}$, we have $f_{s_n}(c)=F_{s_n}(\psi(c))$.

Let $s\in S\setminus \{s_n\}$. If $l(s_1\cdots s_n)< l(s_1\cdots s_ns)$, then by definition of a directed coloring and $F_S$ we have $f_s(c)=1$ and $F_s(\psi(c))=1$.
If $l(s_1\cdots s_n)> l(s_1\cdots s_ns)$, then we could have picked a reduced representation $s_1 \cdots s_n$ of $w$ where $s_n = s$, which reduces the problem to a previous case.

Next, we show that $\psi$ is a bijection. It is clear from the definition that $\psi$ is surjective. It remains to prove that it is injective.
Let $c_1, c_2\in \ch\Delta$ be distinct chambers. If $c_1$ and $c_2$ are at distinct Weyl distances from $c_0$ then by definition $\psi(c_1)\neq \psi(c_2)$.

Assume that $s_1\cdots s_n$ is a reduced representation of the Weyl distance from $c_0$ to both $c_1$ and $c_2$.
Let $i$ be the minimal number in $\{1, \dots, n\}$ such that $v_1^i\neq v_2^i$, where $v_j^i$, for $j\in\{1,2\}$, is the unique chamber in $\Delta$ at Weyl distance $s_1\cdots s_i$ from $c_0$ and at Weyl distance $s_{n}\cdots s_{i+1}$ from $c_j$.
Then $v_1^{i-1}=v_2^{i-1}$, which means that $v_1^i$ and $v_2^i$ are in the same $i$-panel of $\Delta$.
So $f_{s_i}(v_1^i)\neq f_{s_i}(v_2^i)$. This implies, by definition of the map $\psi$, that $\psi(c_1)\neq \psi (c_2)$.
Therefore $\psi$ is a bijection.

Finally, we show that $\psi$ is  a homomorphism.
If $c_1$ and $c_2$ are $s$-adjacent chambers in $\Delta$, for $s\in S$, then
\begin{enumerate}
\item either $c_1$ is at Weyl distance $s_1\cdots s_n$ from $c_0$ and $c_2$ is at Weyl distance $s_1\cdots s_n s$ from $c_0$, or vice-versa. Then by definition of $\psi$ and the definition of adjacency in $\Delta_D$, the chamber $\psi(c_1)$ is $s$\dash adjacent to $\psi(c_2)$.

\item or $c_1$ and $c_2$ are both at Weyl distance $s_1 \cdots s_n s$ from $c_0$ and $f_s(c_1)\neq f_s(c_2)$. Again we conclude that $\psi(c_1)$ is $s$-adjacent to $\psi(c_2)$, this time by the second point of the definition of adjacency in $\Delta_D$.
\end{enumerate}
This shows that $\psi$ is an isomorphism from $\Delta$ to $\Delta_D$ respecting the set of colorings $(f_s)_{s\in S}$, proving the proposition.
\end{proof}

\begin{defn}\label{frab}
We call the building $\Delta_D$  the \emph{directed right-angled building} of type $(W,S)$ with prescribed thickness $(q_s)_{s\in S}$.
\end{defn}

\begin{rem}
    The directed legal coloring $F_s$ is not a legal $s$-coloring unless $s$ commutes with all elements of $S$
    (in which case every weak legal $s$-coloring is a legal $s$-coloring by Lemma~\ref{delta-finite}).
\end{rem}

\begin{prop}\label{prop:g0directed}
Let $\Delta$ be the directed right-angled building of type $(W,S)$ with prescribed thickness $(q_r)_{r \in S}$.
Let $s \in S$ and fix an $s$-tree-wall $\T$ in $\Delta$.
Let $g$ be a permutation of $Y_s \setminus \{1\}$.

Consider the following map $g_\T$ on the set of chambers of $\Delta$. Let $c$ be a chamber represented by the matrix
\[ \begin{pmatrix} s_1&\cdots& s_n \\ \alpha_{1}&\cdots& \alpha_{n} \end{pmatrix}. \]
If there is an $i \in \{1, \dots, n \}$ such that $s_i =s$ and the chambers represented by
\begin{equation}\label{eq:c1}
    \begin{pmatrix} s_1&\cdots& s_{i-1} \\ \alpha_{1}&\cdots& \alpha_{i-1} \end{pmatrix} \text{ and }
    \begin{pmatrix} s_1&\cdots& s_{i} \\ \alpha_{1}&\cdots& \alpha_{i} \end{pmatrix} \text{ are in $\T$},
\end{equation}
then $g_\T$ maps $c$ to the chamber represented by
\begin{equation}\label{eq:c2}
    \begin{pmatrix} s_1&\cdots& s_{i-1} & s & s_{i+1} & \cdots & s_n   \\ \alpha_{1}&\cdots& \alpha_{i-1} & g(\alpha_{i})  & \alpha_{i+1} & \cdots & \alpha_n \end{pmatrix}.
\end{equation}
If there is no such $i$, then $g_\T$ fixes $c$.

Then the map $g_\T$ constructed in this way is an automorphism of $\Delta$.
Moreover, if $g,h$ are two permutations of $Y_s \setminus \{ 1 \}$, then $g_\T h_\T = (gh)_\T$.
\end{prop}
\begin{proof}
Notice that there can be at most one index $i$ with $s_i = s$ for which~\eqref{eq:c1} holds, since $s_1 \dotsm s_n$ is a reduced word and $\T$ is an $s$-tree-wall.

We start by showing that $g_\T$ is well defined, i.e., that our description of $g_\T$ is independent of the representation of the chamber $c$.
Assume that $c$ is represented as in the statement of the proposition.
It suffices to look at equivalence by a single elementary transposition; so assume that there is a $j \in \{1, \dots, n-1 \}$ such that $s_j$ and $s_{j+1}$ commute.
The only non-trivial cases are when condition~\eqref{eq:c1} is satisfied and either $j=i$ or $j=i-1$.
Assume that $j=i$ (the other case can then be handled analogously);
so $s_i = s$ commutes with $s_{i+1}$.

Note that the chambers represented by
\[
\begin{pmatrix} s_1&\cdots& s_{i-1} & s_{i+1} \\ \alpha_{1}&\cdots& \alpha_{i-1} &\alpha_{i+1} \end{pmatrix}
\text{ and }
\begin{pmatrix} s_1&\cdots& s_{i-1} & s_{i+1} & s_i \\ \alpha_{1}&\cdots& \alpha_{i-1} &\alpha_{i+1} & \alpha_i \end{pmatrix}
\]
are still contained in $\ch\T$ by Corollary~\ref{4.39}.
The image under $g_\T$ of $c$ using its representation obtained by applying the elementary transposition is therefore represented by
\[
\begin{pmatrix}
s_1&\cdots& s_{i-1} & s_{i+1} & s_i & s_{i+2} & \cdots & s_n \\
\alpha_{1}&\cdots& \alpha_{i-1} &\alpha_{i+1} & g(\alpha_i) & \alpha_{i+2} & \cdots & \alpha_n
\end{pmatrix},
\]
which is an equivalent representation of the image as in~\eqref{eq:c2}.
This proves that the map $g_\T$ is indeed well defined.

In order to show that $g_\T$ is a homomorphism of $\Delta$, we have to show that $t$\dash panels are mapped to $t$-panels, for all $t \in S$.
However, this is now immediately clear from the description of a $t$-panel in Proposition~\ref{itsbuild}.

Finally observe that each $g_\T$ is invertible, with inverse $(g^{-1})_\T$; this implies that $g_\T$ is an automorphism of $\Delta$.
The last statement is also clear.
\end{proof}

\section{Universal group of a right-angled building}\label{se:univ}

The goal of this section is to extend the concept of universal groups defined for regular trees by Burger and Mozes to the more general setting of right-angled buildings.
These groups will depend on the choice of a transitive permutation group for each $s \in S$.

\subsection{Definition}

Let $(W,S)$ be a right-angled Coxeter system with Coxeter diagram $\Sigma$ with index set $I$ and set of generators $S=\{s_i\}_{i\in I}$.
 For all $s\in S$, let $q_s$ be a cardinal number and $Y_s$ be a set of size $q_s$.
 Consider the right-angled building $\Delta$ of type $(W,S)$ with parameters $(q_s)_{s\in S}$, which is unique up to isomorphism (see \cite[Proposition~1.2]{HP2003}).

\begin{defn}\label{universalgroup}
For each $s \in S$, let $G^s \leq \Sym{Y_s}$ be a transitive permutation group and $h_s \colon \ch\Delta\to Y_s$ be a (weak) legal coloring of $\Delta$.
We define the \emph{universal group} of $\Delta$ \emph{with respect to} the groups $(G^s)_{s\in S}$ as
\begin{align*}
U &= U((G^s)_{s\in S}) \\
&= \{ g\in \Aut \Delta \mid (h_s|_{\Pa_{s,gc}})\circ g\circ (h_s|_{\Pa_{s,c}})^{-1} \in G^s, \text{ for all } s\in S, \\
    & \hspace*{25ex} \text{ all $s$-panels } \Pa_s, \text{ and for all chambers } c\in \Pa_s \},
\end{align*}
where $\Pa_{s,c}$ is the $s$-panel containing $c\in \ch\Delta$.
\end{defn}

If  the group $G^s$ equals $\Sym{Y_s}$, for all $s\in S$, then $U$ is the group $\Aut\Delta$ of all type preserving automorphisms of the right-angled building since we are assuming that the groups $G^s$ are transitive.

\begin{lem}\label{sameuniversal}
For each $s\in S$, let $(h_s)_{s\in S}$ and $(h_s')_{s\in S}$ be two $G^s$-equivalent colorings.
Then the universal groups constructed using $(h_s)_{s\in S}$ and $(h_s')_{s\in S}$ coincide.
\end{lem}
\begin{proof}
Let $g \in U^{h_s}$ be an element of the universal group constructed with the coloring $h_s$.
Let $c\in \ch\Delta$ and $\Pa$ be the $s$-panel of $c$.
By definition we know that $h_s\circ g\circ (h_s)^{-1}|_{\Pa} \in G^s$.
Further, since $h_s$ and $h_s'$ are $G^s$-equivalent colors we know that $h_s|{\Pa}=g_1\circ h_s'|_{\Pa}$ and $h_s|{{\Pa}_{s,gv}}=g_2\circ h_s'|_{{\Pa}_{s,gv}}$ with $g_1, g_2\in G^s$.
Thus
\[ h_s\circ g\circ (h_s)^{-1}|_{\Pa}=g_2\circ h_s'\circ g\circ (h_s')^{-1}\circ g_1^{-1}|_{\Pa}\in G^s, \]
which implies that $h_s'\circ g\circ (h_s')^{-1}|_{\Pa}\in G^s$.
Hence $g$ is an element of the universal group $U^{h_s'}$ constructed using $h_s'$.
Exchanging the colorings in the reasoning, we obtain that the two universal groups coincide.
\end{proof}

This lemma implies, in view of Proposition~\ref{equivcolor}, that in the definition of the universal group restricting to legal colorings, weak legal colorings or directed legal colorings with respect to a fixed chamber, all being $G^s$-equivalent (see Proposition~\ref{equivcolor} and Observation~\ref{obs:form}), all yield the same universal group.

\begin{rem}
The definition of a universal group for a right-angled building also makes sense when the groups $G^s$ are not transitive.
For instance, if the groups $G^s$ are all trivial and $\Delta$ is locally finite, then $U$ is a lattice in $\Aut\Delta$ since it acts freely and cocompactly on $\Delta$.
However, only when all groups $G^s$ are transitive, the universal group will be chamber transitive, which is a property we will need very often.
\end{rem}

\subsection{Basic properties}
In this section we gather some basic properties concerning universal groups for right-angled buildings.

We start by looking at the action on panels.
\begin{defn}
Let $H\leq \Aut\Delta$ and $\Pa$ be a panel of $\Delta$.
We define the \emph{local action} of $H$ at the panel $\Pa$ as the permutation group formed by restricting the action of $H_{\{\Pa\}}$ (the setwise stabilizer of $\Pa$ in $H$) to the panel $\Pa$.
\end{defn}

\begin{lem}\label{lem:local}
The local action of the universal group $U$ on an $s$-panel ($s \in S$) is isomorphic to the transitive group $G^s$.
\end{lem}
\begin{proof}
Consider some chamber $c$ and pick an $r\in S$.
The chambers in the $r$-panel $\Pa :=\Pa_{r,c}$ containing $c$ are parametrized by $Y_r$ via the coloring.
From the definition of the universal group it is immediate that the local action on the $r$-panel is a subgroup of $G^r$.

We will now show that this local action is in fact $G^r$.
Let $g_r \in G^r$.
Let $c'$ be the chamber in $\Pa$ with color $g_r \circ h_r(c)$.
Denote by $(h_s')_{s \in S}$ the set of legal colorings obtained from $(h_s)_{s \in S}$ by replacing the coloring $h_r$ by $g_r^{-1} \circ h_r$ and leaving the other colorings unchanged. Note that $h_r'$ is again a legal coloring, as  it will still satisfy the defining property (L) of legal colorings.
As $h_s(c)$ equals $h_s'(c')$ for every $s \in S$, we can apply Proposition~\ref{existenceconjugate} to find an automorphism $g \in \Aut\Delta$ mapping $c$ to $c'$.

The automorphism $g$ acts locally as the identity for $t$-panels where $t \in S \setminus \{ r\}$, and as $g_r$ on $r$-panels.
Hence $g$ is an element of the universal group and it has the desired action on the panel $\Pa_{r,c}$, whence the claim.
\end{proof}

Assume that the set $Y_s$ contains an element called $1$, and let $G^s_0$ be the stabilizer of this element in $G^s$.

\begin{prop}\label{prop:nicegenerators}
Let $\Delta$ be the directed right-angled building with prescribed thickness $(q_r)_{r \in S}$ (see Definition~\ref{frab}) and base chamber $c_0$. Consider the universal group $U$ with respect to the standard colorings $(f_s)_{s\in S}$ of $\Delta$ directed with respect to $c_0$.
Let $\T$ be an $s$\dash tree wall of $\Delta$.

For each $g \in G^s_0$, let $g_\T$ be as in Proposition~\ref{prop:g0directed}.
Then $\{ g_\T \mid g \in G^s_0 \}$ is a subgroup of $U$ fixing the chambers of the $s$-wing with respect to $\T$ containing the chamber $c_0$.
This subgroup acts locally as $G^s_0$ on each $s$-panel of the tree-wall~$\T$.
\end{prop}
\begin{proof}
Let $g \in G^s_0$ and let $g_\T \in \Aut\Delta$ be as in Proposition~\ref{prop:g0directed}.
We will first prove that $g_\T$ is an element of the universal group $U$.
Let $\Pa$ be a $t$-panel for some $t \in S$.
It is clear that the automorphism $g_\T$ fixes the colorings of the chambers in $\Pa$ unless $t = s$.

Assume that $t=s$. In the case that $\Pa$ is a panel in the $s$-tree-wall $\T$, we have $f_s|_{g_\T.\Pa}\circ g_\T \circ f_s|_{\Pa}=g\in G^s_0$.
If $\Pa \not\in \T$, then this permutation of the colors is the identity in $G^s$.
Hence $g_\T \in U$.

Next, we prove that $g_\T$ fixes the wing $X_s(\proj_\T(c_0))$.
Let \[c=\begin{pmatrix} s_1&\cdots& s_{n} \\ \alpha_{1}&\cdots& \alpha_{n} \end{pmatrix} \]
be a chamber in $X_s(\proj_\T(c_0))$, so that $\proj_\T(c) = \proj_\T(c_0)$.
If there were an $i \in \{1, \ldots, n\}$ such that $s_i=s$ and
\[ \begin{pmatrix} s_1&\cdots& s_{i-1} \\ \alpha_{1}&\cdots& \alpha_{i-1} \end{pmatrix} \text{ and }
    \begin{pmatrix} s_1&\cdots& s_{i} \\ \alpha_{1}&\cdots& \alpha_{i} \end{pmatrix} \text{ are in $\T$}, \]
then
\[\proj_\T(c)= \begin{pmatrix} s_1&\cdots& s_{i-1} & s \\ \alpha_{1}&\cdots& \alpha_{i-1} &\alpha_i \end{pmatrix}
    \neq \begin{pmatrix} s_1&\cdots& s_{i-1} \\ \alpha_{1}&\cdots& \alpha_{i-1} \end{pmatrix}=\proj_\T(c_0),\]
which is a contradiction; hence there is no such $i$, and therefore $g_\T$ fixes $c$ by definition.

By the last statement of Proposition~\ref{prop:g0directed}, the set $G = \{ g_\T \mid g \in G^s_0 \}$ forms a group.
By construction, $G$ acts locally as $G^s_0$ on each $s$-panel of~$\T$.
\end{proof}

We will prove in the following proposition that different choices of legal colorings give rise to conjugate subgroups of $\Aut\Delta$;
this will allow us to omit an explicit reference to the colorings in our notation for the universal groups.

\begin{prop}[Properties of $U$]\label{propU}
Let $\Delta$ be a right-angled building with prescribed thickness $(q_s)_{s\in S}$.
Let $U$ be the universal group of $\Delta$ with respect to the finite transitive permutation groups $(G^s\leq\Sym{Y_s})_{s\in S}$.
Then the following hold.
\begin{enumerate}
\item The subgroup $U \subset \Aut\Delta$ is independent of the choice of the set of legal colorings up to conjugacy.
\item $U$ is a closed subgroup of $\Aut \Delta$.
\item $U$ is a chamber-transitive subgroup of $\Aut\Delta$.
\item\label{U:universal} $U$ is universal for $(G^s)_{s\in S}$, i.e., if $H$ is a closed chamber-transitive subgroup  of $\Aut \Delta$ for which the local action on each $s$-panel is permutationally isomorphic to the group $G^s$, for all $s\in S$, then  $H$ is conjugate in $\Aut\Delta$ to a subgroup of $U$.
\item If $\Delta$ is locally finite, then $U$ is compactly generated.
\end{enumerate}
\end{prop}

\begin{proof}
\begin{enumerate}
	\item
	Let $(h^1_s)_{s\in S}$ and $(h^2_s)_{s\in S}$ be distinct sets of legal colorings.
	We want to show that the universal groups $U^{1}$ and $U^{2}$ defined using the legal colorings $(h_s^1)_{s\in S}$ and $(h_s^2)_{s\in S}$, respectively, are conjugate in $\Aut\Delta$.
	By Proposition~\ref{existenceconjugate} we know that there exists $g \in \Aut\Delta$ such that $h_s^2=h_s^1 \circ g$ for all $s\in S$.
	If $a\in U^2$, we know that
\[ \begin{array}{ll}
 h_s^2 \circ a \circ (h_s^2)^{-1} \in G^s
\iff h_s^1\circ g \circ a \circ \inv g \circ (h_s^1)^{-1} \in G^s, \ \text{ for all } s\in S.\\
\end{array} \]
	Hence $a^g \in U^1$ and $U^1$ and $U^2$ are conjugate in $\Aut\Delta$.
	\item
	To prove Statement 2, we will show that $\Aut \Delta \setminus U$ is open.
	Consider \mbox{$a\in \Aut \Delta \setminus U$}.
	Then there exists $s\in S$, an $s$-panel $\Pa_s$ and $v\in \Pa_s$ such that
\[ h_s|_{\Pa_{s, av}}\circ a \circ (h_s|_{\Pa_{s,v}})^{-1} \not\in G^s. \]
	But then the set $\{ a'\in \Aut \Delta \mid a'|_{\Pa_{s,v}}=a|_{\Pa_{s,v}} \}$ is contained in $\Aut \Delta \setminus U$ and it is a coset of the stabilizer of $\Pa_{s,v}$. Hence $\Aut \Delta \setminus U$ is open.
	\item
	Next we show that $U$ is a chamber-transitive group.
	Since $\Delta$ is connected, it is enough to prove the result for two adjacent chambers.
	Let $v_1$ and $v_2$ be two adjacent chambers (vertices) in the building $\Delta$,
	i.e. there exists an $s\in S$ such that $v_1$ and $v_2$ are in the same $s$-panel $\Pa$.
	By Lemma~\ref{lem:local}, the action of $U_{\{ \Pa\}}$ on $\Pa$ is isomorphic to the group $G^s$.
	Since $G^s$ is assumed to be transitive, there is an element in $G^s$ mapping $v_1$ to $v_2$.
	Hence there is an element $g\in U$ such that $gv_1=v_2$.
	Thus $U$ is a chamber-transitive subgroup of $\Aut\Delta$.
	\item
	Now we prove that the group $U$ is universal.
	Let $H$ be a closed chamber-transitive subgroup  of $\Aut \Delta$ for which the local action on each $s$-panel is permutationally isomorphic to the group $G^s$, for all $s\in S$.
	We will construct weak legal colorings $(h_s)_{s\in S}$ such that $H$ is a subgroup of $U^{h_s}$, i.e., the universal group defined using the set of weak legal colorings $(h_s)_{s\in S}$.

	Let us fix $c_0 \in \ch\Delta$ and  $s\in S$.
	We choose a bijection $h_s^0  \colon  \Pa_{s, c_0} \to Y_s$ such that
\[ h_s^0 \circ H_{\{ \Pa_{s,c_0}\}} \circ (h_s^0)^{-1}|_{\Pa_{s,c_0}} = G^s, \]
where $H_{\{ \Pa_{s,c_0}\}}$ is the setwise stabilizer of $\Pa_{s,c_0}$ in $H$.

	Let $\T_0=\T_{s,c_0}$.
	For each $s$-panel $\Pa$ in the $s$-tree-wall $\T_{0}$ we define
\[ h_s(c)=h_s^0(\proj_{\Pa_{s,c_0}}(c)), \text{ for every } c\in \Pa. \]
	With this procedure we have colored with  a color from $Y_s$ all the vertices in $\T_{0}$.
	Further, since the chambers of parallel panels are in bijection through the projection map, we have
$h_s\circ H|_{\Pa} \circ (h_s|_{\Pa})^{-1} =G^s \text{ for each  panel } \Pa \text{ of } \T_0. $

	Assume that the coloring $h_s$ is defined in the $s$-tree-walls of $\Delta$ at tree-wall distance $\leq n-1$ from $\T_{0}$.

	Let $\T$ be an $s$-tree-wall at tree-wall distance $n$ from $\T_{0}$.
	Fix $c\in\ch\T$. Define a bijection  $h_s^n \colon  \Pa_{s,c} \to Y_s$ such that
    $h_s^n \circ H_{\{ \Pa_{s,c}\}}  \circ (h_s^n)^{-1} = G^s$.
    For all chambers $v\in \ch\T$, we define \[ h_s(v)=h_s^n(\proj_{\Pa_{s,c}}(v)). \]

	In this fashion  we color all the chambers of the building with the colorings $(h_s)_{s\in S}$.
	Now we have to prove that $h_s$ is a weak legal coloring, for each $s\in S$.
	It is clear that $h_s$ is a coloring since it was defined as a bijection in a panel of each $s$-tree-wall
    and parallel panels in $\T$ are in bijection through the projection map.

	To prove that $h_s$ is a weak legal coloring, let $\Pa_1$ and $\Pa_2$ be $s$-panels in a common $s$-tree-wall $\T$.
	Let $c_1 \in \Pa_1$ and $c_2=\proj_{\Pa_2}(c_1)$.
	Let $c=\proj_{\Pa}(c_1)$, where $\Pa$ is the $s$-panel of $\T$ that was used to define $h_s$ in the recursive process.
	Then, by Proposition~\ref{3panels}, $\proj_{\Pa}(c_2)=c$ and hence $h_s(c_1)=h_s(c)=h_s(c_2)$.
	Thus $h_s$ is a weak legal coloring.

	Further, if $g\in H$ then $h_s\circ g\circ (h_s|_{\Pa})^{-1} \in G^s$ for every $s$-panel $\Pa$ and every $s\in S$, by the construction of the weak legal colorings $(h_s)_{s\in S}$.
	Therefore $H$ is a subgroup of $U^{h_s}$.
	Hence $U$ is the largest vertex-transitive closed subgroup of $\Aut\Delta$ which acts locally as the groups $(G^s)_{s\in S}$.
	\item
	Finally, we prove that $U$ is compactly generated when $\Delta$ is locally finite.
	Let $c\in \ch\Delta$. The group $U_c$ is a compact open subgroup of $U$ by the definition of the permutation topology.
	Let $\{v_1, \dots, v_n\}$ be the set of chambers of $\Delta$ adjacent to $c$.
	Since $U$ is chamber-transitive, for each $v_i$, there exists an element $g_i \in U$ such that $g_iv_i=c$.
    Let $G = \{g_1, \dots, g_n \}$.

	Let $g\in U$. We claim that there exists some $g'\in \langle S\rangle$ such that $g'gc=c$.
	This is proved by induction on the discrete distance from $c$ to $gc$.
	If $\dist c {gc}=1$, this follows from the definition of $G$.
	Assume that the claim holds if $\dist c {gc} \leq n$.
	If $\dist c{gc}=n+1$, let $\gamma=(c, v, \dots, gc)$ be a minimal gallery from $c$ to $gc$.
	We have $\dist c v =1$ therefore there is $\overline{g}\in G$ such that $\overline{g}v=c$.
	As $\dist v{gc}=n$ we have $\dist c {\overline{g}gc} =n$.
	Hence there exists $g^*\in \langle G \rangle$ such that $g^*\overline{g}gc=c$ and $g^*\overline{g}\in \langle G \rangle$.
	But then $g^*\overline{g}g\in U_c$.
	So we conclude that the compact subset $U_c \cup G$ generates the group $U$.
\qedhere
\end{enumerate}
\end{proof}

\begin{rem}
    As the referee observed, the maximal subgroups with a prescribed local action from
    Proposition~\ref{propU}\eqref{U:universal} could be compared with minimal subgroups with the same prescribed local action.

    More precisely, let $U$ be the universal group for $(G^s)_{s \in S}$ (still assumed transitive)
    and for each $s \in S$, let $G^s_0$ be a point stabilizer in $G^s$, as before.
    For each {\em spherical} subset $J \subseteq S$, we define
    \[ P_J := \prod_{s \in J} G^s \times \prod_{s \in S \setminus J} G^s_0 . \]
    Let $B := P_\emptyset$.
    Then the groups $(P_J)$, where $J$ runs over all spherical subsets of $S$, form a simple complex of groups.
    Let $M$ denote the direct limit of this complex (which is called the {\em amalgamated sum} in \cite{Tits-amal}),
    and let $C$ denote the chamber system $C(M; B; (P_s)_{s \in S})$ as in \cite{Tits-amal}.
    Then we can apply \cite[Section 2.1 and Theorem 1]{Tits-amal} (see also \cite[Theorem 4.2]{DavisCat0})
    to conclude that $C$ is a building (which must then be isomorphic to $\Delta$),
    that $M$ acts chamber-transitively on $\Delta$,
    and that the stabilizers of the spherical $J$-residues are conjugate to $P_J$.
    In particular, the induced action on each $s$-panel is permutationally isomorphic to~$G^s$.
    This group $M$ is then a minimal subgroup with prescribed local actions $(G^s)_{s \in S}$;
    by Proposition~\ref{propU}\eqref{U:universal}, $M$ is contained in $U = U((G^s)_{s\in S})$.


    When $G^s$ is finite for each $s$, this group $M$ is a cocompact lattice of the universal group $U$.

    (We thank Pierre-Emmanuel Caprace for his help in formulating this remark.)
\end{rem}

\subsection{Extending elements of the universal group}

Let $\Delta$ be a right-angled building of type $(W, S)$ with parameters $(q_s)_{s\in S}$.
For each $s\in S$, let \mbox{$h_s \colon \ch\Delta \to Y_s$} be a legal coloring and let $G^s \leq \Sym{Y_s}$ be a transitive permutation group.
Consider the universal group $U$ of $\Delta$ with respect to the groups $(G^s)_{s\in S}$.

\begin{defn}
If $B$ is a connected subset of the building $\Delta$ containing $c_0$, and $\T$ is an $s$-tree-wall of the building such that $B$ is not entirely in one wing of $\T$, we say that $B$ \emph{crosses} $\T$.
Let $c$ be the projection of $c_0$ on such a tree-wall $\T$.
We call the Weyl distance $\delta(c,c_0)$ the \emph{distance between $c_0$ and $\T$}.
\end{defn}

Before proceeding we mention the following technical lemma.

\begin{lem}[{\cite[Lemma 5.3]{Caprace}}]\label{CapraceP}
Let $n>0$ be an integer, let $C, W$ be sets and $\delta \colon C^n \to W$ be a map.
Let $G$ denote the group of all permutations $g\in \Sym C$ such that $\delta(g.x_1, \dots , g.x_n)=\delta(x_1, \dots ,x_n)$ for all $(x_1, \dots , x_n)\in C^n$.
Let moreover $(V_s)_{s\in S}$ be a collection of groups indexed by a set $S$, and for all $s\in S$, let $\varphi_s \colon  V_s \to G$ be an injective homomorphism such that for all $s\neq r$, the subgroups $\varphi(V_s)$ and $\varphi(V_r)$ have disjoint supports.
Then there is a unique homomorphism
\[ \varphi  \colon  \prod_{r\in S} V_r \to G \]
such that $\varphi_s\circ \iota_s =\varphi_s$ for all $s\in S$, where $\iota_s  \colon  V_s \to \prod_{r\in S} V_r$ is the canonical inclusion.
\end{lem}

Let us now consider, for a fixed $w \in W$ and $s \in S$ the collection $T_{w,s}$ of $s$-tree-walls at distance $w$ from $c_0$ (note that this collection may be empty).
For each tree-wall $\T \in T_{w,s}$ we have, by Proposition~\ref{prop:nicegenerators}, a subgroup $H$ of the universal group $U$, acting faithfully and locally like $G_0^s$ on each $s$-panel of $\T$, and fixing the chambers in the wing of $\T$ containing $c_0$.

Note that these subgroups for different tree-walls in $T_{w,s}$ have disjoint supports.
To see this, consider two  different tree-walls $\T_1$ and $\T_2$ in  $T_{w,s}$. The $s$-wing with respect to $\T_1$ containing $c_0$ contains every $s$-wing with respect to~$\T_2$, except the one containing $c_0$ (which can be seen from the $s$-tree-wall tree), implying that the supports are indeed disjoint.
Hence we can apply Lemma~\ref{CapraceP} and consider the action of the product of these subgroups of the universal group (one for each tree-wall in $T_{w,s}$), which we denote with slight abuse of notation by~$\prod_{T_{w,s}} G_0^s$.

\begin{lem}\label{generators}
Let $\Delta$ be a semi-regular right-angled building of type $(W,S)$ with parameters $(q_s)_{s\in S}$.
Fix a chamber $c_0$ of $\Delta$.
Let $K$ be a connected subset of $W$ (as a Coxeter complex) containing the identity.
Let $B$ be the set of chambers $c$ in $\ch\Delta$ such that $\delta(c_0, c) \in K$.

Then
\[ \left\langle \bigg\{ \prod_{T_{w,s}} G_0^s \bigg\} \middle| s\in S, w\in W \text{ such that } B \text{ crosses the tree-walls in }  T_{w,s}  \right\rangle \]
is a dense subgroup of $ U_{c_0}|_B$.
\end{lem}
\begin{proof}
Note that the subgroups $ \prod_{T_{w,s}} G_0^s$ fix the chamber $c_0$ and stabilize the set $B$, so we may consider these to be subgroups of $U_{c_0}|_B$.

We first prove the result for finite $K$ by induction on the size of $K$.
If $|K|=1$ then $K=\{\id\}$ and $B=\{c_0\}$,
hence $U_{c_0}|_B$ is trivial.

Assume now by induction hypothesis that the result holds for every $K$ of size $n\leq N$.
Let $K \subset W$ be a connected subset of $N+1$ elements containing the identity.
Let $w\in K$ be such that $K \setminus \{w\}$ is still connected.
(This is always possible, for instance by picking a vertex of valency one in a spanning tree of~$K$.)
Let \[K_1=K\setminus \{w\}\ \text{ and } \ A_1=\{c\in \ch\Delta \mid \delta(c_0, c)\in K_1\}.\]
Let $g\in U_{c_0}|_B$ fixing $A_1$ and let $c\in B\setminus A_1$. Note that $B\setminus A_1$ consists exactly of those chambers at Weyl distance $w$ from $c_0$.

As $K$ is connected, we know that there exists a $w_1\in K_1$ which is $s$-adjacent to $w$ for some $s\in S$.
Hence $c$ is $s$-adjacent to some $c_1\in A_1$.
Note that $g$ stabilizes the $s$-panel $\Pa$ of $c$ and $c_1$.
Let $\T$ be the $s$-tree-wall containing $\Pa$, and let $w_2$ be its distance to $c_0$.

If $\T$ is already crossed in $A_1$ then $A_1$ contains chambers in the same $s$-wing of the tree-wall $\T$ as $c$.
So this wing is stabilized, and as the $s$-panel containing $c$ is also stabilized, we can conclude that $g$ fixes the chamber $c$.

If $\T$ was not crossed by $A_1$ then, as $g$ fixes $c_1$, it acts on the $s$-panel $\Pa$ as an element of $G^s_0$ (see Lemma~\ref{lem:local}).
Repeating this reasoning for each possible chamber $c$ at Weyl distance $w$ from $c_0$, we conclude that that $g$ is contained in $ \prod_{T_{w_2,s}} G_0^s$ considered as subgroup of $U_{c_0}|_B$.
By the induction hypothesis any element $g_1 \in U_{c_0}|A_1$ is in the desired conditions, and every element $h\in U_{c_0}|_B$ can be written as a product $g\circ g_1$ with $g$ and $g_1$ as before.
Hence we conclude that the statement of the lemma holds for $K$, and hence for every finite $K$.

If $K$ is infinite then we can approximate $K$ by a sequence $(K_1, K_2, \dots)$ of finite connected subsets of $K$ containing the identity, and such that for every $n$, there is an $N$ such that $K_i$ agrees with $K$ on the ball of radius $n$ around the identity in $W$ for every $i > N$.
Let $(A_1, A_2, \dots)$ be the corresponding sets of chambers in $\Delta$.
The group $U_{c_0}|_B$ is then the inverse limit of the groups $U_{c_0}|_{A_i}$, which are each a quotient of $U_{c_0}|_B$.
As the statement holds for each of these quotients, and the fact that we take the closure, we conclude that the statement holds for every $K$, finite or infinite.
\end{proof}

\begin{defn}\label{localuniversal}
Let $J\subseteq S$ and $\Ra$ be a residue of type $J$ in $\Delta$. We define
\begin{multline*}
        U(\Ra) := \{ g\in \Aut \Ra \mid (h_s|_{\Pa_{s,gc}})\circ g\circ (h_s|_{\Pa_{s,c}})^{-1} \in G^s, \text{ for all } s\in J \\
        \text{ and for all chambers } c\in \ch\Ra \ \}.
\end{multline*}
In particular, $U=U(\Delta)$.
\end{defn}

\begin{lem}\label{equalchamb}
Let $\Ra$ be a $J$-residue containing $c_0$ for some $J\subseteq S$.
With the same notation as above, $U_{c_0}|_{\ch\Ra}$ and  $U_{c_0}(\Ra)$ acting on $\ch\Ra$ are permutationally isomorphic.
\end{lem}

\begin{proof}
We consider both groups as subgroups of the symmetric group acting on $\ch\Ra$.

We observe that $\ch\Ra$ corresponds to the set of chambers at distance $w$ for some $w\in W_J$, the parabolic subgroup of $W$ with set of generators $J\subseteq S$.
As $W_J$ is a connected subset of $W$ (as Coxeter complex), we can apply Lemma~\ref{generators} and obtain that $U_{c_0}(\Delta)|_{\ch\Ra}$  is the closure of its subgroup generated by the groups $ \prod_{T_{w,s}} G_0^s$ (regarded as subgroups of $\Sym{\ch\Ra}$) such that $w \in W_J$ and $s \in J$.

On the other hand, the residue $\Ra$ is, in its own right, a right-angled building of type $(W_J, J)$, to which we may apply the same Lemma~\ref{generators}, yielding that $U_{c_0}(\Ra)$ is the closure of the subgroup generated by the same $\prod_{T_{w,s}} G_0^s$ where $w \in W_J$ and $s \in J$.

We therefore conclude that $U_{c_0}|_{\ch\Ra}$ and $U_{c_0}(\Ra)$ have the same action on $\ch\Ra$.
\end{proof}

\begin{lem}\label{equivalence}
Let $\Ra$ be a $J$-residue of $\Delta$, then the groups $U(\Delta)|_{\ch\Ra}$ and $U(\Ra)$ acting on $\ch\Ra$ are permutationally isomorphic.
\end{lem}
\begin{proof}
It is clear that $U(\Delta)|_{\ch\Ra}\subseteq U(\Ra)$.
Let $g$ be an arbitrary element of $U(\Ra)$ and let $c_0$ a chamber in $\Ra$.
We want to prove that there exists a $g'$ in the stabilizer of $\Ra$ in $U(\Delta)$ with the same action as $g$ on $\ch\Ra$.

As $U(\Delta)$ is chamber-transitive there exists an $h\in U(\Delta)$ such that $h.c_0= g.c_0$.
Note that $h$ necessarily stabilizes the residue $\Ra$.
The automorphism $h^{-1} \circ g$ fixes the chamber $c_0$, hence we may apply Lemma~\ref{equalchamb} and conclude that there exists an element $g_0$ with the same action as $h^{-1} \circ g$ on $\ch\Ra$.
This yields that $g' := h \circ g_0$ has the same action as $g$ on $\ch\Ra$, proving the lemma.
\end{proof}

Another way to state Lemma~\ref{equivalence} is as follows.

\begin{prop}\label{extuni}
Let $J \subseteq S$ and $\Ra$ be a residue of type $J$ in $\Delta$.
Let $g\in U(\Ra)$.
Then $g$ extends to an element $\widetilde{g}\in U$.
\end{prop}

\begin{proof}
This follows directly from Lemma~\ref{equivalence}.
\end{proof}

\subsection{Subgroups of $U$ with support on a single wing}

The wings of tree-walls (see Definition~\ref{wings}) are the analogues of half-trees in trees in the more general setting of right-angled buildings.
Subgroups of the universal group $U$ which have support on a single wing will play a crucial role in what follows.
Concretely, for $s\in S$ and $c\in \ch\Delta$, let
\[ V_s(c) =\{g \in U \mid g.v= v \text{ for all }v \not\in X_s(c)\}. \]
The subgroup $V_s(c)$ fixes the chamber $c$ and stabilizes its $s$-panel.
In \cite[Section 5]{Caprace} similar subgroups are defined in the whole group of type preserving automorphisms of a right-angled building.

The next few results demonstrate the importance of these subgroups of the universal group.
In particular, the following proposition generalizes the well known {\em independence property} for groups acting on trees.
(To be precise, it is a slight variation of the independence property, since an $s$-tree-wall in a tree corresponds to a star around a vertex in the tree.)

\begin{prop}\label{fixpanel}
Let $c\in \ch\Delta$ and $s\in S$. Let $\T$ be the $s$-tree-wall of $c$.
Then
\[ \Fix_U(\T) = \prod_{d\in\Pa_{s, c}} V_s(d). \]
\end{prop}

\begin{proof}
We start by showing that $\prod_{d\in\Pa_{s, c}} V_s(d)$ is a subgroup of the fixator $\Fix_U(\T)$. Let $d\in\Pa_{s, c}$.
Given $x\in \ch{\T}$, we deduce from Lemma~\ref{le:wings} that $V_s(d)$ fixes all chambers of $s$-panel $\Pa_{s,x}$ different from the projection of $d$ to that panel.
Hence $V_s(d)$ fixes $\Pa_{s,x}$. This proves that $V_s(d)$ is contained in $\Fix_U(\T)$.
As the supports of each of the the subgroups $V_s(d)$ are disjoint, we can apply Lemma~\ref{CapraceP} and consider $\prod_{d\in\Pa_{s, c}} V_s(d)$ as a subgroup of the universal group and $\Fix_U(\T)$ in particular.

In order to show the other inclusion, pick an element $g\in \Fix_U(\T)$.
Let $d$ be an arbitrary chamber in the $s$-panel $\Pa_{s,c}$.
Consider the permutation $g_d$ of $\ch\Delta$ defined as
\[g_d \colon  \ch\Delta \to \ch\Delta  \colon  x\mapsto \left\{ \begin{array}{ll} g.x & \text{if } x\in X_s(d),\\ x &\text{otherwise.}\end{array}\right.\]
By the proof of Proposition 5.2 in \cite{Caprace} we know that $g_d$ is a type preserving automorphism of $\Delta$.
Clearly $g_d$ fixes the tree-wall $\T$, therefore it preserves projections to the $s$-panels in the tree-wall.
Hence it also preserves the $s$-wings with respect to the $s$-panel of $c$.
Now we have to show that it is an element of the universal group, i.e., we have to prove that
\[ h_t\circ g_d\circ (h_t|_{\Pa_{t,x}})^{-1}\in G^t, \text{ for all } t\in S \text{ and for all } x\in \ch\Delta. \]

We observe that for any panel $\Pa$ not in the tree-wall $\T$, $\Pa$ is not parallel to $\Pa_{s,c}$. Hence $\proj_{\Pa_{s,c}}(\Pa)$ is a chamber by Lemma~\ref{Lemma2.5}.
Therefore either $\ch{\Pa}\subseteq X_s(d)$ or $\ch{\Pa}\subseteq \ch\Delta\setminus X_s(d)$.

Let $\Pa$ be a $t$-panel for some $t\in S$.
If $\Pa$ is in the tree-wall $\T$ or in one of its wings different from $X_s(d)$, then $\Pa$ is fixed by $g_d$ so the permutation $h_t\circ g_d\circ (h_t|_{\Pa})^{-1}$ is the identity, therefore it is in $G^t$.
If $\Pa\in X_s(d)$ then $g_d.x=g.x$ for all $v\in \ch{\Pa}$, so $g_d$ is in the universal group as $g$ is.

We conclude that $g_d$ is an element of $\Fix_U(\Ra)$ and by construction also of $V_s(d)$.
Moreover, the tuple $(g_d)_{d \in \Pa_{s,c}}$, which is an element of $\prod_{d \in \Pa_{s,c}} V_s(d)$, coincides with $g$.
Therefore $g\in \prod_{d \in \Pa_{s,c}} V_s(d)$.
\end{proof}

\begin{lem}\label{Lemma7.1}
Let $s\in S$ and $c_1, c_2$ be two $s$-adjacent chambers in an $s$-panel $\Pa$.
Let $g\in U$, and let $s_1\cdots s_n$ be a reduced representation of $\delta(c_2, g.c_1)$.
Assume that
\begin{enumerate}[\rm (1)]
	\item
	there exists an $i \in \{1, \dots, n\}$ such that $|s_is|=\infty$, and
	\item
	$\proj_{\Pa}(g.c_1)=c_2$ and $\proj_{\Pa_{s, g.c_1}}(c_2)=g.c_1$.
\end{enumerate}
Then for each $h\in \prod_{d\in \ch{\Pa}\setminus \{c_1,c_2\}} V_s(d)$, there exists an $x\in U$ such that \[ h=[x,g]=xg\inv{x}\inv{g}. \]
\end{lem}

\begin{proof}
Let $V_0=\prod_{d\in \ch{\Pa}\setminus \{c_1,c_2\}} V_s(d)$.
We know that $V_0$ is subgroup of $U$.
For each $n>0$ let
\[\Pa_n=g^n(\Pa),\ c_1^n=g^n(c_1),\ c_2^n=g^n(c_2)\ \text{ and } \ V_n=g^nV_0g^{-n}.\]
For each $n\geq 0$ the support of $V_n$ is contained in $\bigcup_{d\in\ch{\Pa_n}\setminus \{c_1^n, c_2^n\}} X_s(d)$.
Given $d\in \ch{\Pa_n}\setminus \{c_1^n, c_2^n\}$ and $m>n$ we have
\[ d\in X_s(c_1^m) \text{ and } c_1^m\not\in X_s(d) \]
because $\proj_{\Pa}(g.c_1)=c_2$ and $\proj_{s, g.c_1}(c_2)=g.c1$.
Thus $X_s(d)\subset X_s(c_1^m)$ by \cite[Lemma 3.4]{Caprace}.
Similarly, $X_s(e)\subset X_s(c_2^n)$ for $e \in \ch{\Pa_m}\setminus \{c_1^m, c_2^m\}$.
This implies that the sets
\[ \bigcup_{d\in \ch{\Pa_n}\setminus\{c_1^n, c_2^n\}}X_s(d)\ \ \ \text{ and }\bigcup_{d\in \ch{\Pa_m}\setminus\{c_1^m, c_2^m\}}X_s(d) \] are disjoint.
This means that, for $m>n\geq 0$, $V_m$ and $V_n$ have disjoint support.
Lemma~\ref{CapraceP} now implies that the direct product $V=\prod_{n\geq 0} V_n$ is a subgroup of~$U$.
Moreover $gV_ng^{-1}=V_{n+1}$.

Given $h\in V_0$, let $x_n=g^nhg^{-n}$.
Then the tuple $x=(x_n)_{n\geq 0}$ is an element of the product $V \leq U$ and so is the commutator $[x,g]$.
We observe that the commutator $[x,g]$ fixes $c_1^n$ and $c_2^n$ for all $n\geq 0$.

Furthermore, denoting by $y_n$ the $n$-th component of an element $y\in V$ according to the decomposition  $V=\prod_{n\geq 0} V_n$ we have \mbox{$[x,g]_n=x_n(g\inv{x}\inv{g})_n$} for all $n\geq 0$.
Hence $[x,g]_0=h$ and $[x,g]_n=x_ng\inv{x_{n-1}}\inv{g}=x_n\inv{x_n}=1$. Therefore $[x,g]=h$, which proves the lemma.
\end{proof}

\subsection{Normal subgroups of the universal group}

We are now in a position to study the normal subgroups of the universal group.
For this we need the notion of an $s$-hyperbolic element.

\begin{defn}\label{action:treewalltree}
For every $s\in S$, any automorphism of the right-angled building $\Delta$ also induces an automorphism of the $s$-tree-wall tree $\Gamma_s$ (cf.\@ Definition~\ref{deftreewall}).
As the universal group $U$ acts chamber-transitively on $\Delta$, it has an edge-transitive type-preserving action on this tree $\Gamma_s$.

Type preserving automorphisms of trees can act in two different ways.
An action is either \emph{elliptic}, meaning the automorphism fixes some vertex of the tree, or it is \emph{hyperbolic}, when the automorphism does not fix any vertex;
in this case, the vertices with minimal displacement form an embedded bi-infinite line graph $A(g)$ in $\Gamma_s$ translated by $g$, called {\em the axis} of $g$.
We say that an element $g\in U$ is \emph{$s$-hyperbolic} if it induces a hyperbolic action on the $s$-tree-wall tree.
\end{defn}

\begin{lem}\label{faithful}
If the right-angled building $\Delta$ is irreducible, then its universal group $U$ acts faithfully on the $s$-tree-wall tree $\Gamma_s$, for all $s\in S$.
\end{lem}
\begin{proof}
Assume by way of contradiction that some non-trivial group element $g \in U$ acts trivially on $\Gamma_s$ for some $s \in S$.
This implies that $g$ stabilizes every $s$-tree-wall and every residue of type $S\setminus \{s\}$ in the building $\Delta$.
The residues of types $s^\perp$, which are the non-trivial intersections of these two, are hence also stabilized.
As $g$ is non-trivial, there exist distinct chambers $c_1$ and $c_2$ of $\Delta$ such that $g.c_1=c_2$.
These chambers have to be contained in a common residue $\Ra$ of type $s^\perp$.

Let $s_1\dotsm s_n$ be a reduced word representing the Weyl distance between $c_1$ and $c_2$.
For each $i \in \{1, \ldots, n\}$, let $\gamma_i$ be a shortest path in the Coxeter diagram $\Sigma$ between $s_i$ and an element $t\in S$ such that $|st|=\infty$. Both the elements $t$ and the paths exist because we are assuming $\Sigma$ to be connected.

Let $\gamma_j$ be such a path of minimal length between $s_i$ and $t$ (minimized over all possible $i \in \{1, \dots, n\}$ and elements $t$). Denote $\gamma_j =( r_1, \dots, r_k)$, with $r_1=s_j$ and $r_k=t$ with $|ts|=\infty$.
We observe that $|r_ir_{i+1}|=\infty$ for all $i$ by definition of a path in $\Sigma$.
Moreover, $r_i \in s^\perp\setminus\{s_1, \dots, s_n\}$ for all $i\in \{2,\dots k-1\}$ since $\gamma_j$ was chosen as a shortest path of minimal length.
Therefore
\[w_1= r_{k-1}\dots r_2s_1\dotsm s_nr_{2}\dots r_{k-1} \text{ and } w_2= tr_{k-1}\dots r_2s_1\dotsm s_nr_{2}\dots r_{k-1}t\]
are reduced words in $M_S$.

Pick a chamber $d_1$ at Weyl distance $r_2\dots r_{k-1}$ from $c_1$ and let $d_2=g. d_1$.
Observe that the word $w_1$ represents the Weyl distance from $d_1$ to $d_2$.
Hence $d_1$ and $d_2$ also in $\Ra$ and therefore in the same $s$-tree-wall $\T$.

Let $\Pa_1$ and $\Pa_2$ be the $t$-residues of $d_1$ and $d_2$, respectively.
These panels are not contained in $\Ra$ as $|ts|=\infty$.
So, if $d_1'\in \ch{\Pa_1}\setminus \{d_1\}$, then the $s$-tree-wall $\T_1$ of $d_1'$ is distinct from $\T$.
Let $d_2'=g.d_1'\in \ch{\Pa_2}\setminus\{d_2\}$.
Analogously, the $s$-tree-wall $\T_2$ of $d_2'$ does not coincide with $\T$.
Moreover $w_2$ is a reduced representation of the Weyl distance between $d_1'$ and $d_2'$.
This implies that $\T_1$ and $\T_2$ are distinct.
As $g$ maps $d_1'$ to $d_2'$, it maps $\T_1$ to $\T_2$.
We have hence arrived at a contradiction as $g$ should have stabilized $\T_1$, from which we conclude that the action is faithful.
\end{proof}

\begin{lem}\label{hip}
Assume that $\Delta$ is thick. Then every non-trivial normal subgroup of $U$ contains an $s$-hyperbolic element, for every $s\in S$.
\end{lem}
\begin{proof}
Let $s\in S$ and $N$ be a non-trivial normal subgroup of $U$.
Assume that $N$ does not contain any $s$-hyperbolic element.
By~\cite[Proposition 3.4]{Tits} we know that either $N$ fixes some vertex or it fixes an end of the tree $\Gamma_s$.

First assume that $N$ fixes some vertex.
This implies, since $N$ is normal in the type-preserving edge-transitive group $U$, that $N$ fixes all vertices of the same type of this vertex.
Therefore it fixes every vertex of the tree $\Gamma_s$ and hence, by Lemma~\ref{faithful}, it contradicts the non-triviality of $N$.

Next assume that $N$ fixes an end of the tree.
Again using normality and edge-transitivity (and the fact that the tree  $\Gamma_s$ contains vertices of valency at least three), we obtain that $N$ fixes at least three ends of $\Gamma_s$.
Since three ends of a tree determine a unique vertex we get that $N$ has a global fixed point, a possibility already handled in the previous paragraph.

We conclude that for every $s\in S$, the normal subgroup $N$ contains an $s$-hyperbolic element.
\end{proof}

\begin{prop}\label{normfix}
Assume that $\Delta$ is thick and of irreducible type.
Let $s\in S$ and $\T$ be an $s$-tree-wall.
Then any non-trivial normal subgroup $N$ of $U$ contains $\Fix_U(\T)$.
\end{prop}
\begin{proof}
Let $g \in N$ be an $s$-hyperbolic element, which exists by Lemma~\ref{hip}.
Let $\T$ be an $s$-tree-wall for which the corresponding vertex of $\Gamma_s$ lies on the axis $A(g)$.
Let $\Pa$ be an $s$-panel of $\T$, $\Pa_1=g.\Pa$ and $\T_1=g.\T$ (so $\Pa_1\in\T_1$).
We note that $\T$ and $\T_1$ are distinct since $g$ is $s$-hyperbolic.

Let $\{c\}=\proj_{\Pa}(\Pa_1)$ and $\{c_1\}=\proj_{\Pa_1}(\Pa)$.
We claim that $g.c\neq c_1$.
Assume by way of contradiction that $g.c=c_1$.
Then the residue $\Ra=\Ra_{S\setminus\{s\}, c}$ is mapped to $\Ra_1=\Ra_{S\setminus\{s\}, c_1}$, both corresponding to vertices of $\Gamma_s$.
Both of these residues belong to the wings $X_s(c)$ and $X_s(c_1)$ implying that $\dist{\Ra}{\Ra_1}$ (as vertices of $\Gamma_s$) is strictly smaller than $\dist{\T}{\T_1}$.
This a contradiction to the fact that $\T\in A(g)$. Hence $g.c\neq c_1$.

Let $c'=\inv{g} . c_1$.
Applying Lemma~\ref{Lemma7.1} to $c$ and $c'$ yields that
\[ \prod_{d\in\ch{\Pa}\setminus\{c, c'\}} V_s(d) \subseteq N . \]
As $\Pa$ is thick, we can pick a chamber $c_2$ in $\Pa\setminus \{c, c'\}$.
As $G^s$ is transitive, Lemma~\ref{lem:local} implies that there exists a $g_1\in U$ such that $g_1.c_2 = c$.
We observe that $\inv{g_1}\circ V_s(c_2)\circ g_1 = V_s(g_1.c_2) = V_s(c)$, hence $V_s(c) \subseteq N$. Analogously one proves that $V_s(c') \subseteq N$.
We therefore obtain that
\[  \prod_{d\in\ch{\Pa}} V_s(d) \subseteq N\]
and, by Proposition~\ref{fixpanel}, that  $\Fix_U(\T)\subseteq N$.

Since $U$ is chamber-transitive, there exists, for each tree-wall $\T'$, an element $h\in U$ mapping $\T$ to $\T'$.
Hence $\inv h \circ \Fix_U(\T) \circ h = \Fix_U(\T') \subseteq N$.
\end{proof}

\subsection{Simplicity of the universal group}

In this section we prove the simplicity of the universal group provided that the following two conditions are satisfied.
\begin{itemize}
	\item[(IR)]
	The right-angled building is thick, irreducible, and has rank $\geq 2$.
	\item[(ST)]
	For each $s \in S$, the group $G^s$ is transitive and generated by its point stabilizers.
\end{itemize}

\begin{rem}
We observe that if we make the stronger assumption that the local action of the universal group is given by $2$-transitive groups, i.e.\@ if each group $G^s\leq\Sym{Y_s}$ is assumed to be $2$-transitive, then we can use similar arguments to the ones in \cite[Proposition 6.1]{Caprace} to show that the action of~$U$ on $\ch\Delta$ is strongly transitive and to prove simplicity in that manner.
\end{rem}

The second condition (ST) is necessary, as is clear from the following proposition.

\begin{prop}\label{edgestab}
Let $U^+$ be the subgroup of $U$ generated by chamber stabilizers.
Then $U=U^+$ if and only if for every $s\in S$ the group $G^s$ is transitive and generated by point stabilizers.
\end{prop}
\begin{proof}
Notice that, in the definition of $U$, we already assume the groups $G^s$ to be transitive.
Assume first that $U=U^+$.
Fix an element $s\in S$; we will show that $G^s$ is generated by its point stabilizers.
Let $(G^s)^+$ denote the subgroup of~$G^s$ generated by the point stabilizers $(G^s_\alpha)_{\alpha\in Y_s}$.

For any $s$-panel $\Pa$ and any $g,h \in U$, we define
\begin{equation}\label{eq:sigma}
    \sigma(g, \Pa) := h_s|_{g.\Pa}\circ g \circ (h_s|_{\Pa})^{-1} ,
\end{equation}
which is an element of $G^s$ by the very definition of the universal group $U$.
Observe that
\begin{equation}\label{eq:sigma-comp}
    \sigma(gh, \Pa) = \sigma(g, h.\Pa) \circ \sigma(h, \Pa)
\end{equation}
for all $g,h \in U$.

We first claim that if $g\in U_c$ for some chamber $c\in\ch\Delta$, then
\begin{equation}\label{eq:sigma-claim}
    \sigma(g, \Pa) \in (G^s)^+
\end{equation}
for all $s$-panels $\Pa$ of $\Delta$.
We will prove this claim by induction on the distance $\dist c \Pa := \min \{\dist c d \mid d\in \ch\Pa\}$.
If $\dist c \Pa =0$, then $g$ stabilizes $\Pa$ since $g$ fixes $c$, hence $\sigma(g, \Pa)\in G^s_{h_s(c)}\leq (G^s)^+$.

Assume now that the claim is true whenever $\dist c \Pa \leq n$ and
let $\Pa$ be a panel at distance $n+1$ from $c$.
Then there is an $s$-panel $\Pa_1$ at distance $n$ from $c$ such that there are chambers $d\in \ch\Pa$ and $d_1\in \ch{\Pa_1}$
that are $t$-adjacent for some $t \neq s$.
Then $h_s(d)=h_s(d_1)$; denote this value by $\alpha \in Y_s$.
Similarly, $g.d$ and $g.d_1$ are $t$-adjacent, hence $h_s(g.d) = h_s(g.d_1)$.
By~\eqref{eq:sigma}, this implies $\sigma(g, \Pa_1)(\alpha)=\sigma(g, \Pa)(\alpha)$.
We conclude that $\sigma(g, \Pa)=\sigma(g, \Pa_1)g_\alpha$ for some $g_\alpha\in G^s_\alpha$.
By our induction hypothesis, $\sigma(g, \Pa_1)\in (G^s)^+$, and therefore $\sigma(g, \Pa)\in (G^s)^+$, which proves the claim~\eqref{eq:sigma-claim}.

Now let $g\in G^s$ be arbitrary.
Choose an arbitrary $s$-panel $\Pa$.
By Lemma~\ref{lem:local}, the action of $U|_{\{\Pa\}}$ on $\Pa$ is permutationally isomorphic to $G^s$,
so in particular, we can find an element $u \in U|_{\{\Pa\}}$ such that its local action on $\Pa$ is given by~$g$,
i.e., $g = \sigma(u, \Pa)$.
Since $U = U^+$, we can write $u$ as a product $u_{c_1}\dotsm u_{c_n}$, where each $u_{c_i}$ fixes a chamber $c_i$.
Then $g =\sigma(u, \Pa) = \sigma(u_{c_1}\dotsm u_{c_n}, \Pa)$.
It now follows from~\eqref{eq:sigma-comp} and~\eqref{eq:sigma-claim} that $g \in (G^s)^+$,
so $G^s$ is indeed generated by its point stabilizers.

\smallskip

Conversely, assume that for each $s\in S$, the group $G^s$ is generated by point stabilizers.
Let $s\in S$ and $\Pa$ be an $s$-panel.
Then, as $U|_{\{\Pa\}}$ is permutationally isomorphic to $G^s$, if $g\in  U|_{\{\Pa\}}$ for some $s$\dash panel $\Pa$,
then $h_s\circ g \circ (h_s)^{-1}|_\Pa =g_{\alpha_1}\cdots g_{\alpha_n}$, where $g_{\alpha_i}$ fixes the color $\alpha_i$.
Lift the elements $g_{\alpha_1},\dots,g_{\alpha_{n-1}}$ to elements $g_{c_1},\dots,g_{c_{n-1}} \in U$ arbitrarily,
and then define $g_{c_n}\in U$ such that $g=g_{c_1}\cdots g_{c_n}$.
Then each $g_{c_i}$ fixes the corresponding chamber $c_i \in \Pa$ of color $\alpha_i$, and hence $g\in U^+$.
As $G^s$ is transitive, $U^+$ contains elements of~$U$ mapping one chamber to any other $s$-adjacent chamber.
Since this is true for any $s\in S$, the subgroup $U^+$ is chamber-transitive.
Since in addition $U^+$ contains the full chamber stabilizers in $U$,
we conclude that $U^+$ is indeed all of $U$.
\end{proof}

\begin{prop}\label{nontrivialstab}
Let $\Delta$ be a right-angled building satisfying (IR) and (ST).
Let $N$ be a non-trivial normal subgroup of $U$.
Then for each panel $\Pa$ of $\Delta$, we have $N\cap \Stab_U(\Pa)\neq 1$.
\end{prop}
\begin{proof}
Let $s\in S$ and $\Pa$ be an $s$-panel of $\Delta$.
Let $g\in G_0^s$ be a non-trivial element.
Then $g$ induces a non-trivial permutation of $\ch\Pa$ fixing some $c_0\in \ch\Pa$.
By Proposition~\ref{prop:nicegenerators}, we can find a corresponding element $\widetilde{g}\in U$ stabilizing $\Pa$, acting locally as $g$ on $\Pa$, and fixing the chambers in the $s$-wing $X_s(c_0)$.

Let $t \in S$ be such that $|ts|=\infty$ (such an element always exists because $\Delta$ is irreducible).
Let $\T$ be the $t$-tree-wall of $c_0$. Then $\ch\T \subseteq X_s(c_0)$.
Therefore $\widetilde{g}$ fixes $\T$, i.e., $\widetilde{g}\in \Fix_U(\T)$. Hence $\widetilde{g}\in N$ by Proposition~\ref{normfix}.
We conclude that $N\cap \Stab_U(\Pa)$ is non-trivial.
\end{proof}

We are now ready to prove the main result.

\begin{thm}\label{simplicity}
Let $\Delta$ be a thick right-angled building of  irreducible type $(W,S)$ with prescribed thickness $(q_s)_{s\in S}$ and rank at least 2.
For each $s\in S$, let $h_s \colon  \ch\Delta \to Y_s$ be a legal $s$-coloring and $G^s\leq \Sym{Y_s}$ be a transitive group generated by point stabilizers.

Then the universal group $U$ of $\Delta$ with respect to the groups $(G^s)_{s\in S}$ is simple.
\end{thm}
\begin{proof}
We first observe that by Proposition~\ref{edgestab}, $U$ coincides with $U^+$.
We will prove the simplicity by induction on the rank of $\Delta$.

If $\Delta$ has rank 2 then the simplicity of $U^+=U$ follows from \cite[Theorem 4.5]{Tits} since by definition $U$ is the universal group of a biregular tree and it has Tits independence property.

Now assume that the rank of $\Delta$ is at least three, and that we have proven simplicity for lower rank.
Let $N$ be a non-trivial normal subgroup of the universal group $U$.
It suffices to show that $N$ contains the chamber stabilizer $U_c$ of $c$ in $U$, for each chamber $c \in \ch\Delta$.
 Let $c\in \ch\Delta$.
 We will show that the stabilizers $N_c$ and $U_c$ coincide.

Pick a generator $s \in S$ such that $S \setminus \{s\}$ is irreducible.
(Note that this is always possible, by picking $s$ to be a leaf of a spanning tree of the unlabeled Coxeter diagram.)
Let $\Ra$ be the $S \setminus \{s\}$-residue of $\Delta$ containing $c$.
This residue is a right-angled building on its own of irreducible type and of rank at least two.

\begin{claim}\label{lem:fix}
$N$ contains the fixator $\Fix_U(\Ra)$ of $\Ra$ in $U$.
\end{claim}
Let $r$ be an element of $S$ not commuting with $s$ (which is always possible by the irreducibility of the Coxeter system).
The set $\{ r \} \cup r^\perp$ is a subset of $S \setminus \{s\}$ and therefore the residue $\Ra$ contains a residue of type $\{ r \} \cup r^\perp$, which forms an $r$-tree-wall $\T$.
The normal subgroup $N$ contains the fixator $\Fix_U(\T)$ by Proposition~\ref{normfix}, hence it also contains its subgroup $\Fix_U(\Ra)$.

\begin{claim}\label{lem:stab}
The stabilizer $\Stab_N(\Ra)$ maps surjectively onto $U(\Ra)$.
\end{claim}
We first observe that  the image of $N\cap \Stab_U(\Ra)$ in $U(\Ra)$ ($\cong \Stab_U(\Ra)/\Fix_U(\Ra)$) is non-trivial since $\Stab_U(\Pa)\subseteq \Stab_U(\Ra)$ for any panel $\Pa$ in $\Ra$ and $N\cap \Stab_U(\Pa)\neq 1$ by Proposition~\ref{nontrivialstab}.

By induction on the rank we know that that $U(\Ra)$ is simple.
By Proposition~\ref{extuni}, the natural homomorphism from $\Stab_U(\Ra)$ to $U(\Ra)$ is surjective;
it follows that its restriction to $N\cap \Stab_U(\Ra)$ is still surjective.
We then conclude that $\Stab_N(\Ra)$ maps surjectively onto $U(\Ra)$.

The chamber stabilizers $N_c$ and $U_c$ also stabilize the residue $\Ra$, hence we may consider their image in the universal group $U(\Ra)$.
By Claim~\ref{lem:stab}, the images of $N_c$ and $U_c$ in $U(\Ra)$ are both equal to the entire group $U(\Ra)$.
The kernels of the maps from $N_c$ and $U_c$ to $U(\Ra)$ are the fixators $\Fix_N(\Ra)$ and $\Fix_U(\Ra)$, respectively,
which also coincide by Claim~\ref{lem:fix}. We conclude that $U_c$ and $N_c$ are equal for all $c$,
hence $N$ contains all chamber stabilizers $U_c$ and thus coincides with $U$, proving the simplicity.
\end{proof}

\begin{rem}
    As the referee pointed out, it makes sense to put this simplicity result in perspective with the theory
    of groups acting on CAT(0) cube complexes.
    There are indeed various ways to construct geometric realizations of non-$2$-spherical buildings that are CAT(0) cube complexes.
    One particularly interesting approach in the case of right-angled buildings is to view the building as a ``space with walls''
    as in \cite{Bogdan}.
    In this setup, 
    natural candidates for such walls are provided by the partitions of the chamber set into a wing and its complement.

    In \cite[Proposition A.2]{Laz}, it is shown that the group $G^+$ generated by the fixators of half-spaces is a simple group.
    Notice, however, that this result does {\em not} imply our simplicity result, precisely because showing that $G^+ = G$
    would require a similar amount of work (and probably a similar strategy) as in our proof of Theorem \ref{simplicity}.
    (We thank Pierre-Emmanuel Caprace for his help in formulating this remark.)
\end{rem}

\section{Maximal compact open subgroups of~$U$}\label{se:cpt-open}

The universal group of a locally finite semi-regular right-angled building is a locally compact group.
Crucial in understanding such a group is to understand its compact subgroups.
This is the aim of this section.

Let $(W,S)$ be a right-angled Coxeter system with Coxeter diagram $\Sigma$.
Let $(q_s)_{s\in S}$ be a set of natural numbers with $q_s\geq 3$ and let $Y_s = \{1, \dots, q_s\}$.
Consider  the locally finite thick semi-regular right-angled building $\Delta$ of type  $(W,S)$ with prescribed thickness  $(q_s)_{s\in S}$.
Fix a base chamber $c_0\in \ch\Delta$.
For each $s \in S$, let $h_s \colon \ch\Delta \to Y_s$ be a directed legal $s$-coloring with respect to $c_0$ and let $G^s \leq \Sym{Y_s}$ be a transitive permutation group.
Consider the universal group $U$ of $\Delta$ with respect to the groups $(G^s)_{s\in S}$.

\begin{lem}\label{lem:spherstab}
Let $\Ra$ be a spherical residue of $\Delta$. Then the setwise stabilizer $U_\Ra$ of $\Ra$ in $U$ is compact.
\end{lem}
\begin{proof}
By \cite[Lemma~1]{Woess}, each chamber stabilizer is compact.
Since $\Delta$ is locally finite, each spherical residue of $\Delta$ is finite.
The result follows.
\end{proof}

\begin{prop}
The maximal compact open subgroups of $U$ are exactly the stabilizers of maximal spherical residues of $\Delta$.
\end{prop}
\begin{proof}
Let $H$ be a compact open subgroup of $U$.
By Lemma~\ref{lem:spherstab} it is sufficient to show that $H$ is contained in the stabilizer of a spherical residue of $\Delta$.

Since $H$ is compact, the orbits of $H$ on $\ch\Delta$ are finite.
In particular, $H$~acts (type-preservingly) on any $s$-tree-wall tree with finite orbits; therefore, it fixes a vertex in each $s$-tree-wall tree.
Such a vertex corresponds to a residue $\Ra$ of $\Delta$, which is, in its own right, a right-angled building.
Therefore $H \subseteq \Stab_U(\Ra)$.

As the orbits of action of $H$ on $\ch\Ra$ are still finite, we can repeat the above procedure and obtain a residue $\Ra'$ of smaller rank than $\Ra$
such that $H\subseteq \Stab_U(\Ra')$.
We can continue this procedure until there are no non-trivial tree-wall trees left, which happens exactly when the right-angled building is spherical.

We conclude that $H$ is indeed contained in the stabilizer of a spherical residue.
\end{proof}

The chamber stabilizers $U_{c_0}$ (as well as the maximal compact subgroups of which they are finite index subgroups) are totally disconnected compact groups, and are therefore profinite (see \cite[Corollary 1.2.4]{Wilson}), i.e., they are a projective limit of finite groups.
We will make this inverse limit of finite groups explicit for the subgroups $U_{c_0}$ and we will describe the structure of these finite groups.

\subsection{A description of $U_{c_0}$ as iterated semidirect products}\label{inverse.limit}

Let $U_{c_0}|_{\B{c_0}n}$ be the induced action of $U_{c_0}$ on the $n$-ball $\B{c_0}n$, that is,
$U_{c_0}|_{\B{c_0}n} = U_{c_0}/ \Fix_{U_{c_0}}(\B{c_0}n)$.
(Recall that $\B{c_0}n$ is the set of chambers of $\Delta$ at gallery distance smaller than or equal to $n$, as in Definition~\ref{galdist}.)

The restriction of the action of $U_{c_0}|_{\B{c_0}{n+1}}$ to $\B{c_0}n$ maps $U_{c_0}|_{\B{c_0}{n+1}}$ to $U_{c_0}|_{\B{c_0}n}$ and this restriction is onto.
Therefore
\begin{equation}\label{eq:Uc0}
    U_{c_0} = \varprojlim_n U_{c_0}|_{\B{c_0}n} = \varprojlim_n U_{c_0}|_{\Sp{c_0}n},
\end{equation}
where the last equality follows because there is a unique path of each type between any two chambers of~$\Delta$
and we are considering type-preserving automorphisms only.

We will make the extension of the action of $U_{c_0}|_{\B{c_0}n}$ to $\B{c_0}{n+1}$ explicit following the strategy described in \cite[Section 9]{Gibbins}.
\begin{defn}\label{defn:Ai}
\begin{enumerate}
    \item
        For any $w \in W$, let
        \begin{equation}\label{eq:Lw}
            L(w)=\{s\in S \mid l(ws)<l(w)\},
        \end{equation}
        that is, the set of generators which added to a reduced representation for~$w$ do not form a new reduced word.
        Let $W(n)$ denote the set of elements $w\in W$ of length $n$. Define
        \begin{align*}
            W_1(n) &= \{ w\in W(n) \mid |L(w)| = 1 \} , \\
            W_2(n) &= \{ w\in W(n) \mid |L(w)| \geq 2 \} .
        \end{align*}
        Observe that if $w\in W_1(n)$, then the last letter of $w$ is independent of the choice of the reduced representation for $w$.
        We will write this last letter as $r_w$; then $L(w) = \{ r_w \}$.
    \item
        Let $n \geq 1$ and let $A(n)=\Sp{c_0}{n}$ be the $n$-sphere around~$c_0$.
        We will create a partition of $A(n)$ by defining, for each $i \in \{ 1,2 \}$,
        \begin{equation}\label{eq:A1A2}
            A_i(n)=\{ c \in A(n) \mid \delta(c_0,c) \in W_i(n) \} .
        \end{equation}
        \begin{figure}[ht!]
            \captionsetup[subfigure]{justification=centering}
            \begin{subfigure}[b]{0.38\textwidth}
            \begin{tikzpicture}[scale =0.8, line cap=round,line join=round,>=triangle 45,x=1.0cm,y=1.0cm]
            \clip(3.1,2.4) rectangle (11,9.5);
            \draw [shift={(7.196048227047175,-8.159740084379619)}] plot[domain=0.90299902219246:2.132950643098445,variable=\t]({1.0*12.987535105233386*cos(\t r)+-0.0*12.987535105233386*sin(\t r)},{0.0*12.987535105233386*cos(\t r)+1.0*12.987535105233386*sin(\t r)});
            \draw [shift={(7.196048227047175,-8.159740084379619)}] plot[domain=0.8405516372429191:2.200323608157614,variable=\t]({1.0*14.924888309352937*cos(\t r)+-0.0*14.924888309352937*sin(\t r)},{0.0*14.924888309352937*cos(\t r)+1.0*14.924888309352937*sin(\t r)});
            \draw [line width=1.2pt](6.794841094694121,4.821596555066206)-- (6.460693486180265,6.747021618007139);
            \draw [line width=1.2pt](6.794841094694121,4.821596555066206)-- (7.939261914182559,6.746631857240333);
            \draw (6.583558053137442,4.472276888761895) node[anchor=north west] {$\vdots$};
            \draw (8.03581479852291,6.7983491334894675) node[anchor=north west] {$c_3$};
            \draw (6.571250792583328,6.847578175705924) node[anchor=north west] {$c_2$};
            \draw (5.414368300496599,6.7983491334894675) node[anchor=north west] {$c_1$};
            \draw (6.75,8.11522601277968) node[anchor=north west] {$t$};
            \draw (5.65,8.08) node[anchor=north west] {$t$};
            \draw (8.946552079527356,7.9429243650220815) node[anchor=north west] {$t$};
            \draw (5.6,5.862997331376793) node[anchor=north west] {$s$};
            \draw (6.2,6.022991718580276) node[anchor=north west] {$s$};
            \draw (6.9,5.95) node[anchor=north west] {$s$};
            \draw (6.878932306436181,2.9) node[anchor=north west] {$c_0$};
            \draw [shift={(7.196048227047175,-8.159740084379619)}] plot[domain=0.752893986406544:2.2097710988591883,variable=\t]({1.0*16.732763310877026*cos(\t r)+-0.0*16.732763310877026*sin(\t r)},{0.0*16.732763310877026*cos(\t r)+1.0*16.732763310877026*sin(\t r)});
            \draw [line width=1.2pt](5.279356698174269,6.641563393010703)-- (6.794841094694121,4.821596555066206);
            \draw [line width=1.2pt](3.4250688710404686,8.142562903677115)-- (5.279356698174269,6.641563393010703);
            \draw [line width=1.2pt](4.345950880774987,8.328507643459885)-- (5.279356698174269,6.641563393010703);
            \draw [line width=1.2pt](5.887125218266153,8.521749314799218)-- (6.460693486180265,6.747021618007139);
            \draw [line width=1.2pt](6.845058071938161,8.569341597794358)-- (6.460693486180265,6.747021618007139);
            \draw [line width=1.2pt](8.312548470057324,8.53573221277899)-- (7.939261914182559,6.746631857240333);
            \draw [line width=1.2pt](9.391698372060409,8.42834263962847)-- (7.939261914182559,6.746631857240333);
            \draw (7.7,8) node[anchor=north west] {$t$};
            \draw (4.7,8) node[anchor=north west] {$t$};
            \draw (3.5,7.8) node[anchor=north west] {$t$};
            \draw (9.783445797207118,7.056801605125864) node[anchor=north west] {$n$};
            \draw (9.598836888895406,5.173790740346401) node[anchor=north west] {$n-1$};
            \draw (9.746524015544775,9) node[anchor=north west] {$n+1$};
            \begin{scriptsize}
            \draw [fill=black] (7.196048227047175,-8.159740084379619) circle (3.5pt);
            \draw [fill=black] (5.279356698174269,6.641563393010703) circle (3.5pt);
            \draw [fill=black] (6.460693486180265,6.747021618007139) circle (3.5pt);
            \draw [fill=black] (7.939261914182559,6.746631857240333) circle (3.5pt);
            \draw [fill=black] (6.794841094694121,4.821596555066206) circle (3.5pt);
            \draw [fill=black] (6.9191121815025936,2.9434084491899646) circle (3.5pt);
            \draw [fill=black] (19.406165723102667,3.281343725747452) circle (3.5pt);
            \draw [fill=black] (-0.4659084580477993,2.1531227399789077) circle (3.5pt);
            \draw [fill=black] (3.4250688710404686,8.142562903677115) circle (3.5pt);
            \draw [fill=black] (4.345950880774987,8.328507643459885) circle (3.5pt);
            \draw [fill=black] (5.887125218266153,8.521749314799218) circle (3.5pt);
            \draw [fill=black] (6.845058071938161,8.569341597794358) circle (3.5pt);
            \draw [fill=black] (8.312548470057324,8.53573221277899) circle (3.5pt);
            \draw [fill=black] (9.391698372060409,8.42834263962847) circle (3.5pt);
            \end{scriptsize}
            \end{tikzpicture}
            \subcaption{$c_i \in A_1(n)$: for all $t \neq s$, $l(wt) > l(w)$.}
            \end{subfigure}
            \hspace{2.6cm}
            \begin{subfigure}[b]{0.4\textwidth}
            \begin{tikzpicture}[scale=0.83, line cap=round,line join=round,>=triangle 45,x=1.0cm,y=1.0cm]
            \clip(3.7,1.3) rectangle (10.2,8);
            \draw [shift={(7.5,-7.66)}] plot[domain=0.8281072191624526:2.264578884900933,variable=\t]({1.0*11.592980634849695*cos(\t r)+-0.0*11.592980634849695*sin(\t r)},{0.0*11.592980634849695*cos(\t r)+1.0*11.592980634849695*sin(\t r)});
            \draw [shift={(7.5,-7.66)}] plot[domain=0.8158107657899643:2.23901068745554,variable=\t]({1.0*14.417503251256784*cos(\t r)+-0.0*14.417503251256784*sin(\t r)},{0.0*14.417503251256784*cos(\t r)+1.0*14.417503251256784*sin(\t r)});
            \draw [line width=1.2pt](4.877468301282304,6.516978785665895)-- (4.13142200772546,3.432785146660134);
            \draw [line width=1.2pt](4.877468301282304,6.516978785665895)-- (5.810790384722948,3.8092532832637183);
            \draw [line width=1.2pt](5.810790384722948,3.8092532832637183)-- (6.622348908479946,6.730765391790445);
            \draw [line width=1.2pt](5.810790384722948,3.8092532832637183)-- (8.406499532206302,6.7289769823330285);
            \draw [line width=1.2pt](6.622348908479946,6.730765391790445)-- (7.510742996295576,3.9329756571826984);
            \draw [line width=1.2pt](8.406499532206302,6.7289769823330285)-- (8.811721489227729,3.858532316866512);
            \draw [line width=1.2pt](6.484130852902304,3.5) node[anchor=north west] {$\vdots$ };
            \draw (4.7,7.1) node[anchor=north west] {$c_1$};
            \draw[color=black] (6.76802401325785,7.026868238155123) node {$c_2$};
            \draw[color=black] (8.558734717038993,7.026868238155123) node {$c_3$};
            \draw[color=black] (4.3,5.1) node {$t$};
            \draw[color=black] (5.064665051124568,5.2361575343739775) node {$s$};
            \draw[color=black] (6.05,5.5) node {$s$};
            \draw[color=black] (9.6,6.9) node {$n$};
            \draw[color=black] (9.7,4.1) node {$n-1$};
            \draw[color=black] (7.8,5.75) node {$s$};
            \draw[color=black] (6.8,5.6) node {$t$};
            \draw[color=black] (8.35,5.5) node {$t$};
            \draw[color=black] (7.1,1.5) node {$c_0$};
            \begin{scriptsize}
            \draw [fill=black] (4.877468301282304,6.516978785665895) circle (3.5pt);
            \draw [fill=black] (6.622348908479946,6.730765391790445) circle (3.5pt);
            \draw [fill=black] (8.406499532206302,6.7289769823330285) circle (3.5pt);
            \draw [fill=black] (4.13142200772546,3.432785146660134) circle (3.5pt);
            \draw [fill=black] (5.810790384722948,3.8092532832637183) circle (3.5pt);
            \draw [fill=black] (8.811721489227729,3.858532316866512) circle (3.5pt);
            \draw [fill=black] (7.510742996295576,3.9329756571826984) circle (3.5pt);
            \draw [fill=black] (6.75,1.8) circle (3.5pt);
            \end{scriptsize}
            \end{tikzpicture}
            \subcaption{$c_i \in A_2(n)$: for some $t \neq s$, $l(wt) < l(w)$.}
            \end{subfigure}
            \caption{Partition of $A(n)$.}
        \end{figure}
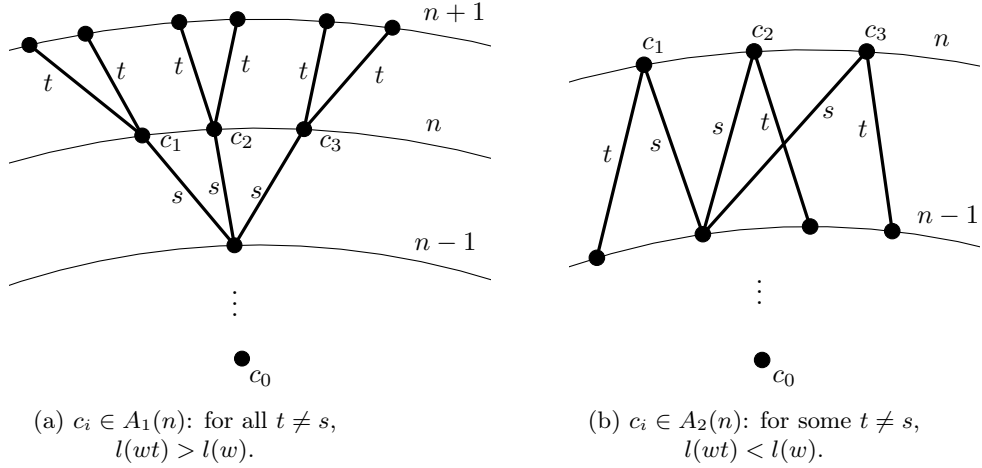
\end{enumerate}
\end{defn}
We observe that if two chambers $c_1$ and $c_2$ in $A(n)$ are $s$-adjacent  for some $s\in S$,
then there is a unique chamber $c$ in their $s$-panel that is in $A(n-1)$.
Thus $\delta(c_0, c_1)=\delta(c_0, c)s=\delta(c_0, c_2)$ and hence $c_1$ and $c_2$ are in the same part $A_i(n)$.
Therefore $A_1(n)$ and $A_2(n)$ are mutually disconnected parts of $A(n)$.

\smallskip

We will extend the action of $U_{c_0}|_{\B{c_0}n}$ to $\B{c_0}{n+1}$ by considering the extension to $A_1(n+1)$ and to $A_2(n+1)$ separately.
Let
\[ C_n = \{ c \in \B{c_0}n \mid c \adj{s} c' \text{ for some } c'\in A_1(n+1) \text{ and some } s\in S\}, \]
and for each $c\in C_n$, let
\[ S_c = \{ s \in S \mid c \adj{s} c' \text{ for some } c'\in A_1(n+1)\} . \]
Now consider the set of pairs
\[ Z_n = \{ (c,s) \mid c\in C_n \text{ and } s\in S_c\}. \]
Equivalently,
\begin{equation}\label{eq:Zn}
    Z_n = \{ (c,s) \in A(n) \times S \mid \delta(c_0,c)s \in W_1(n+1) \} .
\end{equation}
For each element $z=(c,s) \in Z_n$, let $\Pa_z$ be the $s$-panel containing $c$ and let $G_z=G_0^s$, where $G_0^s$ is the stabilizer of the element $1\in Y_s$ in $G^s$.

\begin{lem}\label{le:Uz}
    Let $E := \B{c_0}{n} \cup A_1(n+1)$.
    For each element $z = (c,s) \in Z_n$, there is a subgroup $U_z \leq U_{c_0}|_E$ acting trivially on $E \setminus \Pa_z$ and acting as $G_z$ on~$\Pa_z$.
    In particular, $U_z$ fixes $\B{c_0}{n}$.
\end{lem}
\begin{proof}
    Let $z=(c,s) \in Z_n$, and consider the $s$\dash panel $\Pa = \Pa_z$.
    Notice that $c = \proj_\Pa(c_0)$.
    Since $h_s$ is a directed legal coloring, we know that $h_s(c)=1$ and that, for any $g \in U_{c_0}|_{\B{c_0}n}$,
    also $h_s(g(c))=1$ because $g$ is a type preserving automorphism.

    For each $a \in G_0^s$, we consider the induced automorphism of $\Pa$ given by $a_\Pa = (h_s|_\Pa)^{-1} \circ a \circ h_s|_\Pa \in \Aut\Pa$.
    By Lemma~\ref{prop:nicegenerators} we know that there exists an element $\widetilde{a_{\Pa}} \in U$ acting locally as $a_{\Pa}$ on the panel $\Pa$
    and fixing the $s$-wing $X_s(c)$.
    In particular, $\widetilde{a_{\Pa}}\in U_{c_0}$.

    We claim that $\widetilde{a_{\Pa}}$ fixes each chamber $d \in E \setminus \ch{\Pa}$.
    We start by showing that $\proj_\Pa(d) = c$.
    Indeed, if $\proj_{\Pa}(d)=c_2\in \Sp{c_0}{n+1}$, then by Lemma~\ref{lemmaA} there would exist $c_3\in \Sp{c_0}n$ such that $c_2$ is $t$-adjacent to $c_3$ with $st=ts$,
    contradicting the fact that $c_2\in A_1(n+1)$.
    Therefore $\proj_\Pa(d) = c$.
    It follows that $d \in X_s(c)$, and hence $\widetilde{a_{\Pa}}$ fixes $d$, proving our claim.
    In particular, the restriction of $\widetilde{a_\Pa}$ to $E$ is independent of the choice of the extension $\widetilde{a_\Pa}$;
    denote this element of $U_{c_0}|_E$ by $\widehat{a_\Pa}$.
    Then
    \[ U_z := \{ \widehat{a_\Pa} \mid a \in G_0^s \} \]
    is a subgroup of $U_{c_0}|_E$ isomorphic to $G_0^s$, fixing every chamber outside $\ch{\Pa} \cap A_1(n+1)$.
\end{proof}

\begin{prop}\label{B1}
The group $U_{c_0}|_{\B{c_0}n\cup A_1(n+1)}$ is isomorphic to
\[ (\prod_{z\in Z_n} G_z) \rtimes U_{c_0}|_{\B{c_0}n}, \]
where the action of $U_{c_0}|_{\B{c_0}n}$ on $\prod_{z\in Z_n} G_z$ is given by permuting the entries of the direct product
according to the action of $U_{c_0}|_{\B{c_0}n}$ on $C_n \subseteq \B{c_0}{n}$.
\end{prop}
\begin{proof}
    Let $E = \B{c_0}{n} \cup A_1(n+1)$ as in Lemma~\ref{le:Uz} and
    consider, for each $z \in Z_n$, the subgroup $U_z \leq U_{c_0}|_E$.
    Then the subgroups $U_z$ for different $z \in Z_n$ all have disjoint support,
    and hence form a direct product
    \[ D := \prod_{z \in Z_n} U_z \leq U_{c_0}|_E . \]
    Moreover, $D$ fixes $\B{c_0}n$.

    On the other hand, it is clear that each element $g \in U_{c_0}|_{\B{c_0}n}$ can be uniquely extended to an element $\tilde g \in U_{c_0}|_E$
    in such a way that for each $z \in Z_n$, the induced map from $\Pa_z$ to $\Pa_{g.z}$ preserves the $s$-coloring,
    i.e., induces the identity of $G^s$.
    Let
    \[ F := \{ \tilde g \mid g \in U_{c_0}|_{\B{c_0}n} \} \leq U_{c_0}|_E . \]
    Then indeed $U_{c_0}|_E = D \rtimes F$,
    and the conjugation action of $F$ on $D$ is given by permuting the entries of the direct product
    according to the action of $U_{c_0}|_{\B{c_0}n}$ on $C_n \subseteq \B{c_0}{n}$.
\end{proof}

\begin{prop}\label{B2}
Each element of $U_{c_0}|_{\B{c_0}n}$ extends uniquely to $\B{c_0}n\cup A_2(n+1)$.
In particular, $U_{c_0}|_{\B{c_0}n \cup A_2(n+1)}\cong U_{c_0}|_{\B{c_0}n}$.
\end{prop}
\begin{proof}
Let $c\in A_2(n+1)$ and let $w$ be a reduced representation of $\delta(c_0,c)$.
Since $|L(w)| \geq 2$, we can choose $s,t \in L(w)$ with $s \neq t$.
Consider the chambers $c_s = \proj_{\Pa_{s,c}}(c_0)$ and $c_t = \proj_{\Pa_{t,c}}(c_0)$ in $\B{c_0}n$.
Observe that $c$ is the unique chamber of $A_2(n+1)$ that is $s$-adjacent to $c_s$ and $t$-adjacent to $c_t$:
if $d$ were another such chamber, then $c$ and $d$ would be both $s$-adjacent and $t$-adjacent, which is impossible.
In particular, for any $g \in U_{c_0}$, the restriction of $g$ to $\B{c_0}{n} \cup A_2(n+1)$ is already determined by
the restriction of $g$ to $\B{c_0}{n}$, and the result follows.
\end{proof}

As the sets $A_1(n+1)$ and $A_2(n+1)$ are mutually disconnected,
we can combine the two previous results to obtain an extension of the action of $U_{c_0}|_{\B{c_0}n}$ to $\B{c_0}{n+1}$ and a description of the groups $U_{c_0}|_{\B{c_0}n}$.

\begin{thm}\label{thelimit}
For each $n$, we have
\begin{equation}\label{eq:thelimit}
    U_{c_0}|_{\B{c_0}{n+1}} \cong (\prod_{z\in Z_n} G_z) \rtimes U_{c_0}|_{\B{c_0}{n}},
\end{equation}
where the conjugation action of $U_{c_0}|_{\B{c_0}n}$ on $\prod_{z\in Z_n} G_z$ is given by permuting the entries of the direct product
according to the action of $U_{c_0}|_{\B{c_0}n}$ on $C_n \subseteq \B{c_0}{n}$.
Moreover,
\[ U_{c_0} \cong \varprojlim_n U_{c_0}|_{\B{c_0}{n}} . \]
\end{thm}
\begin{proof}
    This is now an immediate consequence of Propositions~\ref{B1} and~\ref{B2}.
\end{proof}

\begin{rem}
    The semidirect product occurring in Theorem~\ref{thelimit} is almost a (complete) wreath product;
    the only difference is that the groups $G_z$ might be distinct for each orbit of the action of $U_{c_0}|_{\B{c_0}n}$ on $Z_n$.
    On the other hand, see Remark~\ref{rem:Ugwp} below.
\end{rem}

\begin{rem}\label{rem:constr}
    Theorem~\ref{thelimit} gives a procedure to {\em construct} the group $U_{c_0}$ recursively, together with its action on
    the directed right-angled building $\Delta$.
    Indeed, given the group $U_n = U_{c_0}|_{\B{c_0}{n}}$ acting on $\B{c_0}{n}$, we can define a group $U_{n+1}$ as in~\eqref{eq:thelimit},
    and endow it with a faithful action on $\B{c_0}{n+1}$ precisely by extending the action of $\B{c_0}{n}$ following
    the description given in the proofs of Theorems~\ref{B1} and~\ref{B2}.
    In particular, the group $\prod_{z\in Z_n} G_z$ acts trivially on $\B{c_0}{n}$ and on $A_2(n+1)$, and acts naturally on $A_1(n+1)$
    via the isomorphisms $G_z \cong U_z$.
\end{rem}

\subsection{A description of the $w$-sphere stabilizers as generalized wreath products}\label{sec:compgenwreath}

It might be desirable to have a more direct description of the stabilizer of a $w$-sphere for a specific $w \in W$,
where a $w$-sphere is a subset of $\ch\Delta$ of the form
\[ \Sp{c_0}{w} = \{ v \in \ch\Delta \mid \delta (c_0,v) = w \} , \]
where $\delta$ is the Weyl distance.
It turns out that this is indeed possible, using the generalized wreath products described in Section~\ref{sec:genwreath}.
Notice that, by Example~\ref{Intransitive groups}, the group $U_{c_0}|_{\Sp{c_0}{n}}$ is a subdirect product of the groups
$U_{c_0}|_{\Sp{c_0}{w}}$ running over all $w \in W$ with $l(w) = n$;
this gives another point of view on the structure of the group $U_{c_0}$.

\begin{defn}\label{def:XwDw}
    Let $w = s_1 \dotsm s_{n-1} s_n$ be a reduced word in $M_S$ and write $w' = s_1 \dotsm s_{n-1}$.
    Consider the associated partially ordered set $(S_w, \prec_w)$ as in Definition~\ref{def:poset}.
    For each $i \in \{ 1,\dots,n \}$, let $X_i = \{2, \dots, q_{s_i}\}$.
    We define
    \[ X_w = \{ (\alpha_1, \dots, \alpha_n) \mid \alpha_i \in X_i \} . \]
    By the definition of a directed right-angled building, the chambers in the $w$\dash sphere $\Sp{c_0} w$ are indexed by the elements of $X_w$.
    Moreover, every element of $U_{c_0}|_{\Sp{c_0}{w}}$ induces an element of $U_{c_0}|_{\Sp{c_0}{w'}}$ by restricting the elements of $X_w$
    to their first $n-1$ coordinates.
    Let
    \[ D_w = \{ g \in U_{c_0}|_{\Sp{c_0}{w}} \mid g \text{ fixes } \Sp{c_0}{w'} \} . \]
    Notice that $D_w$ depends on the word $w$ in $M_S$: different reduced representations $w_1, w_2$ for a given element of $W$
    might have different associated groups $D_{w_1}, D_{w_2}$.
\end{defn}

\begin{prop}\label{swstructure}
Let $w = s_1 \dotsm s_{n-1} s_n$ be a reduced word in $M_S$.
\begin{enumerate}
    \item
        The group $U_{c_0}|_{\Sp{c_0}w}$ is permutationally isomorphic to the generalized wreath product $G = X_w\gwr_{i \in S_w} G^{s_i}_0$
        with respect to the partial order~$\prec_w$.
    \item
        Under this isomorphism, the subgroup $D_w \leq U_{c_0}|_{\Sp{c_0}w}$ corresponds to the subgroup $D(n) \leq G$ from Definition~\ref{def:ideal}.
    \item
        If we write $w' = s_1 \dotsm s_{n-1}$, then $U_{c_0}|_{\Sp{c_0}w} \cong D_w \rtimes U_{c_0}|_{\Sp{c_0}{w'}}$.
\end{enumerate}
\end{prop}
\begin{proof}
We will identify $\Sp{c_0}{w}$ and $X_w$, i.e., we will view both $U_{c_0}|_{\Sp{c_0}w}$ and~$G$ as subgroups of $\Sym{X_w}$,
and we will show by induction on $n = l(w)$ that they are equal.
Notice that if $n=1$, then these two groups coincide by definition;
hence we may assume that $n \geq 2$.

Let $w_1, w_2, \dots, w_n$ be the sequence of elements in $W$ represented by the words $s_1, s_1s_2, \dots, s_1s_2 \dotsm s_n$, respectively,
and let $K := \{ 1_W, w_1, w_2, \dots, w_n \}$.
As $K$ is a (finite) connected subset of $W$ containing the identity, we can apply Lemma~\ref{generators} and obtain a set of generators
for $U_{c_0}|_{\Sp{c_0}w}$, with an explicit description as in Propositions~\ref{prop:g0directed} and~\ref{prop:nicegenerators}.
Each of these generators (acting on~$X_w$) belongs to the generalized wreath product $G$,
hence $U_{c_0}|_{\Sp{c_0}w} \leq G$.

Now consider $S_{w'} = \{1, \dots, n-1 \}$ with the partial order $\prec_{w'}$, let
$X_{w'} = \{ (\alpha_1, \dots, \alpha_{n-1}) \mid \alpha_i \in X_i \}$
and let $G' = X_{w'} \gwr_{i\in S'_{w'}} G_0^{s_i}$.
By the induction hypothesis, we have $G' = U_{c_0}|_{\Sp{c_0}{w'}}$ as subgroups of $\Sym{X_{w'}}$.

Observe that $\{ n \}$ is an ideal of $S_w$;
Lemma~\ref{le:r-ideal} now implies that
\begin{equation}\label{eq:Gsemi}
    G=D(n)\rtimes H ,
\end{equation}
where $H$ is a subgroup of $G$ isomorphic to $G'$.
Comparing Definition~\ref{def:ideal} and Definition~\ref{def:XwDw}, we see that $D(n) = D_w$, and of course $D_w \leq U_{c_0}|_{\Sp{c_0}w}$.
On the other hand, the embedding of $G'$ into $G$ as in Lemma~\ref{le:r-ideal}.
corresponds precisely to the embedding of $U_{c_0}|_{\Sp{c_0}{w'}}$ into $U_{c_0}|_{\Sp{c_0}{w}}$ obtained by extending each element
trivially on the last coordinate of $X_w$, hence $H \leq U_{c_0}|_{\Sp{c_0}{w}}$.
We conclude that $G \leq U_{c_0}|_{\Sp{c_0}{w}}$, and hence we have equality.
The last statement now follows from~\eqref{eq:Gsemi}.
\end{proof}

\begin{rem}\label{rem:intwr}
The generalized wreath product in Proposition~\ref{swstructure} is closely related to intersections of standard (i.e., complete) wreath products.
The partial order used to construct $G = X_w\gwr_{i \in S_w} G^{s_i}_0$ uses only the structure of the Coxeter diagram $\Sigma$ and the distinct representations of a reduced word $w$.
For each reduced word $w=s_1\dotsm s_n$ in the free monoid $M_S$ with respect to $\Sigma$, let
\[G_w=G_0^{s_n}\wr \cdots \wr G_0^{s_1}\]
denote the iterated wreath product acting on the set $X_w$ of tuples $(\alpha_1, \dots, \alpha_n)$,
where each $\alpha_i$ is an element in $\{2, \dots, q_{s_i}\}$, for $i\in \{1, \ldots, n\}$, with its imprimitive action (see \cite[Section 2.6]{DixonMortimer}).
We can view the iterated wreath product as a subgroup of $\Sym{X_w}$.

For each $\sigma\in \Rep w$, we have
\[G_{\sigma.w}=G_0^{s_{\sigma(n)}}\wr \cdots \wr G_0^{s_{\sigma(1)}};\]
we can view this as a group acting on the same set $X_w$ (but of course taking into account that the entries have been permuted by $\sigma$),
so in particular, it makes sense to consider the intersection
\[ H = \bigcap_{\sigma\in \Rep w} G_{\sigma.w} \]
as a subgroup of $\Sym{X_w}$.
The generalized wreath product $G$ is then permutationally isomorphic to $H$.
\end{rem}

\begin{rem}
    When $\Delta$ is a tree, then for each reduced word $w$, the partial order considered in the set $S_w$ is a chain.
    Therefore, by Remark~\ref{rem:gwp}, the generalized wreath products considered in this construction are iterated wreath products in their imprimitive action.
    Hence Proposition~\ref{swstructure} translates, for the case of trees, to the description provided by Burger and Mozes in \cite[Section 3.2]{BurgerMozes}.
    Notice that the chamber $c_0$ corresponds to an edge of the tree, and for each of the two reduced words $w$ of length $n$, the corresponding
    $w$-sphere stabilizer corresponds to the stabilizer of a ball in the tree around one of the two vertices of that edge.
\end{rem}

We finish this section by computing the order of the groups $U_{c_0}|_{\Sp{c_0}w}$ and $U_{c_0}|_{\Sp{c_0}n}$.
\begin{defn}\label{def:dw}
    Let $w = s_1 \dotsm s_n$ be a reduced word in $M_S$.
    Then we define
    \[ d_w = \prod_{j \succ_{\!w} n}(q_{s_j}-1) \]
    (where $d_w = 1$ if there is no $j \succ_{\!w} n$).
    Notice that if $w \in W_1(n)$, then every $j \in \{ 1,\dots,n-1 \}$ satisfies the condition $j \succ_{\!w} n$,
    so in that case, $d_w = \prod_{j=1}^{n-1} (q_{s_j}-1)$.
\end{defn}
\begin{prop}
    Let $w=s_1\cdots s_n$ be a reduced word.
    Then
    \[ \big\lvert U_{c_0}|_{\Sp{c_0}w} \big\rvert = \prod_{i=1}^{n} |G^{s_i}_0|^{d_i}, \]
    where $d_i = d_{s_1 \dotsm s_i} = \prod_{j \succ_{\!w} i}\,(q_{s_j}-1)$, for each $i\in \{1, \dots, n\}$.
\end{prop}
\begin{proof}
    This follows by induction on $n$ from Proposition~\ref{swstructure} and Proposition~\ref{DIorder}.
    Notice that for each initial subword $w_i = s_1 \dotsm s_i$, the poset $\prec_{w_i}$ is a sub-poset of $\prec_w$,
    so that we can indeed express all $d_i$ using $\prec_w$ only.
\end{proof}

\begin{prop}
For each generator $s\in S$, let
\[ d(s,n) = \sum_{\substack{w\in W_1(n) \\ \textrm{s.t.\@~} r_w=s}} d_w \ \ \text{ and } \ \ t(s,n)=\sum_{i=1}^n d(s,i). \]
Then
\[ |U_{c_0}|_{\B{c_0}{n}}| = \prod_{s\in S} |G_0^s|^{t(s,n)} . \]
\end{prop}
\begin{proof}
    Recall from equation~\eqref{eq:Zn} on p.\@~\pageref{eq:Zn} that $Z_{n-1} = \{ (c,s) \in A(n-1) \times S \mid \delta(c_0,c)s \in W_1(n) \}$, and observe that this set
    can be partitioned as
    \[ Z_{n-1} = \bigsqcup_{w \in W_1(n)}\bigl( \Sp{c_0}{wr_w} \times \{ r_w \} \bigr) . \]
    (Notice that $wr_w$ is the unique initial subword of $w$ of length $n-1$.)
    By the remark in Definition~\ref{def:dw}, the sphere $\Sp{c_0}{wr_w}$
    contains precisely $d_w$ chambers.
    It follows that
    \[ \prod_{z \in Z_{n-1}} G_z = \prod_{s \in S} \ \prod_{\substack{w\in W_1(n) \\ \textrm{s.t.\@~} r_w=s}} G^s_0 = \prod_{s \in S} (G^s_0)^{d(s,n)} . \]
    The result now follows by induction on $n$ from Theorem~\ref{thelimit}.
\end{proof}

\begin{rem}\label{rem:Ugwp}
Similarly as in Proposition~\ref{swstructure}, one can show that also the groups $U_{c_0}|_{\B{c_0}n}$ are generalized wreath products.

We first describe the underlying finite posets of these generalized wreath products.
Consider the Coxeter group $W$ as a thin right-angled building and let $B_n$ be the ball of radius $n$ around the identity element $1_W$.

Let $S_n$ be the set of all tree-walls in $W$ crossing the ball $B_n$.
For each pair $\T_1$ and $\T_2$ of such tree-walls in $S_n$, we set $\T_1 \prec_n \T_2$ if and only if
the tree-wall $\T_2$ is completely contained in the wing of $\T_1$ containing the identity element $1_W$
{\em and} the tree-wall $\T_1$ is contained in one of the wings of $\T_2$ {\em not} containing the identity element $1_W$.
This makes $(S_n, \prec_n)$ into a partially ordered set.

For each $s$-tree-wall $\T \in S_n$, we set $s_\T = s$ and $X_\T = \{2, \dots, q_s \}$.
Let $X_n$ be the Cartesian product of all the sets $X_\T$ for every $\T$ in $S_n$.

Then the group $U_{c_0}|_{\B{c_0}n}$ is permutationally isomorphic to the generalized wreath product $X_n\gwr_{\T \in S_n} G^{s_\T}_0$
with respect to the partial order $\prec_n$.

We refer to the second author's PhD thesis \cite[Section 5.3.3]{PhDAna} for more details.
\end{rem}


\addcontentsline{toc}{section}{References}
\small
\bibliographystyle{amsalpha}
\bibliography{universal-arXiv}

\def\cprime{$'$}
\providecommand{\bysame}{\leavevmode\hbox to3em{\hrulefill}\thinspace}
\providecommand{\MR}{\relax\ifhmode\unskip\space\fi MR }
\providecommand{\MRhref}[2]{%
  \href{http://www.ams.org/mathscinet-getitem?mr=#1}{#2}
}
\providecommand{\href}[2]{#2}
\begin{thebibliography}{Woe91}

\bibitem[AB08]{AbramenkoBrown}
Peter Abramenko and Kenneth~S. Brown, \emph{Buildings}, Graduate Texts in
  Mathematics, vol. 248, Springer, New York, 2008, Theory and applications.
  \MR{2439729 (2009g:20055)}

\bibitem[AT08]{topologicalgroupsandrelatedstrcutures}
Alexander Arhangel{\cprime}skii and Mikhail Tkachenko, \emph{Topological groups
  and related structures}, Atlantis Studies in Mathematics, vol.~1, Atlantis
  Press, Paris; World Scientific Publishing Co. Pte. Ltd., Hackensack, NJ,
  2008. \MR{2433295 (2010i:22001)}

\bibitem[Beh90]{Behrendt}
Gerhard Behrendt, \emph{Equivalence systems and generalized wreath products},
  Acta Sci. Math. (Szeged) \textbf{54} (1990), no.~3-4, 257--268. \MR{1096805
  (92b:20032)}

\bibitem[BM00]{BurgerMozes}
Marc Burger and Shahar Mozes, \emph{Groups acting on trees: from local to
  global structure}, Inst. Hautes \'Etudes Sci. Publ. Math. (2000), no.~92,
  113--150 (2001). \MR{1839488 (2002i:20041)}

\bibitem[Bou97]{Bourdon}
Marc Bourdon, \emph{Immeubles hyperboliques, dimension conforme et rigidit\'e
  de {M}ostow}, Geom. Funct. Anal. \textbf{7} (1997), no.~2, 245--268.
  \MR{1445387 (98c:20056)}

\bibitem[Bro98]{Brown}
Kenneth~S. Brown, \emph{Buildings}, Springer Monographs in Mathematics,
  Springer-Verlag, New York, 1998, Reprint of the 1989 original. \MR{1644630
  (99d:20042)}

\bibitem[Cap14]{Caprace}
Pierre-Emmanuel Caprace, \emph{Automorphism groups of right-angled buildings:
  simplicity and local splittings}, Fund. Math. \textbf{224} (2014), no.~1,
  17--51. \MR{3164745}

\bibitem[CM11]{CM11}
Pierre-Emmanuel Caprace and Nicolas Monod, \emph{Decomposing locally compact
  groups into simple pieces}, Math. Proc. Cambridge Philos. Soc. \textbf{150}
  (2011), no.~1, 97--128. \MR{2739075 (2012d:22005)}

\bibitem[Dav98]{DavisCat0}
Michael~W. Davis, \emph{Buildings are {${\mathrm CAT}(0)$}}, Geometry and
  cohomology in group theory ({D}urham, 1994), London Math. Soc. Lecture Note
  Ser., vol. 252, Cambridge Univ. Press, Cambridge, 1998, pp.~108--123.
  \MR{1709955 (2000i:20068)}

\bibitem[DM96]{DixonMortimer}
John~D. Dixon and Brian Mortimer, \emph{Permutation groups}, Graduate Texts in
  Mathematics, vol. 163, Springer-Verlag, New York, 1996. \MR{1409812
  (98m:20003)}

\bibitem[Gib]{Gibbins}
Aliska Gibbins, \emph{Automorphism groups of graph products of buildings},
  \url{http://arxiv.org/abs/1407.4712v1}.

\bibitem[HP03]{HP2003}
Fr{\'e}d{\'e}ric Haglund and Fr{\'e}d{\'e}ric Paulin, \emph{Constructions
  arborescentes d'immeubles}, Math. Ann. \textbf{325} (2003), no.~1, 137--164.
  \MR{1957268 (2004h:51014)}

\bibitem[Laz14]{Laz}
Nir Lazarovich, \emph{On regular {CAT}(0) cube complexes},
  \url{http://arxiv.org/abs/1411.0178v1}, to appear in {Comment. Math. Helv.},
  2014.

\bibitem[Nic04]{Bogdan}
Bogdan Nica, \emph{Cubulating spaces with walls}, Algebr. Geom. Topol.
  \textbf{4} (2004), 297--309. \MR{2059193}

\bibitem[Sil17]{PhDAna}
Ana~C. Silva, \emph{Universal groups for right-angled buildings}, PhD Thesis,
  Ghent University, 2017.

\bibitem[Tit70]{Tits}
Jacques Tits, \emph{Sur le groupe des automorphismes d'un arbre}, Essays on
  topology and related topics ({M}\'emoires d\'edi\'es \`a {G}eorges de
  {R}ham), Springer, New York, 1970, pp.~188--211. \MR{0299534 (45 \#8582)}

\bibitem[Tit86]{Tits-amal}
\bysame, \emph{Buildings and group amalgamations}, Proceedings of
  groups---{S}t.\ {A}ndrews 1985, London Math. Soc. Lecture Note Ser., vol.
  121, Cambridge Univ. Press, Cambridge, 1986, pp.~110--127. \MR{896503}

\bibitem[Tit92]{Tits-Durham}
\bysame, \emph{Twin buildings and groups of {K}ac-{M}oody type}, Groups,
  combinatorics \& geometry ({D}urham, 1990), London Math. Soc. Lecture Note
  Ser., vol. 165, Cambridge Univ. Press, Cambridge, 1992, pp.~249--286.
  \MR{1200265}

\bibitem[TW11]{TW11}
Anne Thomas and Kevin Wortman, \emph{Infinite generation of non-cocompact
  lattices on right-angled buildings}, Algebr. Geom. Topol. \textbf{11} (2011),
  no.~2, 929--938. \MR{2782548 (2012d:20068)}

\bibitem[Wei09]{Weiss}
Richard~M. Weiss, \emph{The structure of affine buildings}, Annals of
  Mathematics Studies, vol. 168, Princeton University Press, Princeton, NJ,
  2009. \MR{2468338 (2009m:51022)}

\bibitem[Wil98]{Wilson}
John~S. Wilson, \emph{Profinite groups}, London Mathematical Society
  Monographs. New Series, vol.~19, The Clarendon Press, Oxford University
  Press, New York, 1998. \MR{1691054 (2000j:20048)}

\bibitem[Woe91]{Woess}
Wolfgang Woess, \emph{Topological groups and infinite graphs}, Discrete Math.
  \textbf{95} (1991), no.~1-3, 373--384.

\end{thebibliography}

\end{document}